\documentclass[12pt,final]{amsart}
\usepackage{amssymb, amsmath, amsthm}
\usepackage{mathrsfs}  

\usepackage[text={6in,8in},centering]{geometry}

\usepackage{ifdraft}
\ifoptionfinal{
\usepackage[disable]{todonotes}
}{
\usepackage[norefs, nocites]{refcheck}

\usepackage[notref, notcite]{showkeys}
\usepackage[bordercolor=white, color=white]{todonotes}
}
\newcommand{\HOX}[1]{\todo[noline, size=\footnotesize]{#1}}

\makeatletter\providecommand\@dotsep{5}\def\listtodoname{List of Todos}\def\listoftodos{\hypersetup{linkcolor=black}\@starttoc{tdo}\listtodoname\hypersetup{linkcolor=blue}}\makeatother

\usepackage{graphicx}
\usepackage[colorlinks=true, citecolor=green, linkcolor=red, urlcolor=red]{hyperref}
 \usepackage{xcolor}

\usepackage{hyperref}
\usepackage{etoolbox}

\newtheorem{theorem}{Theorem}[section]
\newtheorem{lemma}[theorem]{Lemma}

\newtheorem{proposition}[theorem]{Proposition}
\newtheorem{definition}[theorem]{Definition}

\theoremstyle{remark}
\newtheorem{remark}{Remark}

\numberwithin{equation}{section}

\newcommand{\mltext}{}

\newcommand{\bel}{\begin{equation} \label}
\newcommand{\ee}{\end{equation}}
\def\beq{\begin{equation}}
\def\eeq{\end{equation}}
\newcommand{\bea}{\begin{eqnarray}}
\newcommand{\eea}{\end{eqnarray}}
\newcommand{\beas}{\begin{eqnarray*}}
\newcommand{\eeas}{\end{eqnarray*}}

\def\C{\mathbb C}

\def\R{\mathbb R}

\def\N{\mathbb N}

\def\g{G}

\def\CI{\mathcal C}
\renewcommand{\leq}{\leqslant}
\renewcommand{\geq}{\geqslant}
\DeclareMathOperator{\linspan}{span}
\def\p{\partial}
\DeclareMathOperator{\Tr}{Tr}
\DeclareMathOperator{\supp}{supp}

\def\i{\text{in}}
\def\o{\text{out}}
\def\crit{\text{crit}}

\DeclareMathOperator{\rel}{R}
\DeclareMathOperator{\CP}{CP}

\numberwithin{equation}{section}

\title
[Inverse problems in Lorentzian geometries]{INVERSE PROBLEMS FOR NON-LINEAR HYPERBOLIC EQUATIONS WITH DISJOINT
SOURCES AND RECEIVERS}
\author[A. Feizmohammadi]{Ali Feizmohammadi}
\address{Department of Mathematics, University College London, 
Gower Street, London UK, WC1E 6BT.}
\email{a.feizmohammadi@ucl.ac.uk}
\author[M. Lassas]{Matti Lassas}
\address{Department of Mathematics and Statistics, University of Helsinki,
P.O. Box 68,
FI-00014
 Helsinki, Finland.}
\email{Matti.Lassas@helsinki.fi}
\author[L. Oksanen]{Lauri Oksanen}
\address{Department of Mathematics, University College London, 
Gower Street, London UK, WC1E 6BT.}
\email{l.oksanen@ucl.ac.uk}

\begin{document}
\begin{abstract}
 {\mltext The paper studies inverse problems of determining unknown coefficients in various semi-linear and quasi-linear wave equations.} We introduce a method to solve  inverse problems for non-linear equations using interaction of three waves, that makes it possible to study the inverse problem in all dimensions $n+1\geq 3$. We consider the case when the set $\Omega_{\textrm{in}}$, where the sources are supported, and the set $\Omega_{\textrm{out}}$, where the observations are made, are separated. As model problems we study both a quasi-linear and also a semi-linear wave equation and show in each case that it is possible to uniquely recover the {\mltext background metric up to the natural obstructions for uniqueness that is governed by finite speed of propagation for the wave equation and a  gauge} corresponding to change of coordinates. The proof consists of two independent components. In the first half we study multiple-fold linearization of the non-linear wave equation near real parts of Gaussian beams that results in a three-wave interaction. We show that the three-wave interaction can produce a three-to-one scattering data. In the second half of the paper, we study an abstract formulation of the three-to-one scattering relation showing that it recovers the topological, differential and conformal {\mltext structures of the manifold in a causal diamond set  that
 is the intersection of the future of the point
 $p_{in}\in \Omega_{\textrm{in}}$ and the past of the point $p_{out}\in \Omega_{\textrm{out}}$. 
The results do not require any assumptions on the conjugate or cut points.}
\end{abstract}
\maketitle

 

\section{Introduction}

Let $(M,g)$ be a smooth Lorentzian manifold of dimension $1+n$ with $n\geq 2$ and signature $(-,+,\ldots,+)$. We write $\le$ and $\ll$ for the causal and chronological relations on $(M,g)$, and define the causal future past and future of a point $p\in M$ through
$$J^+(p) = \{x \in M: p \le x\}\quad \text{and} \quad J^-(p) = \{x \in M : x \le p\}.$$
The chronological future and past of $p$ is defined analogously with the causal relation replaced by the chronological relation,
$$I^+(p)=\{x \in M: p \ll x\}\quad \text{and}\quad I^-(p) = \{x \in M : x \ll p\}.$$ 

We will make the standing assumption that $(M,g)$ is {\em globally hyperbolic}. Here, by global hyperbolicity we mean that $(M,g)$ is causal (i.e. no closed causal curve exists) and additionally if $p,q \in M$ with $p\le q$, then $J^+(p)\cap J^-(q)$ is compact \cite{BS}.
Global hyperbolicity implies that the relation $\le$ is closed while $\ll$ is open and consequently that $J^{\pm}$ is closed while $I^{\pm}$ is open. It also implies that there exists a global splitting in ``time'' and ``space'' in the sense that
$(M,g)$ is isometric to $\R\times M_0$ with the metric taking the form 
    \begin{align}\label{splitting}
g=c(x^0,x')\left(-dx^0\otimes\,dx^0+g_0(x^0,x')\right), \quad \forall x^0 \in \R,\ x' \in M_0,
    \end{align}
where $c$ is a smooth positive function and $g_0$ is a Riemannian metric on the $n$-dimensional manifold $M_0$ smoothly depending on the parameter $x^0$. Moreover, each set $\{x^0\}\times M_0$ is a Cauchy hypersurface in $M$, that is to say, any inextendible causal curve intersects it exactly once. For the sake of brevity, we will sometimes identify points, functions and tensors over the manifold $(M,g)$ with their preimage in $\R\times M_0$ without explicitly writing the diffeomorphism $\Phi$.

In this paper, we consider {\mltext the inverse problems with partial data for semi-linear and quasi-linear wave equations, where the set $\Omega_{\i}$, where the sources are supported, and the set ${\Omega_\o}$, where the observations are made, may be separated. Motivated by applications, such problems can 
be called the {\em remote sensing} problems.} The study of the semi-linear model is carried out throughout the paper as a simpler analytical model that clarifies the main methodology. A quasi-linear model is also considered to show the robustness of the method to various kinds of non-linearities. 

The main novelties of the paper are that we {\mltext develop a framework for inverse problems for non-linear equations, where one uses interaction of only three waves. To this end, we formulate
the concept of {\em three-to-one scattering relation} that is applicable for a wide class of non-linear equation (see Theorem \ref{t2}).
This approach makes it possible to study the inverse problem in all dimensions $n+1\geq 3$ and 
the partial data problems with separated sources and observations. }

\begin{figure}
$  $\hspace{-10cm}
\includegraphics[height=4.0cm]{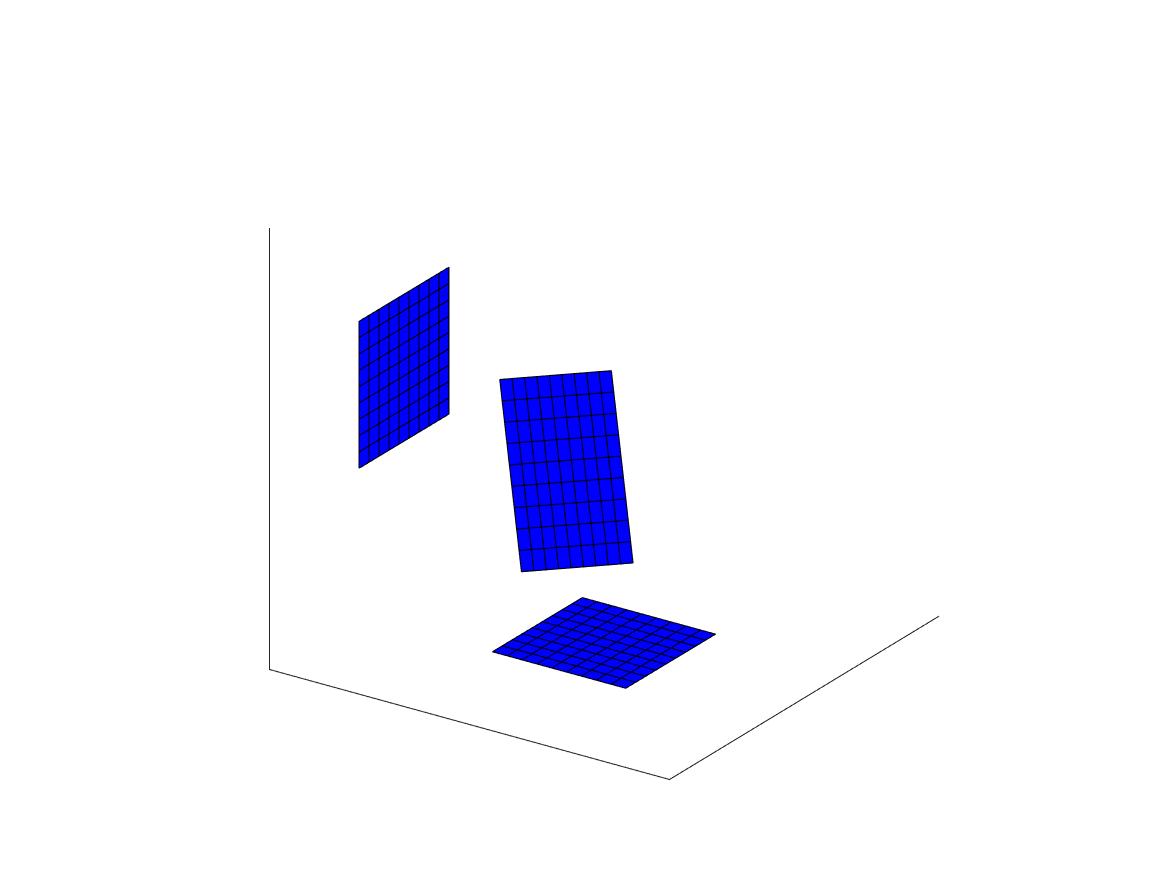}\hspace{-1.8cm}
\includegraphics[height=4.0cm]{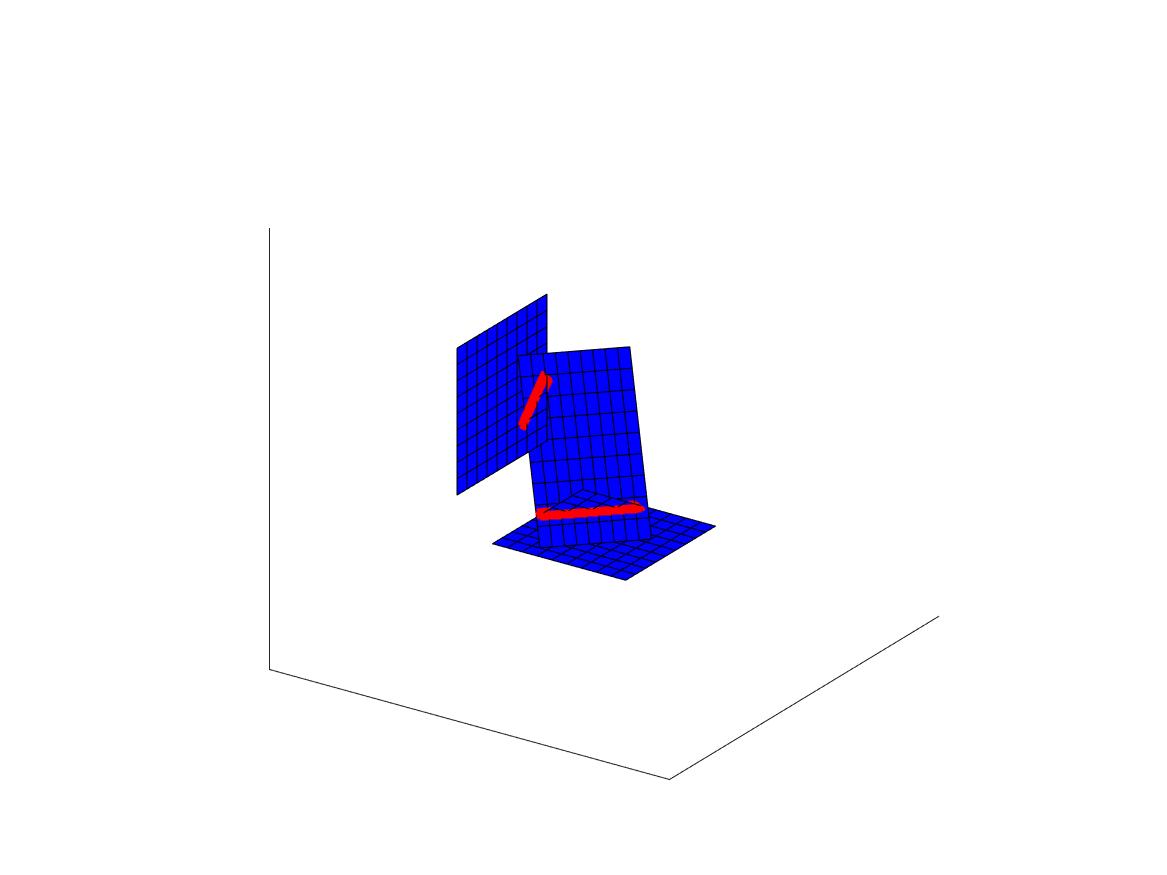}\hspace{-1.8cm}
\includegraphics[height=4.0cm]{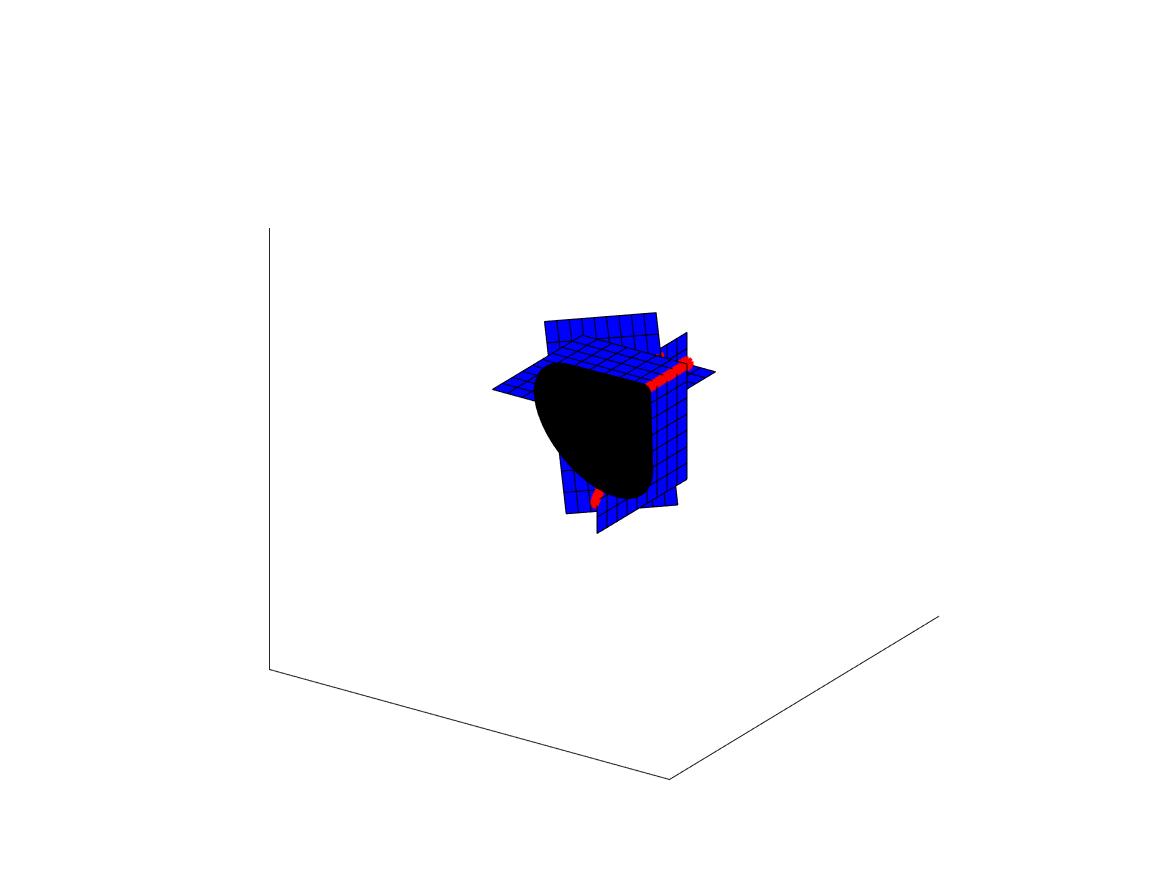}\hspace{-1.8cm}
\includegraphics[height=4.0cm]{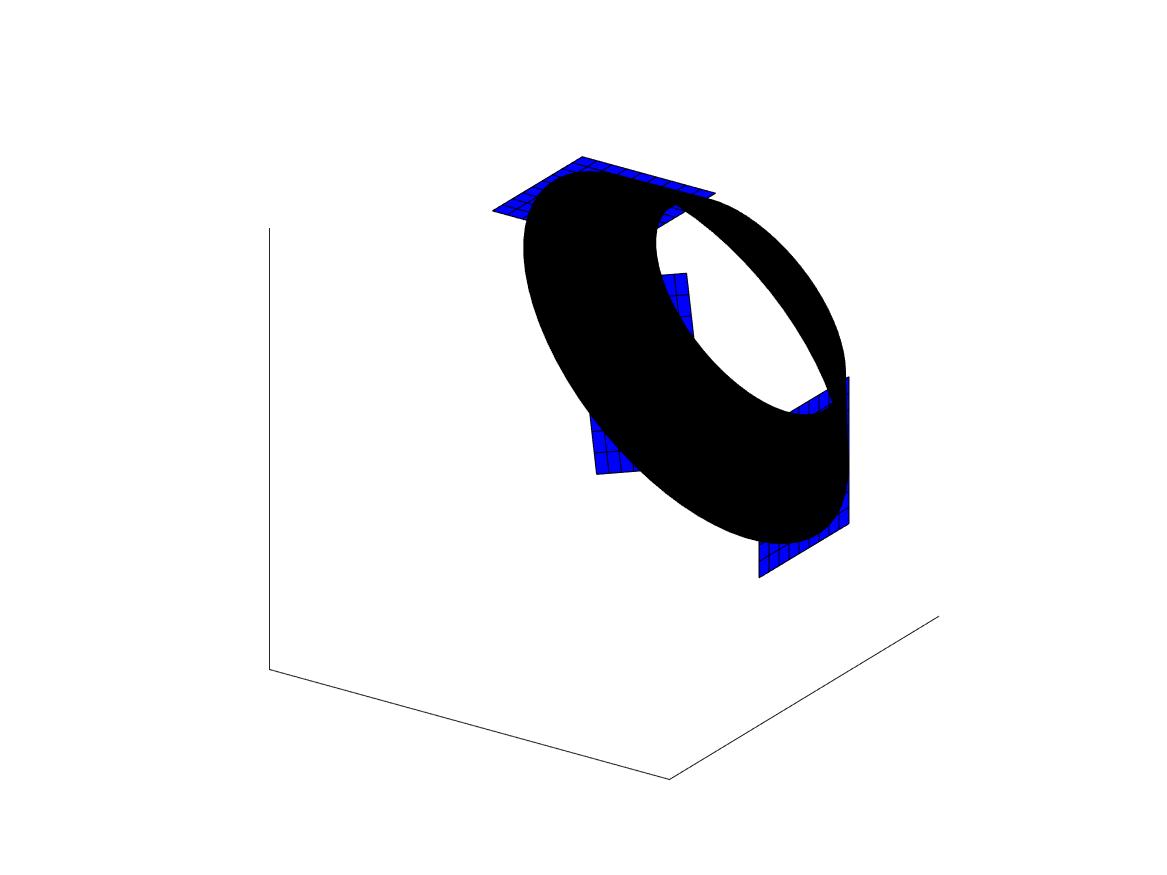}\hspace{-10cm}$  $

\caption{Non-linear interaction of waves in the case $n=m=3$.
{\mltext Three plane waves propagate in space. 
When the planes intersect, the non-linearity of the hyperbolic system  produces
new waves.
The four figures are snapshots of the waves in the space $\R^3$  at different times $t_1,t_2,t_3,t_4$ that show the waves before the interaction of the waves start, when 2-wave interactions have  started, when all  3 waves are interacting, and later
after the interaction. {\bf Left:} Plane waves before  interacting.
 {\bf Middle left:}
 The 2-wave interactions (red line segments) appear  but do not cause
singularities that propagate in new directions. {\bf Middle right:}  All plane waves have intersected and new 
 waves have appeared. The 3-wave interactions causes 
a new conic waves (black surface). {\bf Right:} After the 3-wave interaction, the waves propagate in space to different directions. 
By varying the directions $v_1,v_2,v_3$ of the incoming plane waves (even when all
$v_j$ are in a small neighborhood of a given vector), the wave fronts produced by the 3-wave interactions
can be sent to all directions. Note that on a general Lorentzian manifolds the wave fronts may develop caustics that makes situation more complicated.}}\label{fig_extra}
\end{figure}


%

\subsection{The semi-linear model}

Our main aim is to study quasi-linear equations but to describe how the method works, we start our considerations with semi-linear equations. We consider the model setup
\bel{pf0}
\begin{aligned}
\begin{cases}
\Box_{g}u+u^m=f\,\quad \text{on $(-\infty,T)\times M_0$},
\\
\text{$u=0$ \quad on $(-\infty,-T)\times M_0$.}
\end{cases}
    \end{aligned}
\ee
Here, $m\geq 3$ is an integer and $T>0$ is a parameter to be fixed later and the source $f$ is real-valued and compactly supported in the set $(-T,T)\times M_0$. The wave operator, $\Box_g$, is defined in local coordinates $(x^0,\ldots,x^n)$ by the expression
$$\Box_{g} u = -\sum_{j,k=0}^n\left|\det g\right|^{-\frac{1}{2}} \frac{\p}{\p x^j}( \left|\det g\right|^{\frac{1}{2}}g^{jk} \frac{\p u}{\p x^k}),$$
where $g^{jk}$ stands for the elements of the inverse of $g$. Note that we are using the $(x^0,x')$-coordinate system on $M$ that is given by \eqref{splitting}.

\subsection{The quasi-linear model}
For the quasi-linear wave equation, we first consider a family of smooth real-valued symmetric tensors $G_z(x)=G(x,z)$ with $x \in M$ and $z \in \R$, satisfying
\begin{itemize}
\item[(i)]{$G(x,0)=g(x)$ and $\p_z G(x,0)=0$ for all $x \in M$.} 
\item[(ii)]{The tensor $h(x)=\frac{1}{2}\p^2_z G(x,0)$ satisfies $\langle v,v\rangle_h\neq 0$ for all non-zero $v\in LM$.}  
\end{itemize} 
Here, $LM$ denotes the bundle of light-like vectors on $M$ with respect to the metric $g$. We subsequently consider the equation
\bel{pf}
\begin{aligned}
\begin{cases}
\Box_{G_u}u=f\,\quad \text{on $(-\infty,T)\times M_0$},
\\
\text{$u=0$ \quad on $(-\infty,-T)\times M_0$.}
\end{cases}
    \end{aligned}
\ee
Here, the quasi-linear wave operator $\Box_{G_u}$ is defined in local coordinates $(x^0,\ldots,x^n)$ through:
$$ \Box_{G_u}u = -\sum_{j,k=0}^n\left|\det G(x,{\mltext u(x)})\right|^{-\frac{1}{2}} \frac{\p}{\p x^j}
\bigg( \left|\det G(x,{\mltext u(x)})\right|^{\frac{1}{2}}G^{jk}(x,{\mltext u(x)}) \frac{\p u}{\p x^k}{\mltext (x)}\bigg),$$
where $G^{jk}$ stands for elements of the inverse of $G$. We will assume that the source $f$ in \eqref{pf} is real-valued and compactly supported in $(-T,T)\times M_0$.

In Section \ref{forward_section} we show that each of the Cauchy problems \eqref{pf0}--\eqref{pf} above, admits a unique solution 
$$ u \in \mathcal C^2((-\infty,T)\times M_0),\quad \forall\, f \in \mathscr C_O,$$ 
where given any open and bounded set $O \subset (-T,T)\times M_0$, we define
\bel{C_def}\mathscr C_{O}=\{h \in H^{n+1}(\R\times M_0;\R)\,:\,\supp h \subset O,\quad \|h\|_{H^{n+1}(\R\times M_0)}\leq r_O\}\ee and $r_O$ is a sufficiently small constant depending on $(M,g)$, $O$ and $T$.

\subsection{Source-to-solution map and the remote sensing inverse problem}
Our primary interest lies in the setting that the sources can be actively placed near a world line $\mu_{\i}$ 
and the corresponding unique small solution $u$ will be be measured near another disjoint world line $\mu_{\o}$ corresponding to some observer. The main question is whether such experiments corresponding to the separated source and observation regions determine the structure of the background unperturbed media, i.e. $(M,g)$.  

To state the inverse problem precisely, let us consider two disjoint time-like future pointing smooth paths 
$$\mu_{\i}:[t_0^-,t_0^+]\to M \quad \text{and}\quad \mu_{\o}:[s_0^-,s_0^+]\to M$$
and impose the conditions that
\bel{worldlines}
\mu_{\o}(s_0^+)\notin I^+(\mu_\i(t_0^+))\quad \text{and}\quad \mu_{\i}(t_0^-) \notin J^-(\mu_{\o}(s_0^-)). 
\ee

Next, let us consider the source and observation regions $\Omega_{\i}$ and $\Omega_{\o}$ as small neighborhoods of $\mu_{\i}([t_0^-,t_0^+])$ and $\mu_{\o}([s_0^-,s_0^+])$ in $M$ respectively. These two open neighborhoods will be precisely defined in Section~\ref{main_result}. We will also make the standing assumption that 
$$ (\Omega_{\i}, g_{|_{\Omega_{\i}}}) \quad \text{and}\quad (\Omega_{\o},g_{|_{\Omega_{\o}}})$$
are a priori known as Lorentzian manifolds, that is to say, we are given local coordinates, the transition functions between the local charts, and the metric tensors on these coordinate charts.

 {\mltext The partial data inverse problem with separated sources and observations
 (or, the remote sensing problem)} can now be formulated as follows. Is it possible to uniquely determine the unperturbed manifold $(M,g)$ (recall that $G(x,0)=g(x)$) by observing solutions to the non-linear wave equations \eqref{pf0} or \eqref{pf} in $\Omega_{\o}$ that arise from small sources placed in $\Omega_{\i}$?  
{\mltext The inverse  problems with partial data are widely encountered in applications. 
The partial data problems where the sources and observations are made only
on a part of boundary  have been a focus of research for inverse problems for linear elliptic equations \cite{DKSU,DKLS,GT,IUY,IUY2,Isakov,KS,KSU,KUII,LLLS,LU}. However, in most of these
results it is assumed that the sets where the sources are supported and where the
solutions are observed do intersect, with the notable exceptations in \cite{IUY2,Isakov,KS}. The partial data problems with separated sources 
and observations 
have been studied for linear hyperbolic equations, but
the present results require  convexity or other geometrical restrictions that guarantee 
the exact controllability of the system \cite{KKLO,LO}.
 Let us also remark that we can apply the results in this paper also in  the case when $\Omega_\i$ and $\Omega_\o$ intersect.}

%

To formulate the inverse problem precisely, we define the source-to-solution map $\mathscr L$ associated to the semi-linear Cauchy equation \eqref{pf0} through the expression
\bel{sotsol0}
\mathscr L f= u\,|_{\Omega_{\o}}, \quad \forall\, f \in \mathscr{C}_{\Omega_\i},
\ee
where $u$ is the unique small solution to \eqref{pf0} subject to the source $f$ and the set $\mathscr C_{\Omega_\i}$ is as defined by \eqref{C_def}. Analogously, we define the source-to-solution map $\mathscr N$ for the quasi-linear Cauchy equation \eqref{pf} through the expression
\bel{sotsol}
\mathscr N f= u\,|_{\Omega_{\o}}, \quad \forall \,f\in  \mathscr{C}_{\Omega_\i},
\ee
where $u$ is the unique small solution to \eqref{pf} subject to the source $f$.

Our inverse problem can now be re-stated as whether the manifold $(M,g)$ can be uniquely recovered given the source-to-solution map $\mathscr L$ or $\mathscr N$. Recall that $g(x)=G(x,0)$ in the quasi-linear model. 

Due to finite speed of propagation for the wave equation, the optimal region where one can recover the geometry is the {\em causal diamond} generated by the source region $\Omega_{\i}$ and $\Omega_{\o}$ that is defined through 
\bel{optimal_reg}\mathbb D_e=\left(\bigcup_{q \in \Omega_{\i}} I^{+}(q)\right) \cap \left(\bigcup_{q \in \Omega_{\o}} {I}^{-}(q)\right),\ee
given the knowledge of the source-to-solution map $\mathscr L$ or $\mathscr N$. As we will see, we are able to recover the geometry in the slightly smaller set, {\mltext that is a {\em causal diamond} determined 
by the points $\mu_\i(t_0^-)$  and $\mu_\o(s_0^+)$,}
\bel{recovery_region} \mathbb D= I^{+}(\mu_\i(t_0^-)) \cap {I}^{-}(\mu_\o(s_0^+)).\ee

\begin{figure}
\def\svgwidth{9cm}
\begingroup%
  \makeatletter%
  \providecommand\color[2][]{%
    \errmessage{(Inkscape) Color is used for the text in Inkscape, but the package 'color.sty' is not loaded}%
    \renewcommand\color[2][]{}%
  }%
  \providecommand\transparent[1]{%
    \errmessage{(Inkscape) Transparency is used (non-zero) for the text in Inkscape, but the package 'transparent.sty' is not loaded}%
    \renewcommand\transparent[1]{}%
  }%
  \providecommand\rotatebox[2]{#2}%
  \newcommand*\fsize{\dimexpr\f@size pt\relax}%
  \newcommand*\lineheight[1]{\fontsize{\fsize}{#1\fsize}\selectfont}%
  \ifx\svgwidth\undefined%
    \setlength{\unitlength}{471.70149477bp}%
    \ifx\svgscale\undefined%
      \relax%
    \else%
      \setlength{\unitlength}{\unitlength * \real{\svgscale}}%
    \fi%
  \else%
    \setlength{\unitlength}{\svgwidth}%
  \fi%
  \global\let\svgwidth\undefined%
  \global\let\svgscale\undefined%
  \makeatother%
  \begin{picture}(1,0.91538355)%
    \lineheight{1}%
    \setlength\tabcolsep{0pt}%
    \put(0,0){\includegraphics[width=\unitlength,page=1]{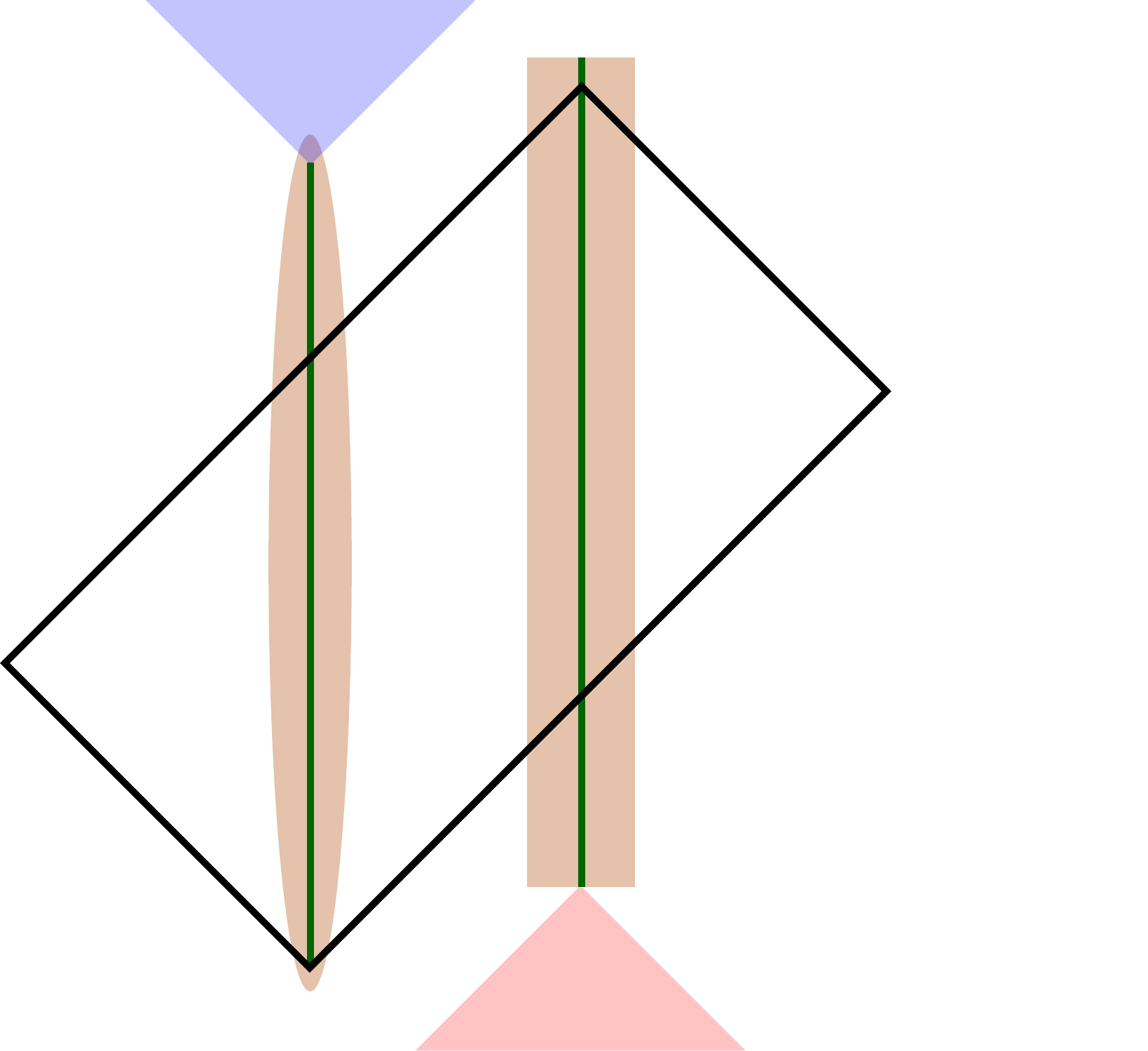}}%
    \put(0.5801126,0.55958442){\color[rgb]{0,0,0}\makebox(0,0)[lt]{\lineheight{1.25}\smash{\begin{tabular}[t]{l}$\Omega_\o$\end{tabular}}}}%
    \put(0.13185862,0.36907252){\color[rgb]{0,0,0}\makebox(0,0)[lt]{\lineheight{1.25}\smash{\begin{tabular}[t]{l}$\Omega_\i$\end{tabular}}}}%
  \end{picture}%
\endgroup%

\caption{
Schematic of the geometric setting.
The time-like paths $\mu_\i$ and $\mu_\o$
in green, and their neighborhoods $\Omega_\i$ and $\Omega_\o$ in orange. The set $\mathbb D$ is enclosed by the black rectangle. The set $I^+({\mltext \mu_\i(t_1^+)})$ in light blue,
and the set $J^-({\mltext \mu_\o(s_1^+)})$ in light red, cf. (\ref{worldlines}). 
}\label{fig_D}
\end{figure}

\subsection{Main results}
\label{main_result}
Before stating the main results let us define, in detail, the source and observation neighborhoods of the two future pointing time-like curves
$$\mu_{\i}:[t_0^-,t_0^+]\to M \quad \text{and}\quad \mu_{\o}:[s_0^-,s_0^+]\to M$$
satisfying \eqref{worldlines}. We begin by extending the time-like paths to slightly larger intervals 
$$\mu_{\i}:[t_1^-,t_1^+]\to M \quad \text{and}\quad \mu_{\o}:[s_1^-,s_1^+]\to M,$$
and proceed to define the source and observation regions as a foliation of time-like future pointing paths near the paths $\mu_\i((t_1^-,t_1^+))$ and $\mu_\o((s_1^-,s_1^+))$. To make this precise, we use Fermi coordinates near these paths. 

Let $\{\alpha_{i}\}_{i=1}^n$ be an orthonormal basis for $\dot{\mu}_{\i}(t_1^-)^\perp$, and subsequently consider $\{e_{i}(t)\}_{i=1}^{n}$ to denote the parallel transport of $\{\alpha_i\}_{i=1}^n$ along $\mu_{\i}$ to the point $\mu_{\i}(t)$. Let $$F_{\i}: (t_1^-,t_1^+) \times B(0,\delta) \to M$$ 
be defined through
$$ F_{\i}(t,y) = \exp_{\mu_{\i}(t)}(\sum_{i=1}^n y^i e_{i}(t)).$$
Here, $B(0,\delta)$ is the ball of radius $\delta$ in $\R^n$. For $\delta$ sufficiently small, the map $F_{\i}$ is a smooth diffeomorphism, and the paths
$$ \mu_a(t) = F_{\i}(t,a) \quad a \in B(0,\delta)$$
are smooth time-like paths. We define $F_{\o}: (s_1^-,s_1^+)\times B(0,\delta)\to M$ analogously as above with $\mu_{\i}$ replaced by $\mu_{\o}$. Finally, we define the source and observation regions through the expression
\bel{foliation}  
\begin{aligned}
\Omega_{\i}&= \{ F_\i(t,y)\,:\, t \in (t_1^-,t_1^+),\quad y\in B(0,\delta)\}\\
\Omega_{\o}&= \{ F_\o(t,y)\,:\, t \in (s_1^-,s_1^+),\quad y \in B(0,\delta)\}. 
\end{aligned}
\ee 
We will impose that $\delta$ is small enough so that the following condition is satisfied. This can always be guaranteed in view of \eqref{worldlines}.
\bel{fol_cond}
\Omega_\i \cap J^{-}(F_\o(\{s_1^-\}\times B(0,\delta))=\emptyset \quad \text{and}\quad \overline{\Omega_\o}\cap J^+(\mu_\i(t_1^+))=\emptyset.
\ee 

Our main result regarding the inverse problems for the semi-linear and quasi-linear models above can be stated as follows.

\begin{theorem}
\label{t1}
Let $m \geq 3$ be an integer and $(M^{(1)},g^{(1)})$, $(M^{(2)},g^{(2)})$ be smooth globally hyperbolic Lorentzian manifolds of dimension $1+n\geq 3$. Let $G^{(j)}_z$, $j=1,2$, be a symmetric tensor on $M^{(j)}$ that satisfies conditions (i)--(ii), and recall that $g^{(j)}(x)=G^{(j)}(x,0)$ for all $x \in M^{(j)}$. Let $\mu_\i^{(j)}:[t_0^-,t_0^+]\to M^{(j)}$ and $\mu_{\o}^{(j)}:[s_0^-,s_0^+]\to M^{(j)}$ be smooth time-like paths satisfying \eqref{worldlines}. For $j=1,2$, let the source region $\Omega_\i^{(j)}$ and the observation region $\Omega^{(j)}_{\o}$ be defined by \eqref{foliation}. We assume that these neighborhoods are sufficiently small so that \eqref{fol_cond} holds. Let $T>0$ be sufficiently large so that
$$ \mathbb D_e^{(j)}\subset (-T,T)\times M^{(j)}_0\quad \text{for $j=1,2$},$$
and also that there exists isometric diffeomorphisms
$$\Psi^k:(\Omega_{k}^{(1)},g^{(1)}|_{\Omega_{k}^{(1)}})\to (\Omega_{k}^{(2)},g^{(2)}|_{\Omega_{k}^{(2)}})\quad k \in \{\text{\em{in}},\text{\em{out}}\}.$$  
Next, and for $j=1,2$, we consider the source-to-solution maps $\mathscr L^{(j)}$ and $\mathscr N^{(j)}$ associated to \eqref{pf0}--\eqref{pf} respectively and assume that one of the following statements hold:
\begin{itemize}
\item[(i)]{$\Psi^{\i}\circ(\mathscr L^{(1)}(f))= \mathscr L^{(2)}(f\circ (\Psi^{\i})^{-1})$ for all sources $f\in \mathscr C^{(1)}_{\Omega_\i^{(1)}}$,}
\end{itemize}
or
\begin{itemize}
\item[(ii)]{$\Psi^\i\circ (\mathscr N^{(1)}(f))= \mathscr N^{(2)}(f\circ (\Psi^{\i})^{-1})$ for all sources $f\in \mathscr C^{(1)}_{\Omega_\i^{(1)}}$,}
\end{itemize}
where the set $\mathscr C^{(1)}_{\Omega^{(1)}_\i}$ is defined by \eqref{C_def} associated to $T>0$ and the manifold $(M^{(1)},g^{(1)})$.

Then, under the hypotheses above, there exists a smooth diffeomorphism $\Psi:\mathbb D^{(1)} \to \mathbb D^{(2)}$ that is equal to $\Psi^\i$ on the set $\Omega^{(1)}_{\i}\cap \mathbb D^{(1)}$ and equals $\Psi^\o$ on the set $\Omega^{(1)}_{\o}\cap \mathbb D^{(1)}$ and such that
$$\Psi^*g^{(2)}=c\,g^{(1)}\quad \text{on $\mathbb D^{(1)}$},$$
for some smooth positive valued function $c=c(x)$ on $\mathbb D^{(1)}$. Moreover, in the case that statement (i) holds and if $(n,m)\neq(3,3)$, we have $c\equiv 1$ on the causal diamond $\mathbb D^{(1)}$. 
\end{theorem}

\begin{remark}
Note that if $\mu_{\o}(s_1^+) \notin I^+(\mu_{\i}(t_1^-)$, then $\mathbb D$ is the empty set and the content of the previous theorem is empty. Therefore, it is implicitly assumed in this paper in addition to \eqref{worldlines} that $\mu_{\o}(s_1^+) \in I^+(\mu_{\i}(t_1^-))$. We also remark that the recovery of the conformal factor in the exceptional case $(n,m)=(3,3)$ is briefly addressed in the last section of the paper.
\end{remark}

\subsection{Recovery of geometry from the  three-to-one scattering relation}
\label{subsec: 3-1 relation}


The proof of Theorem~\ref{t1} will be divided into an analytical and a geometrical part, with Sections~\ref{prelim_sec}--\ref{linearization_section_gauss} covering the analytical part and Sections~\ref{geometry_prelim}--\ref{earliest_sec} covering the geometric part of the analysis. In the analysis part, we use the idea of multiple-fold linearization of the wave equation first used {\mltext in \cite{KLU}} together with the principle of propagation of singularities for the wave equation, resulting in a geometrical data on the set $\mathbb D$,  the {\em three-to-one scattering relation}, that we will next define.

Before formulation of the definition, we set some notations.
We say that geodesics $\gamma_{v_{1}}$ and $\gamma_{v_2}$
are distinct, if the maximal geodesics that are extensions of $\gamma_{v_{1}}$ and $\gamma_{v_{2}}$ do coincide as subsets of $M$.
Also, for $v=(x,\eta)\in L^+M$, let $s(v)=\sup\{s>0:\ \gamma_{v}([0,s))\subset M\}$ 
and $\rho(v)=\sup\{s\in [0,s(v)]:\ \gamma_{v}(s)\not\in I^+(x)\}$. As discussed in
Subsection \ref{subsec: optimizing geodesic}, $\gamma_v(\rho(v))$ is called the first cut point
of $\gamma_v$.

For $v=(x,\eta)\in L^+M$, let
$\overrightarrow{\;\gamma_{v}}=\gamma_{x,\eta}([0,s(x,\eta)))$
be the light-like geodesics that is maximally extended in $M$ to the future 
from $v$. Also, let 
 $\overleftarrow{\;\gamma_{v}}=\gamma_{x,-\eta}([0,s(x,-\eta)))$
be the geodesic that emanates from  $v$ to  the past. 

Next we consider a set $\rel\subset (L^+M)^4$ of 4-tuples of vectors 
$(v_0,v_1, v_2,v_3)$ and say that these vectors satisfy relation
$\rel$ if $(v_0,v_1, v_2,v_3)\in \rel$.

\begin{definition}\label{def: good relation}
Let $\Omega_\i, \Omega_\o \subset M$ be open. We say that  a relation $\rel \subset L^+ \Omega_\o \times (L^+ \Omega_\i)^3$
is a three-to-one scattering relation if 
it 
has the  following two properties:
\begin{itemize}
\item[(R1)] If $(v_0,v_1, v_2,v_3) \in \rel$,
then there exists an intersection point $y \in \overleftarrow{\;\gamma_{v_0}} \cap \bigcap_{j=1}^3  \overrightarrow{\gamma_{v_{j}}}$.
\item[(R2)] Assume that $\gamma_{v_{j}}$, $j=0,1,2,3$, are distinct
and  there exists $y \in \overleftarrow{\;\gamma_{v_{0}}} \cap \bigcap_{j=1}^3  \overrightarrow{\gamma_{v_{j}}}.$ Moreover, assume that  
$y=\gamma_{v_{0}}(s_0)$ with $s_0\in (-\rho(v_0),0]$ and $y=\gamma_{v_{j}}(s_j)$  for all $j=1,2,3$,
with
 $s_j\in [0,\rho(v_j))$.
Denote $\xi_j=\dot \gamma_{v_{j}}(s_j)$  for $j=0,1,2,3$, and
assume that $\xi_0 \in \linspan(\xi_1, \xi_2, \xi_3)$. Then, 
 it holds that $(v_{0},v_{1}, v_{2},v_{3}) \in \rel$.
\end{itemize}
\end{definition}
In other words, (R1) means  that if 
$(v_0,v_1, v_2,v_3) \in \rel$ then it is necessary that the future pointing geodesics 
$\gamma_{v_1},\gamma_{v_2},$ and $\gamma_{v_3}$ must intersect at some point $y$, and some future pointing geodesic emanating from $y$ arrives to $v_0$. The condition (R2) means for $(v_0,v_1, v_2,v_3) \in \rel$  it is sufficient that the future pointing geodesics $\gamma_{v_1},\gamma_{v_2},$ and $\gamma_{v_3}$ intersect at some point $y$ before their first cut points,
and that the past pointing null geodesic $\gamma_{v_0}$ arrives to the point $y$ in the direction $\xi_0$ that is in the span of the velocity vectors of 
$\gamma_{v_1},\gamma_{v_2},$ and $\gamma_{v_3}$ at the point $y$ and finally that the geodesic $\gamma_{v_0}([s_0,0])$ has no cut-points.


The following theorem states that the three-to-one scattering relation determines uniquely the topological, differential and conformal structure of the set $\mathbb D$.

\begin{theorem}
\label{t2}
{\mltext Let $(M^{(1)},g^{(1)})$, $(M^{(2)},g^{(2)})$ be smooth} globally hyperbolic Lorentzian manifolds of dimension $1+n\geq 3$. Let $\mu_\i^{(j)}:[t_0^-,t_0^+]\to M^{(j)}$ and $\mu_{\o}^{(j)}:[s_0^-,s_0^+]\to M^{(j)}$ be smooth time-like paths satisfying \eqref{worldlines}. For $j=1,2$, let the source region $\Omega_\i^{(j)}$ and the observation region $\Omega^{(j)}_{\o}$ be defined by \eqref{foliation}. We assume that these neighborhoods are sufficiently small so that \eqref{fol_cond} holds. Moreover, we assume that there are isometric diffeomorphisms
$$\Psi^k:(\Omega_{k}^{(1)},g^{(1)}|_{\Omega_{k}^{(1)}})\to (\Omega_{k}^{(2)},g^{(2)}|_{\Omega_{k}^{(2)}})\quad k \in \{\text{\em{in}},\text{\em{out}}\}.$$ 
Suppose, next, that there are relations $\rel^{(j)}\subset 
L^+\Omega_{out}^{(j)}\times (L^+\Omega_{in}^{(j)})^3$, $j=1,2$,
that satisfy conditions (R1) and (R2) in Definition \ref{def: good relation}
for manifolds $(M^{(j)},g^{(j)})$, and that
   \begin{align}
\rel^{(2)}=\bigg\{(\Psi^{in}_*v_0,\Psi^{out}_*v_1,\Psi^{out}_*v_2,\Psi^{out}_*v_3)\ \bigg|\ (v_0,v_1,v_2,v_3)\in \rel^{(1)}\bigg\}.
   \end{align}
Then there exists a smooth diffeomorphism $\Psi:\mathbb D^{(1)} \to \mathbb D^{(2)}$ that is equal to $\Psi^\i$ on the set $\Omega^{(1)}_{\i}\cap \mathbb D^{(1)}$ and equals $\Psi^\o$ on the set $\Omega^{(1)}_{\o}\cap \mathbb D^{(1)}$ and such that
$$\Psi^*g^{(2)}=c\,g^{(1)}\quad \text{on $\mathbb D^{(1)}$},$$
for some smooth positive valued function $c=c(x)$ on $\mathbb D^{(1)}$. 
\end{theorem}

The motivation of Definition \ref{def: good relation} and Theorem \ref{t2} is to provide a general framework that allows the results of this paper to be applicable for other non-linear
hyperbolic equations similar to those studied in this paper. Indeed, to consider 
some different kind of  non-linear
hyperbolic equation, one can define that $(v_0,v_1, v_2,v_3)$
satisfies the relation $\rel_{sim}$ if  three singular waves sent to directions
$v_1,v_2$ and $v_3$ interact so that the interaction produces 
a wave which wave front contain the covector corresponding to $v_0$.
Then to apply Theorem \ref{t2} one has to show that $\rel_{sim}$ satisfies conditions (R1) and (R2).
{\mltext We note that condition (R1) is  natural as the second order interaction of waves 
does not produce singularities that propagate to new directions, see \cite{GU1,KLU}. 
Condition (R2) is motivated by the general results for the interaction of three waves, see \cite{MR1,SaB,SaB-W} and references therein. We emphasize that to verify condition (R2) one has to consider only geodesics
that have no conjugate points and thus this condition can be verified without analyzing interaction of waves near caustics.}
\subsection{Previous literature}

The study of non-linear wave equations is a fascinating topic in analysis with a rich literature. In contrast with the study of linear wave equations, there are numerous challenges in studying the existence, uniqueness and stability of solutions to such equations. These equations physically arise in the study of general relativity, such as the Einstein field equations. They also appear in the study of vibrating systems or the detection of perturbations arising in electronics, such as the telegraph equation or the study
of semi-conductors, see for instance \cite{CH}. We mention in particular that the quasi-linear model \eqref{pf} studied in this paper is a model for studying Einstein's equations in wave coordinates \cite{LT}.  

{\mltext  
This paper uses extensively the non-linear interaction of three waves to solve the inverse problems. To analyze this, we use
gaussian beams. An alternative way to consider the non-linear interaction  is to use microlocal analysis and conormal singularities, see \cite{GU1,GuU1,MU1}. 
There are many results on such non-linear interaction,  starting with the studies of 
 Bony \cite{Bony}, Melrose and Ritter \cite{MR1}, and Rauch and Reed  \cite{R-R}.
 However, these studies concern the direct problem and differ from the setting of this paper  in that they assumed
 that the geometrical setting of the interacting singularities, and in particular the locations and types of caustics, is known a priori. 
 In inverse problems  we need to study waves on
an unknown manifold,
so we do not know the  underlying geometry and, therefore, the location of the singularities
of the waves. 
For example, the waves can have caustics that
may even be of an unstable type.}

{\mltext
Earlier, inverse problems for non-linear hyperbolic equations with unknown metric
 have been studied using interaction of four or more waves and only in (1+3)-dimensional spacetimes. Inverse problems for non-linear scalar wave equation with a quadratic non-linearity
 was studied in \cite{KLU} using the multiple-fold linearization. 
 Together with the phenomena of propagation of singularities for the wave equation, the authors reduced the inverse problem for the wave equation to the study of light observation sets. This approach was  extended in \cite{HU,Ivanov,LSa}.
 In \cite{KLOU}, the coupled 
 Einstein and scalar field equations were studied. The result has been more recently strengthened  in \cite{LUW,UW} for the Einstein scalar field  equations with general sources and for the Einstein-Maxwell equations. In particular, a technique 
to determine the conformal factor using the microlocal symbols of the observed waves is developed in \cite{UW}.}

Aside from the works mentioned above, the majority of works have been on inverse problems
 for semi-linear wave equations with quadratic non-linearities studied in \cite{KLU}, a general semi-linear term studied in \cite{HUZ,LUW} and with quadratic derivative in \cite{WZ},
{\mltext see also \cite{Lassas} and references therein.} All these 
 works concern the {\mltext $(1+3)$-dimensional case}. In recent works \cite{CLOP,CLOP2,FO}, the authors have also studied problems of recovering zeroth and first order terms for semi-linear wave equations
 {\mltext with Minkowski metric}. We note that three wave interactions were
 used in \cite{CLOP,CLOP2} to determine the lower order terms in the equations. In \cite{FO19,LLLS19} similar multiple-fold linearization methods have been introduced to study inverse problems for elliptic non-linear equations, see also \cite{KUI}. 

All of the aforementioned works consider inverse problems for various types of non-linear wave equations subject to small sources. The presence of a non-linear term in the PDE is a strong tool in obtaining the uniqueness results. To discuss this feature in some detail, we note that the analogous inverse problem for the linear wave equation (see \eqref{linear_eq}) is still a major open problem. For this problem, much is known in the special setting that the coefficients of the metric are time-independent. We refer the reader to the work of Belishev and Kurylev in \cite{Bel92} that uses the boundary control method introduced in \cite{Bel87} to solve this problem and to \cite{AKKLT,KKL,KOP} for an state of the art result in the application of the boundary control method. The boundary control method is known to fail in the case of general time-dependent coefficients, since it uses the unique continuation principle of Tataru \cite{Tataru}. This principle is known to be false in the cases that the time-dependence of coefficients is not real-analytic \cite{Al,AB}.  We refer the reader to \cite{Esk} for recovery of coefficients of a general linear wave equation under an analyticity assumption with respect to the time coordinate. 

In the more challenging framework of general time-dependent coefficients and by using the alternative technique of studying the propagation of singularities for the wave equation, the inverse problem for the linear wave equation (see \eqref{linear_eq}) is reduced to the injectivity of the scattering relation on $M$, see the definition \eqref{scattering_def}. The injectivity of the scattering relation is open unless the geometry of the manifold is static and an additional convex foliation codition is satisfied \cite{SUV} on the spatial part of the manifold. In the studies of recovery of sub-principal coefficients for the linear wave equation, we refer the reader to the recent works \cite{FIKO,FIO,Ste} for recovery of zeroth and first order coefficients and to \cite{SY} for a reduction from the boundary data for the inverse problem associated to \eqref{linear_eq} to the study of geometrical transforms on $M$. This latter approach has been recently extended to general real principal type differential operators \cite{OSSU}.

The main underlying principle in the presence of a non-linearity is that linearization of the equation near the trivial solution results in a non-linear interaction of solutions to the linear wave equation producing much richer dynamics for propagation of singularities. Owing to this richer dynamics, and somewhat paradoxically, inverse problems for non-linear wave equations have been solved in much more general geometrical contexts than their counterparts for the linear wave equations.

Let us now discuss some of the main novelties of the present work. Firstly, we consider three-fold linearization of the non-linear equations \eqref{pf0}--\eqref{pf} and can therefore analyze inverse problems using interaction of three waves instead of earlier works relying on interaction of four waves. This makes it possible to consider more general equations with simpler techniques. Due to the new techniques, we can consider inverse problems on Lorenzian manifolds with any dimension
 $n+1\ge 3$. As a second novelty, we introduce a new concept, the 3-to-1 scattering relation that can be applied for many kinds of non-linear equations and which we hope to be useful for other researchers in the field of inverse problems. Also, this makes it possible to consider inverse problems in the remote sensing setting that includes both forward and back scattering problems. Finally, we mention our quite general quasi-linear model problem \eqref{pf} with an unknown non-linearity in the leading order term. We successfully study this complicated model with the use of Gaussian beams and show that the source-to-solution map determines the three-to-one scattering relation.

\subsection{Outline of the paper}

We begin with some preliminaries in Section~\ref{prelim_sec}, starting with Proposition~\ref{wellposed_prop} that shows that the forward problems \eqref{pf0}--\eqref{pf} are well-posed. We also recall the technique of multiple-fold linearization and apply it to the semi-linear and quasi-linear equation separately. This will relate the source-to-solution maps, $\mathscr L$ and $\mathscr N$, to the study of products of solutions to the linear wave equation, see \eqref{sotosol_data_semi} and \eqref{sotosol_data_quasi}. In Section~\ref{formalgaussian}, we briefly recall the construction of the classical Gaussian beams for the linear wave equation. We also show that it is possible to explicitly construct real valued sources supported in the source and observation regions, that generate exact solutions to the linear wave equations that are close in a suitable sense to the real parts of Gaussian beams. In Section~\ref{linearization_section_gauss} we prove the main analytical theorems, showing that the source-to-solution maps lead to a three-to-one scattering relation, see Theorem~\ref{thm_anal_data_1}--\ref{thm_anal_data_2}. Combined with Theorem~\ref{t2} this proves the first half of Theorem~\ref{t1} on the recovery of the topological, differential and conformal structure of the casual diamond $\mathbb D$. 

The geometrical sections of the paper are concerned with the study of a general three-to-one scattering relation $\rel$ on $M$ and the proof of Theorem~\ref{t2}. In Section~\ref{geometry_prelim}, we recall some technical lemmas on globally hyperbolic Lorentzian geometries. In Section~\ref{earliest_sec}, we prove Theorem~\ref{thm_obs}, showing that it is possible to use the three-to-one scattering relation to construct the earliest arrivals on $\mathbb D$. Combining this with the results {\mltext of \cite{KLU}} leads to unique recovery of the topological, differential and conformal structure of $(\mathbb D, g|_{\mathbb D})$ that completes the proof of Theorem~\ref{t2}. 

Finally, Section~\ref{pf_thm_sec} is concerned with the proof of Theorem~\ref{t1}. The first half of the proof, that is the recovery of the topological, differential and conformal structure of the manifold follows immediately from combining Theorem~\ref{thm_anal_data_1}--\ref{thm_anal_data_2} together with Theorem~\ref{t2}. The remainder of this section deals with the recovery of the conformal factor $c$ on $\mathbb D$.

\section{Preliminaries}
\label{prelim_sec}

\subsection{Forward problem}
\label{forward_section}
In this section, we record the following proposition about existence and uniqueness of solutions to \eqref{pf0}--\eqref{pf} subject to suitable sources $f$. The local existence of solutions to semi-linear and quasi-linear wave equations are well-studied in the literature, see for example \cite{Ho5,S,T,W}.
\begin{proposition}
\label{wellposed_prop}
Given any open and bounded set $O \subset (-T,T)\times M_0$, there exists a sufficiently small $r_O>0$ such that given any $f \in \mathscr C_O$ (with $\mathscr C_O$ as defined by \eqref{C_def}),
each of the equations \eqref{pf0} or \eqref{pf} admits a unique real-valued solution $u$ in the energy space
$$ u \in L^{\infty}(-T,T;H^{n+2}(M_0))\cap \mathcal C^{0,1}(-T,T;H^{n+1}(M_0))\cap \CI^2((-T,T)\times M_0).$$
Moreover, the dependence of $u$ to the source $f$ is continuous.
\end{proposition}

The proof of this proposition 
in the semi-linear case \eqref{pf0} follows by minor modifications {\mltext to \cite{Kato1975}.} In the quasi-linear case, the proof 
follows with minor modifications to the proof of \cite[Theorem 6]{W}.

\subsection{Multiple-fold linearization}
\label{linearization_section}

We will discuss the technique of multiple-fold linearization of non-linear equations that was first used in \cite{KLU}. Before presenting the approach in our semi-linear and quasi-linear settings, we consider the linear wave equation on $M$,
\bel{linear_eq}
\begin{aligned}
\begin{cases}
\Box_{g} u =f, 
&\text{on  $M$},
\\
u=0,
&\text{on $M\setminus J^+(\supp f)$}
\end{cases}
    \end{aligned}
\ee
with real valued sources $f \in C^{\infty}_c(\Omega_\i)$. We also need to consider the wave equation with reversed causality, that is to say,
\bel{linear_eq_adjoint}
\begin{aligned}
\begin{cases}
\Box_{g} u =f, 
&\text{on  $M$},
\\
u=0,
&\text{on $M\setminus J^-(\supp f)$}
\end{cases}
    \end{aligned}
\ee
with real valued sources $f \in C^{\infty}_c(\Omega_\o)$. 

\subsubsection{m-fold linearization of the semi-linear equation \eqref{pf0}}

Let $m\geq 3$. We consider real-valued sources $f_0 \in C^{\infty}_c(\Omega_\o)$ and $f_j \in C^{\infty}_c(\Omega_\i)$, $j=1,\ldots,m$. We denote by $u_j$, $j=1,\ldots,m$ the unique solution to \eqref{linear_eq} subject to source $f_j$ and denote by $u_0$ the unique solution to \eqref{linear_eq_adjoint} subject to source $f_0$. Let $\varepsilon=(\varepsilon_1,\ldots,\varepsilon_m) \in \R^m$ be a small vector and define the source
$$ f_\varepsilon= \sum_{j=1}^m\varepsilon_j f_{j}.$$
Given $\varepsilon$ sufficiently close to the origin in $\R^m$, we have $f_\epsilon \in \mathscr C_{\Omega_\i}$. Let us define
\bel{w} w = \frac{\p^m}{\p\varepsilon_1 \p\varepsilon_2\ldots\p\varepsilon_m}\,u_\varepsilon\bigg |_{\varepsilon=0},\ee
where $u_\varepsilon$ is the unique small solution to \eqref{pf0} subject to the source $f_\varepsilon \in \mathscr C_{\Omega_\i}$.

It is straightforward to see that the function $w$ defined by \eqref{w} solves
\bel{three_fold_1}
\begin{aligned}
\begin{cases}
\Box_{g} w =-m! \,u_1\,u_2\,u_3\ldots u_m, 
&\text{on  $M$},
\\
w=0,
&\text{on $M\setminus J^+(\bigcup_{j=1}^m\supp f_m)$,}
\end{cases}
    \end{aligned}
\ee
Multiplying the latter equation with $u_0$ and using the Green's identity,
\bel{green}
\int_M w\,\Box_g u_0\,dV_g=\int_M u_0\,\Box_g w \,dV_g,
\ee
we deduce that
\bel{sotosol_data_semi}
\int_{\Omega_\o} f_0 \frac{\p^m}{\p\varepsilon_1\ldots\p\varepsilon_m} \mathscr L f_\varepsilon \bigg|_{\varepsilon=0}\,dV_g = -m!\,\int_M u_0\,u_1\ldots u_m\,dV_g. 
\ee
We emphasize that by global hyperbolicity, the integrand on the right hand side is supported on the compact set
\bel{mho}  J^-(\supp f_0)\cap\bigcup_{j=1}^m J^+(\supp f_j)\subset \mathbb D_e \subset (-T,T)\times M_0,\ee
that makes the integral well-defined (see \eqref{optimal_reg}). Note that the latter inclusion is due to the hypothesis of Theorem~\ref{t1} on the size of $T$. We deduce from \eqref{sotosol_data_semi} that the source-to-solution map $\mathscr L$ for the semi-linear equation \eqref{pf0} determines the knowledge of integrals of products of $m$ solutions to the linear wave equation \eqref{linear_eq} and a solution to \eqref{linear_eq_adjoint}.
\subsubsection{Three-fold linearization of the quasi-linear equation \eqref{pf}}
We consider sources $f_j \in C^{\infty}_c(\Omega_\i)$, $j=0,1,2,3$ and for each small vector $\varepsilon=(\varepsilon_1,\varepsilon_2,\varepsilon_3)\in \R^3$, consider the three parameter family of sources 
$$f_\varepsilon=\varepsilon_1f_1+\varepsilon_2f_2+\varepsilon_3f_3 .$$  
Let $u_\varepsilon$ be the unique small solution to \eqref{pf} subject to the source $f_\varepsilon\in \mathscr C_{\Omega_\i}$. Recall that, by definition, $G(x,0)=g(x)$, $\p_z G(x,0)=0$ and $h(x)=\frac{1}{2} \frac{\p^2}{\p z^2} G(x,0)$. First, we note that the following identities hold in a neighborhood of $z=0$:
$$ G_z^{jk}=g^{jk}-S^{jk}\,z^2+\mathcal O(|z|^3),$$
$$ \left|\det G_z\right|=\left|\det g\right|\, (1+\Tr(hg^{-1})z^2)+\mathcal O(|z|^3),$$
where $S^{jk}=\sum_{j',k'=0}^n g^{jj'}\,h_{j'k'}\,g^{k'k}.$
Using these identities in the expression for $\Box_{G_z}$ with $z$ replaced with $u_\varepsilon$, it follows that the function
\bel{w_quasi}
w = \frac{\p^3}{\p\varepsilon_1 \p\varepsilon_2\p\varepsilon_3}\,u_\varepsilon \bigg |_{\varepsilon=0}
\ee
solves the following equation on $M$:
\begin{multline}
\label{big_exp_quasi}
\Box_g w - \Tr(hg^{-1}) u_1u_2f_3- \Tr(hg^{-1}) u_2u_3f_1- \Tr(hg^{-1}) u_3u_1f_2\\
-\sum_{j,k=0}^n\left| \det g\right|^{-\frac{1}{2}}\frac{\p}{\p x^j}(\left|\det g\right|^{\frac{1}{2}}\Tr(hg^{-1})g^{jk}u_1u_2\frac{\p u_3}{\p x^k})\\
-\sum_{j,k=0}^n\left| \det g\right|^{-\frac{1}{2}}\frac{\p}{\p x^j}(\left|\det g\right|^{\frac{1}{2}}\Tr(hg^{-1})g^{jk}u_2u_3\frac{\p u_1}{\p x^k})\\
-\sum_{j,k=0}^n\left| \det g\right|^{-\frac{1}{2}}\frac{\p}{\p x^j}(\left|\det g\right|^{\frac{1}{2}}\Tr(hg^{-1})g^{jk}u_3u_1\frac{\p u_2}{\p x^k})\\
+2\sum_{j,k=0}^n\left| \det g\right|^{-\frac{1}{2}}\frac{\p}{\p x^j}(\left|\det g\right|^{\frac{1}{2}}S^{jk}u_1u_2\frac{\p u_3}{\p x^k})\\
+2\sum_{j,k=0}^n\left| \det g\right|^{-\frac{1}{2}}\frac{\p}{\p x^j}(\left|\det g\right|^{\frac{1}{2}}S^{jk}u_2u_3\frac{\p u_1}{\p x^k})\\
+2\sum_{j,k=0}^n\left| \det g\right|^{-\frac{1}{2}}\frac{\p}{\p x^j}(\left|\det g\right|^{\frac{1}{2}}S^{jk}u_3u_1\frac{\p u_2}{\p x^k})=0,
\end{multline}
subject to the initial conditions $w=0$ on $M\setminus J^+(\bigcup_{j=1}^3\supp f_j)$. Note that the knowledge of the source-to-solution map $\mathscr N$ determines $w|_{\Omega_\o}$. Also,
$$ \nabla^g u = \sum_{j,k=0}^ng^{jk}\frac{\p u}{\p x^j} \frac{\p}{\p x^k}\quad \forall\, u \in \mathcal C^{\infty}(M),$$
which implies that
$$ \sum_{j,k=0}^n g^{jk}\frac{\p u}{\p x^j}\frac{\p v}{\p x^k}= \langle \nabla^g u,\nabla^g v \rangle_g,$$
and 
$$ \sum_{j,k=0}^nS^{jk}\frac{\p u}{\p x^j}\frac{\p v}{\p x^k}= \langle \nabla^g u,\nabla^g v \rangle_h,$$
for all $u, v \in \mathcal C^{\infty}(M)$.
Therefore, recalling that 
$$ dV_g = \left|\det g\right|^{\frac{1}{2}}\, dx^0\wedge dx^1\wedge \ldots \wedge dx^n,$$
and multiplying equation \eqref{big_exp_quasi} with $u_0$ that solves \eqref{linear_eq_adjoint} subject to a source $f_0\in \CI^{\infty}_c(\Omega_\o)$ followed by integrating by parts (analogously to \eqref{green}), we deduce that
\begin{multline}
\label{sotosol_data_quasi} 
\int_{\Omega_\o} f_0 \frac{\p^3}{\p\varepsilon_1\p\varepsilon_2\p\varepsilon_3} \mathscr N f_\varepsilon \bigg |_{\varepsilon=0}\,dV_g=\\
2\int_{M} \left(u_1u_2\langle \nabla^gu_3,\nabla^gu_0\rangle_h+u_2u_3\langle \nabla^gu_1,\nabla^gu_0\rangle_h+u_3u_1\langle \nabla^gu_2,\nabla^gu_0\rangle_h\right)\,dV_g\\
-\int_{M}\Tr(hg^{-1})\left(u_1u_2\langle \nabla^gu_3,\nabla^gu_0 \rangle_g+u_2u_3\langle \nabla^gu_1,\nabla^gu_0 \rangle_g+u_3u_1\langle \nabla^gu_2,\nabla^gu_0 \rangle_g\right)\,dV_g\\
+\int_M \Tr(hg^{-1})\,u_0\left(u_1u_2f_3+u_2u_3f_1+u_3u_1f_2\right)\,dV_g.
\end{multline}
Analogously to \eqref{sotosol_data_semi}, the integrands on the right hand side expression are all supported on the compact set 
$$ J^-(\supp f_0)\cap\bigcup_{j=1}^m J^+(\supp f_j)\subset \mathbb D_e \subset (-T,T)\times M_0,$$
for $T$ sufficiently large as stated in the hypothesis of Theorem~\ref{t1}.

\section{Gaussian beams}
\label{formalgaussian}

Gaussian beams are approximate solutions to the linear wave equation 
$$ \Box_g u=0 \quad \text{on $(-T,T)\times M_0$},$$
that concentrate on a finite piece of a null geodesic $\gamma: (a,b) \to (-T,T)\times M_0$, exhibiting a Gaussian profile of decay away from the segment of the geodesic. Here, we are considering an affine parametrization of the null geodesic $\gamma$, that is to say, 
\bel{affine}
\nabla^g_{\dot{\gamma}(s)}\dot{\gamma}(s)=0,\quad \langle \dot{\gamma}(s),\dot{\gamma}(s)\rangle_g=0\quad \forall\, s \in (a,b).
\ee
We will make the standing assumption that the end points $\gamma(a)$ and $\gamma(b)$ lie outside $(-T,T)\times M_0$. Gaussian beams are a classical construction and go back to \cite{Ba,Ralston}. They have been used in the context of inverse problems in many works, see for example \cite{BeKa,FIKO,FO,KaKu}. In order to recall the expression of Gaussian beams in local coordinates, we first briefly recall the well-known Fermi coordinates near a null geodesic. 

\begin{lemma}[Fermi coordinates]
\label{fermi}
Let $\hat\delta > 0$, $a < b$ and let $\gamma: (a-\hat{\delta},b+\hat{\delta}) \to M$ be a null geodesic on $M$ parametrized as given by \eqref{affine} and whose end points lie outside $(-T,T)\times M_0$. There exists a coordinate neighborhood  $(U,\psi)$ of $\gamma([a,b])$, with the coordinates denoted by $(y^0:=s,y^1,\ldots,y^n)=(s,y')$, such that:
\begin{itemize}
\item[(i)] {$\psi(U)=(a-\delta',b+\delta') \times B(0,\delta')$ where $B(0,\delta')$ is the ball in $\mathbb{R}^{n}$ centered at the origin with radius $\delta' > 0$.}
\item[(ii)]{$\psi(\gamma(s))=(s,\underbrace{0,\ldots,0}_{n \hspace{1mm}\text{times}})$}. 
\end{itemize}
Moreover, the metric tensor $g$ satisfies in this coordinate system  
    \begin{align}\label{g_on_gamma}
g|_\gamma = 2ds\otimes dy^1+ \sum_{\alpha=2}^n \,dy^\alpha\otimes \,dy^\alpha,
    \end{align}
and $\frac{\p}{\p y^i} g_{jk}|_\gamma = 0$ for $i,j,k=0,\ldots,n$. Here, $|_\gamma$ denotes the restriction on the curve $\gamma$.
\end{lemma}
We refer the reader to \cite[Section 4.1, Lemma 1]{FO} for a proof of this lemma. Using the Fermi coordinates discussed above, Gaussian beams can be written through the expression
\begin{equation}\label{tau_pos} \mathcal U_\lambda(y) = e^{i\lambda \phi(y)} A_\lambda(y) \quad \text{for} \quad \lambda>0\end{equation}
and
\begin{equation}\label{tau_neg} \mathcal U_\lambda(y) = e^{-i\lambda\bar{\phi}(y)}\bar{A}_\lambda(y) \quad \text{for}\quad\lambda<0.\end{equation}
Here, $\bar{\cdot}$ stands for the complex conjugation and the phase function $\phi$ and the amplitude function $A_\lambda$ are given by the expressions
\bel{phase-amplitude}
\begin{aligned}
 \phi(s,y') = \sum_{j=0}^{N} \phi_j(s,y')& \quad \text{and} \quad A_\lambda(s,y')= \chi(\frac{|y'|}{\delta'}) \sum_{j=0}^{N} \lambda^{-j}a_{j}(s,y'),\\
&a_j(s,y')=\sum_{k=0}^{N} a_{j,k}(s,y'),
\end{aligned}
\ee
where for each $j,k=0,\ldots,N$, $\phi_j$ is a complex valued homogeneous polynomial of degree $j$ in the variables $y^{1},\ldots, y^n$ and $a_{j,k}$ is a complex valued homogeneous polynomials of degree $k$ with respect to the variables $y^1,\ldots,y^n$, and finally $\chi:\R\to\R$ is a non-negative smooth function of compact support such that $\chi(t)=1$ for $|t| \leq \frac{1}{4}$ and $\chi(t)=0$ for $|t|\geq \frac{1}{2}$.

The determination of the phase terms $\phi_j$ and amplitudes $a_j$ with $j=0,1,2,\ldots,N$ are carried out by a WKB analysis in the semi-classical parameter $\lambda$, based on the requirement that
\bel{wkb} 
\begin{aligned}
&\frac{\p^{|\alpha|}}{\p y'^\alpha}\langle d\phi,d\phi\rangle_g=0 \quad \text{on $(a,b)\times\{y'=0\}$,}\\ 
&\frac{\p^{|\alpha|}}{\p y'^{\alpha}}\left( 2\langle d\phi,da_j \rangle_g+ (\Box_g\phi)a_j + i \Box_ga_{j-1}\right)=0\quad \text{on $(a,b)\times\{y'=0\}$},
\end{aligned}\ee
for all $j=0,1,\ldots,N$ and all multi-indices $\alpha=(\alpha_1,\ldots,\alpha_n) \in \{0,1,\ldots\}^n$ with $|\alpha|=\alpha_1+\ldots+\alpha_n\leq N$.

We do not proceed to solve these equations here as this can be found in all the works mentioned above, but instead summarize the main properties of Gaussian beams as follows:
\begin{enumerate}
\item {$\phi(s,0)=0$. }
\item {$\Im(\phi)(s,y') \geq C |y'|^2$ for all points $y \in (a,b)\times B(0,\delta')$.}
\item {$\| \Box_{g} \mathcal U_\lambda \|_{H^k((-T,T)\times M_0)} \lesssim |\lambda|^{-K},$ where $K=\frac{N+1}{2}+\frac{n}{4}-k-2$.}
\end{enumerate}
Here, $\Im$ stands for the imaginary part of a complex number and by the notation $A\lesssim B$ we mean that there exists a constant $C$ independent of the parameter $\lambda$ such that $A \leq C B$.

For the purposes of our analysis, we also need to recall the Fermi coordinate expressions for $\phi_1$, $\phi_2$ and $a_{0,0}$, see \eqref{phase-amplitude}. We recall from \cite{FO} that
\begin{equation}
\label{explicit_terms}
\begin{aligned}
\phi_0(s,y')=0,\quad \phi_1(s,y')&= y_1,\quad \phi_2(s,y')=\sum_{j,k=1}^n H_{jk}(s)y^jy^k,\\
\quad a_{0,0}(s)&=(\det Y(s))^{-\frac{1}{2}},
\end{aligned}
\end{equation}  

The matrices $H$ and $Y$ are described as follows. The symmetric complex valued matrix $H$ solves the Riccati equation
\bel{riccati}
\frac{d}{ds} H + HCH + D=0, \quad \forall s \in (a,b), \quad H(\hat{s}_0)=H_0,\quad \Im H_0>0.
\ee
where $C$ and $D$ are the matrices defined through
\bel{Cmatrix}
\begin{cases}
C_{11}= 0&\\
C_{jj}=2& \quad j=2,\ldots,n, \\
 C_{jk}=0& \quad \text{otherwise,}
\end{cases}
\qquad \text{where $D_{jk}= \frac{1}{4} \frac{\p^2 g^{11}}{\p y^j \p y^k}$}.
\ee 

We recall the following result from \cite[Section 8]{KKL} regarding solvability of the Riccati equation:
\begin{lemma}
\label{ricA}
Let $\hat{s}_0 \in (a,b)$ and let $H_0=Z_0Y_0^{-1}$ be a symmetric matrix with $\Im H_0 > 0$.
The Riccati equation (\ref{riccati}), together with the initial condition $H(\hat{s}_0) = H_0$, has a unique solution $H(s)$ for all $s \in (a,b)$. We have $\Im H>0$ and $H(s)=Z(s)Y^{-1}(s)$, where the matrix valued functions $Z(s),Y(s)$ solve the first order linear system
$$ \frac{d}{ds} Y = CZ\quad \text{and}\quad  \frac{d}{ds} Z = -DY, \quad \text{subject to} \quad Y(\hat{s}_0)=Y_0,\quad Z(\hat{s}_0)=H_0.$$ 
Moreover, the matrix $Y(s)$ is non-degenerate on $(a,b)$, and there holds
$$
\det(\Im H(s)) \cdot |\det(Y(s))|^2=\det(\Im(H_0)).
$$
\end{lemma}

As for the remainder of the terms $\phi_j$ with $j\geq 3$ and the rest of the amplitude terms $a_{j,k}$ with $j,k$ not both simultaneously zero, we recall from \cite{FO} that they solve first order ODEs along the null geodesic $\gamma$ and can be determined uniquely by fixing their values at some fixed point $\hat{s}_0 \in (a,b)$.

\section{Source terms that generate real parts of Gaussian beams}
\label{sourcegauss}

Let $v=(q,\xi) \in L^+\Omega_\i$ or $v=(q,\xi) \in L^+\Omega_\o$ where $L^+\Omega_\i$ and $L^+\Omega_\o$ denote the bundle of future pointing light like vectors in $\Omega_\i$ and $\Omega_\o$ respectively. We consider a Gaussian beam solution $\mathcal U_\lambda$ of order $N=\lceil\frac{3n}{2}\rceil+8$ as in the previous section, concentrating on a future-pointing null geodesic $\gamma_{q,\xi}:(a,b)\to M$ passing through $q$ in the direction $\xi$. Here $\gamma$ is parametrized as in \eqref{affine} subject to 
\bel{geo_par}\gamma(0)=q \quad \text{and}\quad \dot{\gamma}(0)=\xi.\ee 
As before, we assume that the endpoints of the null geodesic lie outside the set $(-T,T)\times M_0$. 

In this section, we would like to give a canonical way of constructing Gaussian beams followed by a canonical method of constructing real valued sources that generate the real parts of these Gaussian beams. Recall that the construction of a Gaussian beam $\mathcal U_\lambda$ that concentrates on $\gamma$ has a large degree of freedom associated with the various initial data for the governing ODEs of the phase and amplitude terms in \eqref{phase-amplitude}. The support of a Gaussian beam around a geodesic that is given by the parameter $\delta'$ is also another degree of freedom in the construction.  

We start with fixing the choice of the phase terms $\phi_j$ with $j\geq 3$ and the amplitude terms  $a_{j,k}$ with $j,k=0,1,\ldots,N$ (and both indices not simultaneously zero), by assigning zero initial value for their respective ODEs along the null geodesic $\gamma$ at the point $q=\gamma(0)$. Therefore, to complete the construction of the Gaussian beam $\mathcal U_\lambda$, it suffices to fix a small parameter $\delta'>0$ and also to choose $Y(0)$ and $Z(0)$ in Lemma~\ref{ricA}. This will then fix the remaining functions $\phi_2$ and $a_{0,0}$ in the Gaussian beam construction. To account for the latter two degrees of freedom in the construction, we introduce the notation $\iota=(Y(0),Z(0))\in \mathcal T$ where 
\begin{equation}\label{Y_init}
\begin{aligned}\mathcal T = \{ (Y_0,Z_0)\in \C^{(1+n)^2}\times \C^{(1+n)^2}\,:\,&\text{ $Z_0Y_0^{-1}$ is symmetric} \\
&\text{and $\Im( Z_0Y_0^{-1})>0$.}\}\end{aligned}\ee

Using the notations above, we can explicitly determine (or identify) Gaussian beam functions 
\bel{gauss_par}\mathcal U_\lambda=\mathcal U_{\lambda,v,\iota,\delta'},\ee 
subject to each $\lambda \in \R$ that denotes the asymptotic parameter in the construction, a vector $v=(q,\xi) \in L^+\Omega_\i$ or $v=(q,\xi) \in L^+\Omega_\o$ that fixes the geodesic $\gamma_{v}$, a small $\delta'>0$ that fixes the support around the geodesic and the choice of $\iota \in \mathcal T$ that fixes the initial values for the ODEs governing $\phi_2$ and $a_{0,0}$. As discussed above the rest of the terms in the Gaussian beam are fixed by setting the initial values for their ODEs to be zero at the point $q$. For the sake of brevity and where there is no confusion, we will hide these parameters in the notation $\mathcal U_\lambda$. 

Our aim in the remainder of this section is to construct a source $f\in \CI^{\infty}_c(\Omega_\i)$ such that 
the solution to the linear wave equation \eqref{linear_eq} with this source term is close to the real part of the complex valued Gaussian beam $\mathcal U_\lambda$, in a sense that will be made precise below. We then give an analogous construction of sources for $\Omega_\o$.
\begin{remark}
Let us emphasize that the need to work with real-valued sources is due to the fact that in the case of the quasi-linear wave equation \eqref{pf}, the solution to the wave equation appears in the tensor $G(x,u)$. Therefore, for the sake of physical motivations of our inverse problem, it is crucial to work with real-valued solutions to the wave equations \eqref{pf0}--\eqref{pf}. 
\end{remark}

To simplify the notations, we use the embedding of $M$ into $\R \times M_0$ to define the $(x^0,x')$-coordinates on $M$ that was described in the introduction. Next, we write $q=(q_0,q')$ and define two function $\zeta_{\pm,v,\delta'} \in \CI^\infty(\R)$ such that 
\bel{zeta_-}
\zeta_{-,v,\delta'}(x^0)= \begin{cases}
        0  & \text{if}\,\, x^0\leq q_0-\delta',\\
             1  & \text{if}\,\,x^0\geq q_0-\frac{\delta'}{2}.
     \end{cases}
\ee
and
\bel{zeta_+}
\zeta_{+,v,\delta'}(x^0) = \begin{cases}
        0  & \text{if}\,\, x^0\geq q_0,\\
             1  & \text{if}\,\,x^0 \leq q_0-\frac{\delta'}{2}.
    \end{cases}
\ee

We are now ready to define the source. Emphasizing the dependence on $\lambda$, $v=(q,\xi)$, $\delta'$ and the initial values for ODEs of $Y(s)$ and $Z(s)$ at the point $s=0$ (see \eqref{geo_par}) that is governed by $\iota \in \mathcal T$, we write
    \begin{align}\label{def_f_qxi}
f^{+}_{\lambda,v,\iota,\delta'} = \zeta_{+,v,\delta'} \Box_{g}  (\zeta_{-,v,\delta'} \Re \,\mathcal U_\lambda) \in \CI^{\infty}_c(\Omega_\i),
    \end{align}
where $\Re$ denotes the real part of a complex number. We require the parameter $\delta'$ to be sufficiently small so that 
$$  \supp f^{+}_{\lambda,v,\iota,\delta'} \subset \Omega_\i.$$
Note that since $(\Omega_\i,g|_{\Omega_\i})$ is assumed to be known (see the hypothesis of Theorem~\ref{t1}), $f_{\lambda,v,\iota,\delta'}^+$ will also be known. As $\zeta_{-,v,\delta'} = 1$ on the support of $1-\zeta_{+,v,\delta'}$, it holds that
$$
\Box_{g} (\zeta_{-,v,\delta'} \Re\, \mathcal U_\lambda) - f^+_{\lambda,v,\iota,\delta'} = (1-\zeta_{+,v,\delta'})\Box_{g} \Re \mathcal U_\lambda.
$$ 
Now applying (iii) in Section~\ref{formalgaussian} with $k=\frac{n}{2}+2$ and the fact that $N=\lceil\frac{3n}{2}\rceil+8$ implies the estimate
$$
\|\Box_{g}(\zeta_{-,v,\delta'} \Re\, \mathcal U_\lambda) - f^+_{\lambda,v,\iota,\delta'} \|_{H^k((-T,T)\times M_0)} \lesssim |\lambda|^{-\frac{n+1}{2}}|\lambda|^{-1}.
$$ 

We write $u^+_{\lambda,v,\iota,\delta'} = u$ where $u$ is the solution of the linear wave equation (\ref{linear_eq}) with the source $f = f^+_{\lambda,v,\iota,\delta'}$. By combining the above estimate with the usual energy estimate for the wave equation and the Sobolev embedding of $\CI^1((-T,T)\times M_0)$ in $H^{k}((-T,T)\times M_0)$ with $k=\frac{n}{2}+2$, we obtain
    \begin{align}\label{mathcalU_estimate}
\|\zeta_{-,v,\delta'} \Re\, \mathcal U_\lambda - u^+_{\lambda,v,\iota,\delta'} \|_{\CI^1((-T,T)\times M_0)} \lesssim |\lambda|^{-\frac{n+1}{2}}|\lambda|^{-1}.
    \end{align}
Observe that while the Gaussian beam $\mathcal U_\lambda$ is supported near the geodesic $\gamma$, the function $u^+_{\lambda,v,\iota,\delta'}$ is not supported near this geodesic anymore, but very small away from a tubular neighborhood of the geodesic.

We will also need a test function corresponding to $v=(q,\xi) \in L^+\Omega_\o$ whose construction differs from that of $f^+_{\lambda,v,\iota,\delta'}$ above only to the extent that the roles of $\zeta_{+,v,\delta'}$ and $\zeta_{-,v,\delta'}$ are reversed in (\ref{def_f_qxi}). That is, we define 
    \begin{align}\label{def_f_qxi_p}
f_{\lambda,v,\iota,\delta'}^- = \zeta_{-,v,\delta'} \Box_{g} (\zeta_{+,v,\delta'}\Re\, \mathcal U_\lambda) \in \CI^{\infty}_c(\Omega_\o).
    \end{align}
Since $(\Omega_\o,g|_{\Omega_\o})$ is assumed to be known, $f_{\lambda,v,\delta'}^-$ will also be known, and the analogue of (\ref{mathcalU_estimate}) reads
    \begin{align}\label{mathcalU_estimate_p}
\|\zeta_{+,v,\delta'} \Re\, \mathcal U_\lambda - u^-_{\lambda,v,\iota,\delta'} \|_{\CI^1((-T,T)\times M_0)} \lesssim |\lambda|^{-\frac{n+1}{2}}|\lambda|^{-1},
    \end{align}
where $u^-_{\lambda,v,\iota,\delta'}=u$ is now defined as the solution to the linear wave equation with reversed causality
\bel{eq_back}
\begin{aligned}
\begin{cases}
\Box_{g} u = f, 
&\text{on $M$},
\\
u=0& \text{on $M\setminus J^-(\supp f)$},
\end{cases}
    \end{aligned}
\ee 
with the source $f=f_{\lambda,v,\iota,\delta'}^-$.

Finally and before closing the section, we record the following estimate for the compactly supported sources $f^{\pm}_{\lambda,v,\iota,\delta'}$ that follows from the definitions \eqref{def_f_qxi}, \eqref{def_f_qxi_p} and property (iii) in the definition of Gaussian beams:
\bel{source_est} 
\|f^{\pm}_{\lambda,v,\iota,\delta'}\|_{L^2(M)} \lesssim |\lambda|^{1-\frac{n}{2}}. 
\ee

\section{Reduction from the source-to-solution map to the three-to-one scattering relation}
 \label{linearization_section_gauss}

We begin by considering $$v_{0}=(q_{0},\xi_{0}) \in L^+\Omega_\o\quad \text{and}\quad v_1=(q_1,\xi_1) \in L^+\Omega_\i$$ and require that the null geodesics $\gamma_{v_j}$, $j=0,1,$ are distinct. Here, $L^{\pm}\Omega_{\o}$ and $L^{\pm}\Omega_{\i}$ denote the bundle of future and past pointing light-like vectors on $\Omega_\o$ and $\Omega_\i$ respectively. 

As mentioned above, we impose that $\gamma_{v_0}$ and $\gamma_{v_1}$ are not reparametrizations of the same curve. This condition can always be checked via the map $\mathscr L$ or $\mathscr N$. To sketch this argument, we note that based on a simple first order linearization of the source-to-solution map, that is $\p_{\varepsilon} \mathscr N(\varepsilon f)|_{\varepsilon=0}$ or $\p_{\varepsilon} \mathscr L(\varepsilon f)|_{\varepsilon=0}$, we can obtain the source-to-solution map $L^{\textrm{lin}}_g$ associated to the linearized operator $\Box_{g}$ with sources in $\Omega_\i$ and receivers in $\Omega_\o$. To be precise, $L^{\textrm{lin}}_g: L^2(\Omega_\i) \to H^1(\Omega_\o)$ is defined through the mapping
$$ L^{\textrm{lin}}_g f = u|_{\Omega_\o}\quad \forall\, f \in L^2(\Omega_\i),$$
where $u$ is the unique solution to \eqref{linear_eq} subject to the source $f$.

Then, for example, based on the main result of \cite{SY} we can determine the scattering relation, $\Lambda_g$, for sources in $\Omega_\i$ and receivers in $\Omega_\o$, that is to say, the source-to-solution map $\mathscr L$ or $\mathscr N$ uniquely determine 
\bel{scattering_def} 
\Lambda_g(v)= \{(\gamma_v(s),c\,\dot{\gamma}_v(s))\,:c\in \R\setminus \{0\},\,s>0,\quad \gamma_v(s)\in \Omega_\o\}, \quad \forall \,v\in L^+\Omega_\i.
\ee
Using this scattering map, it is possible to determine if the two null geodesics $\gamma_{v_0}$ and $\gamma_{v_1}$ above are distinct or not. Indeed, to remove the possibility of identical null geodesics, we must have
\bel{scat_cond} v_0 \notin \Lambda_g(v_1).\ee

Having fixed $v_0\in L^+\Omega_\o$, $v_1\in L^+\Omega_\i$ subject to the requirement \eqref{scat_cond}, we proceed to define the test set $\Sigma_{v_0,v_1}$, as the set of all tuplets given by 
\bel{Sigma_def}
\begin{aligned}
\Sigma_{v_0,v_1}=\{(v_0,\kappa_0,\iota_{0},\ldots, v_3,\kappa_3, \iota_{3})\,:\, &v_2, v_3 \in L^+\Omega_\i, \quad \kappa_j \in \R\setminus \{0\},\\ 
    & \iota_j \in \mathcal T,\quad j=0,1,2,3,\}
\end{aligned}
\ee
where we recall that $\mathcal T$ is defined by \eqref{Y_init}. Note that $v_0$ and $v_1$ are a priori fixed and their inclusion in the tuplets $\sigma \in \Sigma_{v_0,v_1}$ is purely for aesthetic reasons. 

Given any small $\delta'>0$ and $\sigma=(v_0,\kappa_0,\iota_{0},\ldots, v_3,\kappa_3, \iota_{3}) \in \Sigma_{v_0,v_1}$, we consider the null geodesics $\gamma_{v_{j}}$, $j=0,1,2,3$, passing through $q_{j}$ in the directions $\xi_{j}$ and parametrization as in \eqref{affine}. We also denote by $y^{(j)}$ the Fermi coordinates near $\gamma_{v_j}$ given by Lemma~\ref{fermi} and subsequently, following the notation \eqref{gauss_par}, we construct for each $\lambda>0$, the Gaussian beams $\mathcal U_{\kappa_j\lambda}^{(j)}=\mathcal U_{\kappa_j\lambda,v_j,\iota_j,\delta'}$ of order 
\begin{equation}\label{N_choice}
N=\lceil\frac{3n}{2}\rceil+8,
\end{equation}
and the form
\bel{gaussianf}   
\mathcal U^{(j)}_{\kappa_j\lambda}(x)= 
     \begin{cases}
       e^{i\kappa_j\lambda\phi^{(j)}(x)}A^{(j)}_{\kappa_j\lambda}(x) &\quad\text{if}\quad \kappa_j>0,\\
       e^{i\kappa_j\lambda\bar{\phi}^{(j)}(x)}\bar{A}^{(j)}_{\kappa_j\lambda}(x)&\quad\text{if}\quad \kappa_j<0.
     \end{cases}
\ee
We recall that the functions $\phi^{(j)}, A^{(j)}_{\kappa_j\lambda}$ are exactly as in Section~\ref{formalgaussian} with a support $\delta'$ around the null geodesic $\gamma_{v_j}$ (see \eqref{phase-amplitude}) and the initial conditions for all ODEs assigned at the points $q^{(j)}$. As discussed in Section~\ref{sourcegauss}, we set the initial values for the phase terms $\phi^{(j)}_k$ with $k=3,\ldots,N$ and all $a^{(j)}_{k,l}$ with $k,l=0,1,\ldots,N$ (and not both simultaneously zero), to be zero at the point $q_{j}$. Finally and to complete the construction of the Gaussian beams, we set
$$ (Y^{(j)}(0),Z^{(j)}(0)) = \iota_j \in \mathcal T.$$

\subsection{Reduction from the source-to-solution map $\mathscr L$ to the three-to-one scattering relation}

Let $v_0 \in L^+\Omega_\o$ and $v_1\in L^+\Omega_\i$ subject to \eqref{scat_cond}. We begin by considering an arbitrary element $\sigma=(v_0,\kappa_0,\iota_0,\ldots,v_3,\kappa_3,\iota_3) \in \Sigma_{v_0,v_1}$ and also an arbitrary function $f \in C^{\infty}_c(\Omega_\i)$. Let the source terms $f^+_{\kappa_j\lambda,v_{j},\iota_j,\delta'}$ for $j=1,2,3$ and the test source $f_{\kappa_0\lambda,v_0,\iota_0,\delta'}^-$ be defined as in Section~\ref{sourcegauss} and define for each small vector $\varepsilon=(\varepsilon_1,\varepsilon_2,\varepsilon_3,\ldots,\varepsilon_m) \in \R^m$, the source $F^{\textrm{semi}}_{\varepsilon,\lambda,\sigma,\delta',f}$ given by the equation 
\bel{f_family_final_semi}
F^{\textrm{semi}}_{\varepsilon,\lambda,\sigma,\delta',f}= \begin{cases}
\varepsilon_1\,f^+_{\kappa_1\lambda,v_1,\iota_1,\delta'}+\varepsilon_2\,f^+_{\kappa_2\lambda,v_2,\iota_2,\delta'}+\varepsilon_3\,f^+_{\kappa_3\lambda,v_3,\iota_3,\delta'}  & \text{if}\,\, m=3,\\
\varepsilon_1\,f^+_{\kappa_1\lambda,v_1,\iota_1,\delta'}+\varepsilon_2\,f^+_{\kappa_2\lambda,v_2,\iota_2,\delta'}+\varepsilon_3\,f^+_{\kappa_3\lambda,v_3,\iota_3,\delta'}+\sum_{j=4}^m \varepsilon_j\,f,  & \text{if}\,\, m\geq 4,\\
\end{cases}
\ee

For a fixed $\lambda>0$ and small enough $\varepsilon_j$, $j=1,2,3,\ldots,m$, 
it holds that $F^{\textrm{semi}}_{\varepsilon,\lambda,\sigma,\delta',f} \in \mathscr C_{\Omega_\i}$. We let $u^{\textrm{semi}}_{\varepsilon,\lambda,\sigma,\delta',f}$ denote the unique small solution to (\ref{pf0}) subject to this source term.  Note that in particular, there holds:
$$ \p_{\varepsilon_j} u^{\textrm{semi}}_{\varepsilon,\lambda,\sigma,\delta'}\bigg|_{\varepsilon = 0}=\begin{cases}
          u^+_{\kappa_j\lambda,v_{j},\iota_j,\delta'} & \text{if}\,\, j=1,2,3,\\
             u_f  & \text{if  $m \geq 4$ and $j=4,\ldots,m$.}
     \end{cases}
  $$
where $u_f$ is the unique solution to \eqref{linear_eq} subject to the source $f$ and, as discussed in Section~\ref{sourcegauss}, $u^+_{\kappa_j\lambda,v_{j},\iota_j,\delta'}$ is the unique solution to equation \eqref{linear_eq} subject to the source $f_{\kappa_j\lambda,v_{j},\iota_j,\delta'}^+$ and is close, in the sense of the estimate (\ref{mathcalU_estimate}), to the real part of the Gaussian beam solutions of forms (\ref{gaussianf}) supported in a $\delta'$-neighborhood of the light ray $\gamma_{v_j}$ with $j=1,2,3$.

Finally, we define for each small $\delta'>0$, $\sigma \in \Sigma_{v_0,v_1}$ and $f \in C^{\infty}_c(\Omega_\i)$, the analytical data $\mathscr D^{\textrm{semi}}_{\sigma,\delta',f}$ corresponding to the semi-linear equation \eqref{pf0} by the expression
\bel{semi_anal_da}
\mathscr D^{\textrm{semi}}_{\sigma,\delta',f}= \lim_{\lambda\to+\infty}\,\lambda^{\frac{n+1}{2}} \int_{\Omega_\o}f_{\kappa_0\lambda,v_0,\iota_0,\delta'}^-\, \frac{\p^m}{\p\varepsilon_1\ldots\p\varepsilon_m} u^{\textrm{semi}}_{\varepsilon,\lambda,\sigma,\delta',f} \bigg |_{\varepsilon=0}\,dV_g
\ee 
where we recall that $u^{\textrm{semi}}_{\varepsilon,\lambda,\sigma,\delta',f}$ is the unique solution to \eqref{pf0} subject to the source $F^{\textrm{semi}}_{\varepsilon,\lambda,\sigma,\delta',f}$ given by \eqref{f_family_final_semi}. Let us emphasize that the knowledge of the source-to-solution map, $\mathscr L$, determines the analytical data $\mathscr D^{\textrm{semi}}_{\sigma,\delta',f}$. We have the following theorem.

\begin{theorem}
\label{thm_anal_data_1}
Let $v_0\in L^+\Omega_\o$ and $v_1\in L^+\Omega_\i$ be such that \eqref{scat_cond} holds. The following statements hold:
\begin{itemize}
\item[(i)]{If $\mathscr D^{\textrm{semi}}_{\sigma,\delta'_j,f} \neq 0$ for some $\sigma \in\Sigma_{v_0,v_1}$ and $f \in C^{\infty}_c(\Omega_\i)$ and a sequence $\{\delta'_j\}_{j=1}^{\infty}$ converging to zero, then there exists an intersection point $y \in \overleftarrow{\;\gamma_{v_0}} \cap \bigcap_{j=1}^3  \overrightarrow{\gamma_{v_{j}}}$.}
\item[(ii)]{Let $v_2,v_3 \in L^+\Omega_\i$. Assume that $\gamma_{v_{j}}$, $j=0,1,2,3$, are distinct
and there exists a point $y \in \overleftarrow{\;\gamma_{v_{0}}} \cap \bigcap_{j=1}^3  \overrightarrow{\gamma_{v_{j}}}.$ Moreover, assume that  
$y=\gamma_{v_{0}}(s_0)$ with $s_0\in (-\rho(v_0),0]$ and $y=\gamma_{v_{j}}(s_j)$  for all $j=1,2,3$,
with $s_j\in [0,\rho(v_j))$. Denote $\xi_j=\dot \gamma_{v_{j}}(s_j)$  for $j=0,1,2,3$
and assume that $\xi_0 \in \linspan(\xi_1, \xi_2, \xi_3)$. Then, there exists $f \in C^{\infty}_c(\Omega_\i)$, $\kappa_j \in \R\setminus \{0\}$ and $\iota_j \in \mathcal T$, $j=0,1,2,3$ such that $\mathscr D^{\textrm{semi}}_{\sigma,\delta',f}\neq 0$ for all $\delta'$ sufficiently small.}
\end{itemize}
\end{theorem}
We will prove Theorem~\ref{thm_anal_data_1} in Section~\ref{subsec_semi}. Observe that as an immediate corollary of Theorem~\ref{thm_anal_data_1} it follows that the relation
   \begin{align*}
\rel_{\text{semi-lin}} = \{&(v_{0},v_{1}, v_{2},v_{3})\in  L^+ \Omega_\o \times (L^+ \Omega_\i)^3: 
\text{$\gamma_{v_j}$'s are pair-wise not identical,}
\\& \text{there are $f \in C^{\infty}_c(\Omega_\i)$, $\kappa_j \in \R \setminus \{0\}$ \quad \text{and}\quad $\iota_j \in \mathcal T$, $j=0,1,2,3$, }
\\&\text{s.t for all small $\delta'>0$,
$\mathscr D^{\textrm{semi}}_{\sigma,\delta',f} \ne 0$ where $\sigma = (v_{0},\kappa_0,\iota_0,\dots, v_{3},\kappa_3,\iota_3)$}
\},
    \end{align*}
is a three-to-one scattering relation, that is to say, it satisfies (R1) and (R2) in Definition~\ref{def: good relation}. Therefore, since the source-to-solution map $\mathscr L$ determines $\rel_{\text{semi-lin}}$, the first part of Theorem~\ref{t1}, that is the recovery of the topological, differential and conformal structure of $\mathbb D$ from the source-to-solution map $\mathscr L$, follows immediately from combining Theorem~\ref{thm_anal_data_1} and Theorem~\ref{t2}.

\subsection{Reduction from the source-to-solution map $\mathscr N$ to the three-to-one scattering relation}

Analogously to the previous section, we begin by considering an arbitrary element $\sigma \in \Sigma_{v_0,v_1}$. Next, we define the three parameter family of sources $F^{\textrm{quasi}}_{\varepsilon,\lambda,\sigma,\delta'}$ with $\varepsilon = (\varepsilon_1, \varepsilon_2, \varepsilon_3)$
given by the equation 
\bel{f_family_final_quasi}
F^{\textrm{quasi}}_{\varepsilon,\lambda,\sigma,\delta'}= \varepsilon_1\,f^+_{\kappa_1\lambda,v_1,\iota_1,\delta'}+\varepsilon_2\,f^+_{\kappa_2\lambda,v_2,\iota_2,\delta'}+\varepsilon_3\,f^+_{\kappa_3\lambda,v_3,\iota_3,\delta'}.
\ee

For a fixed $\lambda>0$ and small enough $\varepsilon_j$, $j=1,2,3$, 
it holds that $F^{\textrm{quasi}}_{\varepsilon,\lambda,\sigma,\delta'} \in \mathscr C_{\Omega_\i}$. We let $u^{\textrm{quasi}}_{\varepsilon,\lambda,\sigma,\delta'}$ denote the unique small solution to \eqref{pf} subject to this source term. Note that:
$$ \p_{\varepsilon_j} u^{\textrm{quasi}}_{\varepsilon,\lambda,\sigma,\delta'}\bigg|_{\varepsilon = 0}=u^+_{\kappa_j\lambda,v_{j},\iota_j,\delta'}\quad \text{for $j=1,2,3$},$$
where we recall from Section~\ref{sourcegauss} that $u^+_{\kappa_j\lambda,v_{j},\iota_j,\delta'}$ is the unique solution to equation \eqref{linear_eq} subject to the source $f_{\kappa_j\lambda,v_{j},\iota_j,\delta'}^+$ and is close, in the sense of the estimate (\ref{mathcalU_estimate}), to the real part of the Gaussian beam solutions of forms (\ref{gaussianf}) supported in a $\delta'$-neighborhood of the light ray $\gamma_{v_j}$ with $j=1,2,3$.

Finally, we define for each small $\delta'>0$ and $\sigma \in \Sigma_{v_0,v_1}$, the analytical data $\mathscr D^{\textrm{quasi}}_{\sigma,\delta'}$ by the expression
\bel{quasi_anal_da}
\mathscr D^{\textrm{quasi}}_{\sigma,\delta'}= \lim_{\lambda\to+\infty}\,\lambda^{\frac{n-3}{2}} \int_{\Omega_\o}f_{\kappa_0\lambda,v_0,\iota_0,\delta'}^-\, \frac{\p^3}{\p\varepsilon_1\p\varepsilon_2\p\varepsilon_3} u^{\textrm{quasi}}_{\varepsilon,\lambda,\sigma,\delta'} \bigg |_{\varepsilon=0}\,dV_g,
\ee 
where we recall that $u^{\textrm{quasi}}_{\varepsilon,\lambda,\sigma,\delta'}$ is the unique solution to \eqref{pf} subject to the source $F^{\textrm{quasi}}_{\varepsilon,\lambda,v,\delta'}$ given by \eqref{f_family_final_quasi}. Let us emphasize that the knowledge of the source-to-solution map, $\mathscr N$, determines the analytical data $\mathscr D^{\textrm{quasi}}_{\sigma,\delta'}$. We have the following theorem.

\begin{theorem}
\label{thm_anal_data_2}
Let $v_0\in L^+\Omega_\o$ and $v_1\in L^+\Omega_\i$ be such that \eqref{scat_cond} holds. The following statements hold:
\begin{itemize}
\item[(i)]{If $\mathscr D^{\textrm{quasi}}_{\sigma,\delta'_j} \neq 0$ for some $\sigma \in\Sigma_{v_0,v_1}$ and a sequence $\{\delta'_j\}_{j=1}^{\infty}$ converging to zero, then there exists an intersection point $y \in \overleftarrow{\;\gamma_{v_0}} \cap \bigcap_{j=1}^3  \overrightarrow{\gamma_{v_{j}}}$.}
\item[(ii)]{Let $v_2,v_3 \in L^+\Omega_\i$. Assume that $\gamma_{v_{j}}$, $j=0,1,2,3$, are distinct
and there exists a point $y \in \overleftarrow{\;\gamma_{v_{0}}} \cap \bigcap_{j=1}^3  \overrightarrow{\gamma_{v_{j}}}.$ Moreover, assume that  
$y=\gamma_{v_{0}}(s_0)$ with $s_0\in (-\rho(v_0),0]$ and $y=\gamma_{v_{j}}(s_j)$  for all $j=1,2,3$,
with $s_j\in [0,\rho(v_j))$. Denote $\xi_j=\dot \gamma_{v_{j}}(s_j)$  for $j=0,1,2,3$
and assume that $\xi_0 \in \linspan(\xi_1, \xi_2, \xi_3)$. Then, there exists $\kappa_j \in \R\setminus \{0\}$ and $\iota_j \in \mathcal T$, $j=0,1,2,3$ such that $\mathscr D^{\textrm{quasi}}_{\sigma,\delta'}\neq 0$ for all $\delta'$ sufficiently small.}
\end{itemize}
\end{theorem}
We prove Theorem~\ref{thm_anal_data_2} in Section~\ref{subsec_quasi}. Observe that as an immediate corollary of Theorem~\ref{thm_anal_data_2} it follows that the relation
   \begin{align*}
\rel_{\text{quasi-lin}} = \{&(v_{0},v_{1}, v_{2},v_{3})\in  L^+ \Omega_\o \times (L^+ \Omega_\i)^3: 
\text{$\gamma_{v_j}$'s are pair-wise not identical,}
\\& \text{there are $\kappa_j \in \R \setminus \{0\}$ \quad \text{and}\quad $\iota_j \in \mathcal T$, $j=0,1,2,3$, }
\\&\text{s.t for all small $\delta'>0$,
$\mathscr D^{\textrm{quasi}}_{\sigma,\delta'} \ne 0$ where $\sigma = (v_{0},\kappa_0,\iota_0,\dots, v_{3},\kappa_3,\iota_3)$}
\},
    \end{align*}
is a three-to-one scattering relation, that is to say, it satisfies (R1) and (R2) in Definition~\ref{def: good relation}. Therefore, since the source-to-solution map $\mathscr N$ determines $\rel_{\text{quasi}}$, the first part of Theorem~\ref{t1}, that is the recovery of the topological, differential and differential structure of $\mathbb D$ from the source-to-solution map $\mathscr N$, follows immediately from combining Theorem~\ref{thm_anal_data_2} and Theorem~\ref{t2}.

\subsection{Proof of Theorem~\ref{thm_anal_data_1}}
\label{subsec_semi}

Note that by expression \eqref{sotosol_data_semi} in Section~\ref{linearization_section}, the source-to-solution map $\mathscr L$ determines the knowledge of the expression
$$\mathcal I_{\lambda,\sigma,\delta',f}=\int_M u^{-}_{\kappa_0\lambda,v_0,\iota_0,\delta'} u^{+}_{\kappa_1\lambda,v_1,\iota_1,\delta'}u^{+}_{\kappa_2\lambda,v_2,\iota_2,\delta'} u^{+}_{\kappa_3\lambda,v_3,\iota_3,\delta'}\,u_f^{m-3}dV_g,$$
where we recall that the notations $u^{\pm}_{\kappa_j\lambda,v_j,\iota_j,\delta'}$ are as defined in Section~\ref{linearization_section_gauss}. Recall also that the function $u^{+}_{\kappa_j\lambda,v_j,\iota_j,\delta'}$, $j=1,2,3$ (respectively $u^-_{\kappa_0\lambda,v_0,\iota_0,\delta'}$) is close in the sense of \eqref{mathcalU_estimate} (respectively \eqref{mathcalU_estimate_p}) to the Gaussian beams $\mathcal U_{\kappa_j\lambda}^{(j)}=\mathcal U_{\kappa_j\lambda,v_j,\iota_j,\delta'}$ (respectively $\mathcal U_{\kappa_0\lambda}^{(0)}=\mathcal U_{\kappa_0\lambda,v_0,\iota_0,\delta'}$). Finally, the function $u_f$ is the unique solution to \eqref{linear_eq} subject to the source $f \in C^{\infty}_c(\Omega_\i)$. Note also that by \eqref{sotosol_data_semi}, there holds:
$$ \mathscr D^{\textrm{semi}}_{\sigma,\delta',f} = -m! \lim_{\lambda \to \infty}\lambda^{\frac{n+1}{2}} \mathcal I_{\lambda,\sigma,\delta',f}.$$

Next, we observe from the definitions \eqref{phase-amplitude} that the Gaussian beams $\mathcal U^{(j)}_{\kappa_j\lambda}$, $j=0,1,2,3$ satisfy the uniform bounds
\bel{gauss_uni} \|\mathcal U_{\kappa_j\lambda}^{(j)}\|_{L^{\infty}((-T,T)\times M_0)}\leq C_j,\ee
for some constants $C_j$ independent of $\lambda$. Together with the estimates \eqref{mathcalU_estimate}--\eqref{mathcalU_estimate_p}, it follows that 
$$\lambda^{\frac{n+1}{2}}\mathcal I_{\lambda,\sigma,\delta',f}=\lambda^{\frac{n+1}{2}}\int_{M}\zeta_{+,v_0,\delta'}\Re\, \mathcal U_{\kappa_0\lambda}^{(0)}\left(\prod_{j=1}^3\zeta_{-,v_j,\delta'}\Re\, \mathcal U_{\kappa_j\lambda}^{(j)}\right)u_f^{m-3}\,dV_{g}+\mathcal O(\lambda^{-1}),$$
which implies that
\bel{expf_1}
\mathscr D^{\textrm{semi}}_{\sigma,\delta',f}=-m! \lim_{\lambda\to \infty} \lambda^{\frac{n+1}{2}}\int_{M}\zeta_{+,v_0,\delta'}\Re\, \mathcal U_{\kappa_0\lambda}^{(0)}\left(\prod_{j=1}^3\zeta_{-,v_j,\delta'}\Re\, \mathcal U_{\kappa_j\lambda}^{(j)}\right)u_f^{m-3}\,dV_{g}.
\ee
Note that given $\delta'>0$ sufficiently small, the latter integrand is supported on a compact subset of $\mathbb D_e$ (see \eqref{optimal_reg}).

Before proving Theorem~\ref{thm_anal_data_1}, we state the following lemma.

\begin{lemma}
\label{non-vanishing_lem_wave}
Given any point $p$ in $M$ that lies on a null geodesic $\overrightarrow{\gamma_v}$ with $v \in L^+\Omega_\i$, there exists a real valued source $f \in C^{\infty}_c(\Omega_\i)$, such that the solution $u_f$ to \eqref{linear_eq} with source $f$ satisfies 
$$ u_f(p)\neq 0.$$ 
\end{lemma}

\begin{proof}
Let $p=\gamma_v(s)$ for some $v=(q,\xi) \in L^+\Omega_\i$ and some $s\geq 0$. Let $y=(y^0,y^1,\ldots,y^n)$ denote the Fermi coordinate system in a tubular neighborhood of $\gamma$ and note that $p=(s,0)$. We consider for each $\lambda>0$, $\iota \in \mathcal T$ and $\delta'$ small, the Gaussian beam $\mathcal U_\lambda$ of order $N=\lceil\frac{3n}{2}\rceil+8$, near the geodesic $\gamma$ that is fixed by the choices $\lambda$, $\delta'$ and $\iota$ (see Section~\ref{sourcegauss}). Next, consider the source $f=f^+_{\lambda,v,\iota,\delta'}$ and recall that the solution $u=u^+_{\lambda,v,\iota,\delta'}$ to equation \eqref{linear_eq} with source term $f$ is asymptotically close to $\mathcal U_{\lambda}$ in the sense of \eqref{mathcalU_estimate}. Together with the explicit expressions \eqref{phase-amplitude}, we deduce that
$$u(p)=u(s,0)= a_{0,0}(s)+\mathcal O(\lambda^{-1}).$$ 
Recalling the expression for the principal amplitude term $a_{0,0}$ (see \eqref{explicit_terms}), we deduce that there exists $\iota\in \mathcal T$ such that $a_{0,0}(s)=1$. The claim follows trivially with this choice of $\iota$ and $\lambda$ sufficiently large.
\end{proof}

\subsubsection{Proof of statement (i) in Theorem~\ref{thm_anal_data_1}} 
We assume that $\mathscr D^{\textrm{semi}}_{\sigma,\delta',f} \neq 0$ for some $\sigma=(v_0,\kappa_0,\iota_0,\ldots,v_3,\kappa_3,\iota_3)\in \Sigma_{v_0,v_1}$ and $f \in C^{\infty}_c(\Omega_\i)$ and a family of $\delta'$'s converging to zero. First, observe that the corresponding null geodesics $\gamma_{v_j}$ with $j=0,1,2,3$ must simultaneously intersect at least once on $\mathbb D_e$, since otherwise the support of the amplitude functions $A^{(j)}_{\kappa_j\lambda}$ in the expression \eqref{expf_1} become disjoint sets for all sufficiently small $\delta'$. Subsequently, the integrand in \eqref{expf_1} vanishes independent of the parameter $\lambda$ implying that $\mathscr D^{\textrm{semi}}_{\sigma,\delta',f}=0$ for all $\delta'$ small. Let $\mathcal A=\{y_1,\ldots,y_N\}$ denote the set of intersection points of the four null geodesics $\gamma_{v_j}$, $j=0,1,2,3$ on $\mathbb D_e$. In terms of the set $\mathcal A$, we observe that given $\delta'$ sufficiently small, the expression \eqref{expf_1} reduces as follows:
\bel{exp_f_2}\mathscr D^{\textrm{semi}}_{\sigma,\delta',f}=-m! \lim_{\lambda\to \infty} \lambda^{\frac{n+1}{2}}\sum_{\ell=1}^N\int_{U_\ell}\zeta_{+,v_0,\delta'}\Re\, \mathcal U_{\kappa_0\lambda}^{(0)}\left(\prod_{j=1}^3\zeta_{-,v_j,\delta'}\Re\, \mathcal U_{\kappa_1\lambda}^{(j)}\right)u_f\,dV_{g},\ee
where $U_\ell$, $\ell=1,\ldots,N,$ is a small open neighborhood of the point $y_\ell$ in $\mathbb D_e$ that depends on $\delta'$.

To complete the proof of statement (i), we need to show that there is a point $y \in \mathcal A$ that satisfies the more restrictive casual condition $y \in \overleftarrow{\;\gamma_{v_{0}}} \cap \bigcap_{j=1}^3  \overrightarrow{\gamma_{v_{j}}}$. It is straightforward to see that if $y_\ell \notin \overleftarrow{\;\gamma_{v_{0}}} \cap \bigcap_{j=1}^3  \overrightarrow{\gamma_{v_{j}}}$ for some $\ell=1,\ldots,N$, then 
$$\int_{U_\ell}\zeta_{+,v_0,\delta'}\Re\, \mathcal U_{\kappa_0\lambda}^{(0)}\left(\prod_{j=1}^3\zeta_{-,v_j,\delta'}\Re\, \mathcal U_{\kappa_1\lambda}^{(j)}\right)dV_{g}=0,$$
for all $\delta'$ sufficiently small. Indeed, this follows from the definitions of the cut-off functions $\zeta_{+,v_0,\delta'}$ and $\zeta_{-,v_j,\delta'}$, $j=1,2,3$, see \eqref{zeta_-}--\eqref{zeta_+}. Since $\mathscr D^{\textrm{semi}}_{\sigma,\delta',f} \neq 0$ for a sequence $\{\delta'_k\}_{k=1}^{\infty}$ converging to zero by the hypothesis (i) of the theorem, it follows that there must exists a point $y \in \mathcal A$ such that $y \in \overleftarrow{\;\gamma_{v_{0}}} \cap \bigcap_{j=1}^3  \overrightarrow{\gamma_{v_{j}}}$.
\subsubsection{Proof of statement (ii) in Theorem~\ref{thm_anal_data_1}} 
We are assuming here that $$(v_0,v_1,v_2,v_3) \in L^+\Omega_\o \times (L^+\Omega_\i)^3$$ satisfies the hypothesis of statement (ii) and want to prove that there exists a real valued function $f \in C^{\infty}_c(\Omega_\i)$, and $\kappa_j \in \R\setminus \{0\}$ and $\iota_j \in \mathcal T$ with $j=0,1,2,3$ such that given $$\sigma=(v_0,\kappa_0,\iota_0,\ldots,v_3,\kappa_3,\iota_3),$$ and all $\delta'>0$ sufficiently small, there holds:
$$ \mathscr D^{\textrm{semi}}_{\sigma,\delta',f} \neq 0.$$

Let us first emphasize that given the hypothesis of statement (ii), there is a unique point $y \in \overleftarrow{\;\gamma_{v_{0}}} \cap \bigcap_{j=1}^3  \overrightarrow{\gamma_{v_{j}}}$. To see this, we suppose for contrary that there is another distinct point $\tilde{y} \in \overleftarrow{\;\gamma_{v_{0}}} \cap \bigcap_{j=1}^3  \overrightarrow{\gamma_{v_{j}}}$. If $\tilde y \le y$, then there exists a broken path consisting of null geodesics that connects one of the points $\pi(v_j)$, $j=1,2,3,$ to the point $y$. Together with  \cite[Proposition 10.46]{O'Neill} we obtain that $\tau(\pi(v_j),y)>0$ yielding a contradiction since $y=\gamma_{v_j}(s_j)$ with $s_j \in [0,\rho(v_j))$. In the alternative case that $y \le \tilde{y}$, there exists a broken path consisting of null geodesics that connects the point $y$ to the point $\pi(v_0)$. Together with  \cite[Proposition 10.46]{O'Neill} we obtain that $\tau(y,\pi(v_0))>0$ yielding a contradiction since $y=\gamma_{v_0}(s_0)$ with $s_0 \in (-\rho(v_0),0]$.      

Next, we observe that given $\delta'$ sufficiently small together with the fact that $y \in \overleftarrow{\;\gamma_{v_{0}}} \cap \bigcap_{j=1}^3  \overrightarrow{\gamma_{v_{j}}}$, the expression \eqref{expf_1} reduces as follows:
\bel{expf}   
\mathscr D^{\textrm{semi}}_{\sigma,\delta',f}=-m! \lim_{\lambda\to \infty} \lambda^{\frac{n+1}{2}}\int_{M}\left(\prod_{j=0}^3\Re\, \mathcal U_{\kappa_1\lambda}^{(j)}\right)\,u_f^{m-3}\,dV_{g}.
\ee  
We will use the method of stationary phase to analyze the product of the four Gaussian beams in (\ref{expf}). Let us begin by considering the unique point $y \in \overleftarrow{\;\gamma_{v_{0}}} \cap \bigcap_{j=1}^3  \overrightarrow{\gamma_{v_{j}}}$. We choose the real valued function $f \in C^{\infty}_c(\Omega_\i)$ such that $u_f(y) \neq 0$. This is possible thanks to Lemma~\ref{non-vanishing_lem_wave}. Next, we choose the non-zero constants $\kappa_0,\ldots,\kappa_3$, so that 
\bel{linear dependence f} \sum_{j=0}^3 \kappa_j \dot{\gamma}_{v_j}(s_j)=0,\ee
where $\gamma_{v_j}(s_j)=y$. Recall that these constants exist by our assumption that the tangents to $\gamma_{v_j}$ are linearly dependent at $y$. That the constants $\kappa_j$, $j=0,1,2,3$ are all non-zero follows from the fact that any three pair-wise linearly independent null vectors are linearly independent.

We consider the four families of Gaussian beams along the geodesics $\gamma_{v_j}$, $j=0,1,2,3$, as in (\ref{gaussianf}). We choose the initial datum $\iota_j \in \mathcal T$ for the initial values $(Y^{(j)}(0),Z^{(j)}(0))$ governing ODEs of the matrices $Y^{(j)}(s)$ and $Z^{(j)}(s)$, so that  
\bel{YZ_init} Y^{(j)}(s_j)=I \quad \text{and}\quad Z^{(j)}(s_j) = i\, I,\ee
where we recall that $\gamma_{v_j}(s_j)=y$. Note in particular that given this choice of $\iota_j \in \mathcal T$, $j=0,1,2,3$, there holds:
\begin{equation}\label{Y_cond} a_{0,0}^{(j)}(y)= 1\quad \text{for $j=0,1,2,3$}.\end{equation}
In the remainder of the proof, we show that given the function $f$ and the tuplet $\sigma=(v_0,\kappa_0,\iota_0,\ldots,v_3,\kappa_3,\iota_3)$ constructed as above and all $\delta'>0$ sufficiently small, there holds $ \mathscr D^{\textrm{semi}}_{\sigma,\delta',f} \neq 0.$
 
It can be easily verified from the choice of the parametrization \eqref{affine}, Lemma~\ref{fermi} and the expression for the phase function given by \eqref{explicit_terms}, that
\bel{phase_tangent}
\dot{\gamma}_{v_j}(s)=\nabla^g\phi^{(j)}|_{\gamma_{v_j}(s)}=\nabla^g\bar{\phi}^{(j)}|_{\gamma_{v_j}(s)}\quad \text{for $j=0,1,2,3$}.
\ee

Let us define
\bel{S}S(x):= \Phi^{(0)}(x)+ \Phi^{(1)}(x)+ \Phi^{(2)}(x)+ \Phi^{(3)}(x),\ee
where
\bel{phi_sign}   
\Phi^{(j)}(x) = 
     \begin{cases}
       \kappa_j\phi^{(j)}(x) &\quad\text{if}\quad \kappa_j>0,\\
       \kappa_j\bar{\phi}^{(j)}(x)&\quad\text{if}\quad \kappa_j<0.
     \end{cases}
\ee
We have the following lemma.
\begin{lemma}
\label{lem_5}
Suppose that $y=\overleftarrow{\;\gamma_{v_{0}}} \cap \bigcap_{j=1}^3  \overrightarrow{\gamma_{v_{j}}}$ and that \eqref{linear dependence f} holds. Let $S$ be defined by \eqref{S} and denote by $d$ an auxiliary Riemannian distance function on $M$. There holds:
\begin{itemize}
\item[(i)]{$S(y)=0$.}
\item[(ii)]{$\nabla^{g}S(y)=0$.}
\item[(iii)]{$\Im S(\tilde{y}) \geq a\, d(\tilde{y},y)^2$ for all points $\tilde{y}$ in a neighborhood of $y$. Here $a>0$ is a constant.}
\end{itemize}
\end{lemma}
We refer the reader to \cite[Lemma 5]{FO} for the proof of this lemma. 
\begin{lemma}
\label{stationary_lem}
Suppose that $y=\overleftarrow{\;\gamma_{v_{0}}} \cap \bigcap_{j=1}^3  \overrightarrow{\gamma_{v_{j}}}$ and that \eqref{linear dependence f} holds. Let $S$ be defined by \eqref{S} and let $F \in \mathcal C^1(M)$ be compactly supported in a sufficiently small neighborhood of the point $y$. There holds:
$$ \lim_{\lambda \to \infty}\lambda^{\frac{n+1}{2}} \int_M e^{i\lambda S(x)}F(x)\,dV_g = C_0\,F(y),$$
where $C_0 \in \C$ only depends on $(M,g)$ and $c_0=\Re\, C_0 \neq 0$.
\end{lemma}
\begin{proof}
We fix a coordinate system $(x^0,\ldots,x^n)$ in a small neighborhood about the point $y$, so that $y=(0,\ldots,0)$. By Lemma \ref{lem_5}, there holds: 
$$ S(x) = \sum_{j,k=0}^n Q_{jk}x^jx^k + R(x),$$
where $|R(x)|=\mathcal O(|x|^3)$ and the matrix $Q=(Q_{jk})_{j,k=0}^n$ has a positive definite imaginary part.
We assume that $F$ is supported in a sufficiently small neighborhood $U$ of the point $y$, so that
$$ \Im S(x) \geq \frac{1}{2}(\sum_{j,k=0}^n \Im Q_{jk}x^jx^k)\geq C|x|^2\quad \text{on $U$},$$
for some $C>0$ that depends on $(M,g)$. Next, we note that:
\begin{multline*}
\left|\int_{U} (F(x)-F(0)) e^{i\lambda S(x)}\,dV_g\right| \leq \int_{U} |F(x)-F(0)| e^{-C\lambda|x|^2}\left|\det g\right|^{\frac{1}{2}}\,dx\\
\leq \|F\|_{\mathcal C^1(M)} \int_U |x| e^{-C\lambda|x|^2}\left|\det g\right|^{\frac{1}{2}}\,dx\lesssim \lambda^{-\frac{n+1}{2}}\lambda^{-\frac{1}{2}}. 
\end{multline*}
Therefore, 
\begin{multline*}
\lim_{\lambda \to \infty}\lambda^{\frac{n+1}{2}} \int_M e^{i\lambda S(x)}F(x)\,dV_g= F(0)\left|\det g(0)\right|^{\frac{1}{2}}\lim_{\lambda\to\infty} \lambda^{\frac{n+1}{2}} \int_Ue^{i\lambda\sum_{j,k=0}^nQ_{jk}x^jx^k}\,dx\\
= C_0 F(y),
\end{multline*}
where we applied the method of stationary phase in the last step, see e.g. Theorem 7.7.5 in \cite{Ho1}.
\end{proof}

Let us now return to the expression \eqref{expf} and note that it reduces as follows:
\begin{multline}
\label{summation_index}
\lim_{\lambda\to\infty}\lambda^{\frac{n+1}{2}}\mathcal I_{\lambda,\sigma,\delta',f}=\lim_{\lambda\to\infty}\lambda^{\frac{n+1}{2}}\int_{M}u_f^{m-3}\Re \,\mathcal U_{\kappa_0\lambda}^{(0)}\,\Re\, \mathcal U_{\kappa_1\lambda}^{(1)}\, \Re\, \mathcal U_{\kappa_2\lambda}^{(2)}\,\Re \,\mathcal U_{\kappa_3\lambda}^{(3)}\,dV_{g}\\
= \lim_{\lambda\to\infty}2^{-4}\lambda^{\frac{n+1}{2}}\sum_{\ell_0,\ell_1,\ell_2,\ell_3 =1,2}\int_{M}u_f^{m-3}\vartheta_{\ell_0}(\mathcal U_{\kappa_0\lambda}^{(0)})\vartheta_{\ell_1}(\mathcal U_{\kappa_1\lambda}^{(1)})\vartheta_{\ell_2}(\mathcal U_{\kappa_2\lambda}^{(2)})\vartheta_{\ell_3}(\mathcal U_{\kappa_3\lambda}^{(3)})\,dV_{g},
\end{multline}
where $$\vartheta_{1}(z)= z\quad \text{and} \quad\vartheta_2(z) = \bar{z},\quad \text{for all $z \in \C$}.$$
\begin{lemma}
\label{real_part_lem}
Given $\sigma \in \Sigma_{v_0,v_1}$, with $y=\overleftarrow{\;\gamma_{v_{0}}} \cap \bigcap_{j=1}^3  \overrightarrow{\gamma_{v_{j}}}$, the choice of $\kappa_0,\ldots,\kappa_3$ satisfying \eqref{linear dependence f} and the initial data $\iota_0,\ldots\iota_3\in \mathcal T$ satisfying \eqref{YZ_init}, there holds:
$$\mathscr D^{\textrm{semi}}_{\sigma,\delta',f}=\lim_{\lambda\to \infty}2^{-3}\lambda^{\frac{n+1}{2}} \Re\left(\int_{M}u_f^{m-3}\,\mathcal U_{\kappa_0\lambda}^{(0)}\mathcal U_{\kappa_1\lambda}^{(1)}\mathcal U_{\kappa_2\lambda}^{(2)}\mathcal U_{\kappa_3\lambda}^{(3)}\,dV_{g}\right).$$
\end{lemma}

\begin{proof}
Observe that the summation in expression \eqref{summation_index} contains 16 terms and we are claiming that only two terms here contribute in the limit as $\lambda$ approaches infinity, when $(\ell_0,\ell_1,\ell_2,\ell_3)\in\{(1,1,1,1),(2,2,2,2)\}$. To see that the other terms do not contribute, we note that
\bel{other_terms} \int_{M}u_f^{m-3}\prod_{j=0}^3\vartheta_{\ell_j}(\mathcal U_{\kappa_j\lambda}^{(j)})\,dV_{g}= \int_M u_f^{m-3}\,e^{i\lambda S_{\ell_0,\ell_1,\ell_2,\ell_3}}(\prod_{j=0}^3 \vartheta_{\ell_j}(B_{\kappa_j\lambda}^{(j)}))\,dV_g,\ee
where $S_{\ell_0\ell_1\ell_2\ell_3}= \sum_{j=0}^3\vartheta_{\ell_j}(\kappa_j\Phi^{(j)})$ and 
\bel{B_sign}   
B^{(j)}_{\kappa_j\lambda}(x) = 
     \begin{cases}
       A^{(j)}_{\kappa_j \lambda}(x) &\quad\text{if}\quad \kappa_j>0,\\
       \bar{A}^{(j)}_{\kappa_j\lambda}(x)&\quad\text{if}\quad \kappa_j<0.
     \end{cases}
\ee 

We observe that since $y$ is the only point of the intersection between the four null geodesics $\gamma_{v_j}$ in $\mathbb D_e \subset (-T,T)\times M_0$, the integral above is supported in a small neighborhood of the point $y$. It is easy to see verify that $$S_{1,1,1,1}=S\quad \text{and}\quad S_{2,2,2,2}=\bar{S},$$ where $S$ is as defined in \eqref{S}. Thus, by Lemma~\ref{lem_5} there holds:
$$ \nabla^g S_{\ell_0,\ell_1,\ell_2,\ell_3}(p)=0 \quad \text{for $(\ell_0,\ell_1,\ell_2,\ell_3)\in \{(1,1,1,1),(2,2,2,2)\}$}.$$ 
Moreover, using the identity \eqref{phase_tangent} together with the fact that any three pair-wise linearly independent null vectors must be linearly independent (see Lemma~\ref{span_two_null}), we conclude that
$$ \nabla^g S_{\ell_0,\ell_1,\ell_2,\ell_3}(p)\neq 0 \quad \text{for $(\ell_0,\ell_1,\ell_2,\ell_3)\notin \{(1,1,1,1),(2,2,2,2)\}$}.$$ 
This implies that for all $(\ell_0,\ell_1,\ell_2,\ell_3)\notin \{(1,1,1,1),(2,2,2,2)\}$, the phase function $S_{\ell_0,\ell_1,\ell_2,\ell_3}$ appearing in \eqref{other_terms} does not have a critical point. Thus, we can repeatedly use integration by parts to conclude that:
$$ \lim_{\lambda\to \infty}\int_M e^{i\lambda S_{\ell_0,\ell_1,\ell_2,\ell_3}}(\prod_{j=0}^3 \vartheta_{\ell_j}(B_{\kappa_j\lambda}^{(j)}))u_f^{m-3}\,dV_g= \mathcal O(\lambda^{-\infty}),$$
whenever $(\ell_0,\ell_1,\ell_2,\ell_3)\notin \{(1,1,1,1),(2,2,2,2)\}$. Thus, by combining the above arguments we obtain:
\begin{multline*} \mathscr D^{\textrm{semi}}_{\sigma,\delta',f}= 2^{-4}\lim_{\lambda\to \infty} \lambda^{\frac{n+1}{2}}\left(\int_{M}u_f^{m-3}\prod_{j=0}^3\mathcal U_{\kappa_j\lambda}^{(j)}\,dV_{g}+\int_{M}u_f^{m-3}\prod_{j=0}^3\overline{\mathcal U_{\kappa_j\lambda}^{(j)}}\,dV_{g}\right)\\
=2^{-3}\lim_{\lambda\to \infty} \lambda^{\frac{n+1}{2}}\Re \int_{M}\left(\prod_{j=0}^3\mathcal U_{\kappa_j\lambda}^{(j)}\right)u_f^{m-3}\,dV_{g}.\end{multline*}
\end{proof}

Using Lemma~\ref{real_part_lem} we conclude that the expression for $\mathscr D_{\sigma,\delta'}$ reduces to 
$$ \mathscr D^{\textrm{semi}}_{\sigma,\delta',f}=2^{-3}\Re\left(\lim_{\lambda\to\infty}\lambda^{\frac{n+1}{2}}\int_M e^{i\lambda S(x)} u_f^{m-3}(x)\,B_{\kappa_0\lambda}^{(0)}(x)B_{\kappa_1\lambda}^{(1)}(x)B_{\kappa_2\lambda}^{(2)}(x)B_{\kappa_3\lambda}^{(3)}(x)\,dV_g\right).$$
Note that thanks to \eqref{Y_cond}, there holds: 
\bel{real_B} B^{(j)}_{\kappa_j\lambda}(y)=1,\quad \text{for $j=0,1,2,3$.}\ee

We expand the amplitudes $a_{\kappa_j\lambda}^{(j)}$ in the expressions for $B^{(j)}_{\kappa_j \lambda}$ in terms of the functions $a_k^{(j)}$ as in (\ref{phase-amplitude}), and apply Lemma~\ref{lem_5} together with the method of stationary phase (see e.g. Theorem 7.7.5 in \cite{Ho1}) to (\ref{expf}), term-wise after this expansion. Using the key hypothesis \eqref{real_B} together with Lemma~\ref{stationary_lem}, we conclude that
$$
\mathscr D^{\textrm{semi}}_{\sigma,\delta',f}= c_{0}u_f(y),
$$
where $c_0$ is a non-zero real constant as given by Lemma~\ref{stationary_lem} and $y$ is the unique intersection point given in hypothesis (ii) of the Theorem. Note that the application of Lemma~\ref{stationary_lem} is justified here since the product of the four amplitude functions is supported in a small neighborhood of $y$ that depends on the parameter $\delta'$ and so the hypothesis of Lemma~\ref{stationary_lem} is satisfied for $\delta'$ sufficiently small. Finally, since $u_f(y) \neq 0$, it follows that $\mathscr D^{\textrm{semi}}_{\sigma,\delta',f} \neq 0$ thus completing the proof of Theorem~\ref{thm_anal_data_1}.

\subsection{Proof of Theorem~\ref{thm_anal_data_2}}
\label{subsec_quasi}
Applying the linearization argument in Section~\ref{linearization_section}, we deduce that the source-to-solution map $\mathscr N$ determines the knowledge of the expression 
\begin{multline*}
\tilde{\mathcal I}_{\lambda,\sigma,\delta'}=\int_M \Tr(hg^{-1})\,u_0\left(u_1u_2f_3+u_2u_3f_1+u_3u_1f_2\right)\,dV_g\\
+2\int_{M} \left(u_1u_2\langle \nabla^gu_3,\nabla^gu_0\rangle_h+u_2u_3\langle \nabla^gu_1,\nabla^gu_0\rangle_h+u_3u_1\langle \nabla^gu_2,\nabla^gu_0\rangle_h\right)\,dV_g\\
-\int_{M}\Tr(hg^{-1})\left(u_1u_2\langle \nabla^gu_3,\nabla^gu_0 \rangle_g+u_2u_3\langle \nabla^gu_1,\nabla^gu_0 \rangle_g+u_3u_1\langle \nabla^gu_2,\nabla^gu_0 \rangle_g\right)\,dV_g,
\end{multline*}
where
$u_j= u^+_{\kappa_j\lambda,v_j,\delta'}$ for $j=1,2,3$ and $u_0=u^-_{\kappa_0\lambda,v_0,\iota_0,\delta'}$. Also, $f_0=f^-_{\kappa_0\lambda,v_0,\iota_0,\delta'}$ and $f_j=f^+_{\kappa_j\lambda,v_j,\iota_j,\delta'}$ for $j=1,2,3$.
Note also that
$$ \mathscr D_{\sigma,\delta'}= \lim_{\lambda\to\infty}\lambda^{\frac{n-3}{2}}\tilde{\mathcal I}_{\lambda,\sigma,\delta'}.$$
The proof of (i) in Theorem~\ref{thm_anal_data_2} is exactly as the proof of (i) in Theorem~\ref{thm_anal_data_1}. To show (ii), we proceed as before by showing that if there is a point $y\in\overleftarrow{\;\gamma_{v_{0}}} \cap \bigcap_{j=1}^3  \overrightarrow{\gamma_{v_{j}}}$ that satisfies the hypothesis of statement (ii), then there exists $\kappa_j \in \R\setminus \{0\}$ and $\iota_j \in \mathcal T$, such that $\mathscr D^{\textrm{quasi}}_{\sigma,\delta'}\neq 0$, for all $\delta'$ sufficiently small and $\sigma=(v_0,\kappa_0,\iota_0,\ldots,v_3,\kappa_3,\iota_3)$. 

Observe that using the same argument as in the preceding section, we can show that there is a unique point in $\overleftarrow{\;\gamma_{v_{0}}} \cap \bigcap_{j=1}^3  \overrightarrow{\gamma_{v_{j}}}$. Again, analogously to the previous section we observe that since the tangent vectors to $\gamma_{v_j}$, $j=1,2,3,$ are linearly dependent at the point $y$, there exists non-zero constants $\kappa_0,\kappa_1,\kappa_2,\kappa_3$ such that the linear dependence equation \eqref{linear dependence f} holds at the point $y$. We also choose $\iota_j \in \mathcal T$ such that \eqref{YZ_init} holds at the point $y$ and subsequently define the Gaussian beams along the geodesics $\gamma_{v_j}$, $j=0,1,2,3$, as in (\ref{gaussianf}). Recall that due to the choice of initial conditions given by \eqref{YZ_init}, the amplitude functions satisfy \eqref{Y_cond}. 

We proceed to show that given this choice of $\sigma$, there holds $\mathscr D^{\textrm{quasi}}_{\sigma,\delta'}\neq 0$ for all $\delta'$ small. This will be achieved by proving the following three estimates:
 \bel{quasi_est_1}\lim_{\lambda\to \infty} \lambda^{\frac{n-3}{2}}\int_M \Tr(hg^{-1})\,u_0\left(u_1u_2f_3+u_2u_3f_1+u_3u_1f_2\right)\,dV_g=0,\ee
 and
  \bel{quasi_est_2}
 \begin{aligned}
 \lim_{\lambda\to \infty} \lambda^{\frac{n-3}{2}}\int_{M}\Tr(hg^{-1})(u_1u_2\langle \nabla^gu_3,\nabla^gu_0 \rangle_g&+u_2u_3\langle \nabla^gu_1,\nabla^gu_0 \rangle_g\\
 &+u_3u_1\langle \nabla^gu_2,\nabla^gu_0 \rangle_g)\,dV_g=0,
 \end{aligned}
\ee
and finally that
\begin{multline}
\label{quasi_est_3}
\lim_{\lambda\to \infty}\lambda^{\frac{n-3}{2}}\int_{M} (u_1u_2\langle \nabla^gu_3,\nabla^gu_0\rangle_h+u_2u_3\langle \nabla^gu_1,\nabla^gu_0\rangle_h\\
+u_3u_1\langle \nabla^gu_2,\nabla^gu_0\rangle_h)\,dV_g=c_0\,\kappa_0^2\,h(\dot{\gamma}^{(0)}(s_0),\dot{\gamma}^{(0)}(s_0)),
\end{multline}
where $c_0$ is a non-zero constant depending on the geometry $(M,g)$. Note that by assumption (ii) on the family of metrics $\g_z$, $h$ is non-degenerate on null-vectors and therefore the right hand side of the above expression is non-zero. Thus, it follows from the above three estimates that $\mathscr D^{\textrm{quasi}}_{\sigma,\delta'}$ is non-zero.
 
Let us begin by showing that \eqref{quasi_est_1} holds. Using the estimates \eqref{mathcalU_estimate}--\eqref{mathcalU_estimate_p} together with the uniform boundedness of Gaussian beams in $\lambda$ (see \eqref{gauss_uni}) and the estimate \eqref{source_est}, it follows that
\begin{multline*} 
\left| \int_M \Tr(hg^{-1})\,u_0\left(u_1u_2f_3+u_2u_3f_1+u_3u_1f_2\right)\,dV_g \right| \\
\lesssim  \|u_0\|_{\CI(V)}(\|u_1\|_{\CI(V)} \|u_2\|_{\CI(V)}\|f_3\|_{\CI(V)}+ \|u_2\|_{\CI(V)} \|u_3\|_{\CI(V)}\|f_1\|_{\CI(V)}\\
+\|u_3\|_{\CI(V)} \|u_1\|_{\CI(V)}\|f_2\|_{\CI(V)}) \lesssim \lambda^{1-\frac{n}{2}},
\end{multline*}
where $V=J^-(\supp f_0) \cap \bigcup_{j=1}^3 J^+(\supp f_j)$ is compact and lies inside $(-T,T)\times M_0$ by the hypothesis of Theorem~\ref{t1}.

Next, we show that \eqref{quasi_est_2} holds. We use again the estimates \eqref{mathcalU_estimate}--\eqref{mathcalU_estimate_p} together with the uniform boundedness of Gaussian beams in $\lambda$ (see \eqref{gauss_uni}) to write:

 \begin{multline*}
\int_{M}\Tr(hg^{-1})u_1u_2\langle du_3,du_0 \rangle_g\,dV_g\\
 =2^{-4}\sum_{\ell_0,\ldots,\ell_3=1,2}\int_{M}\Tr(hg^{-1})\vartheta_{\ell_j}(\mathcal U_{\kappa_1\lambda}^{(1)})\vartheta_{\ell_j}(\mathcal U_{\kappa_2\lambda}^{(2)})\langle \vartheta_{\ell_j}(\nabla^g\mathcal U_{\kappa_3\lambda}^{(3)}),\vartheta_{\ell_j}(\nabla^g\mathcal U_{\kappa_0\lambda}^{(0)}) \rangle_g\,dV_g\\
 +\mathcal O(\lambda^{-1}\lambda^{-\frac{n-3}{2}}).
 \end{multline*}
Here, recalling \eqref{phase_tangent} and applying property (ii) in Lemma~\ref{lem_5} together with a similar argument as in the proof of Lemma~\ref{real_part_lem}, we can show that as $\lambda$ approaches infinity, only two terms in the above sum contribute so that
\begin{multline*}
\lim_{\lambda\to \infty} \lambda^{\frac{n-3}{2}} \int_{M}\Tr(hg^{-1})u_1u_2\langle \nabla^gu_3,\nabla^gu_0 \rangle_g\,dV_g\\
=2^{-3}\Re \left(\lim_{\lambda\to \infty}\lambda^{\frac{n-3}{2}}\int_M \Tr(g^{-1}h) \mathcal U_{\kappa_1\lambda}^{(1)}\mathcal U_{\kappa_2\lambda}^{(2)}\langle \nabla^g\mathcal U_{\kappa_3\lambda}^{(3)},\nabla^g\mathcal U_{\kappa_0\lambda}^{(0)}\rangle_g\,dV_g\right).
\end{multline*}

Here, using the defining expressions \eqref{gaussianf} and \eqref{S} together with the uniform boundedness of Gaussian beams in the parameter $\lambda$ (see \eqref{gauss_uni}), we write 
$$ \mathcal U_{\kappa_1\lambda}^{(1)}\mathcal U_{\kappa_2\lambda}^{(2)}\langle \nabla^g\mathcal U_{\kappa_3\lambda}^{(3)},\nabla^g\mathcal U_{\kappa_0\lambda}^{(0)}\rangle_g= e^{i\lambda S(x)}\left(-\lambda^2\langle \nabla^g\Phi^{(3)},\nabla^g\Phi^{(0)}\rangle_g\prod_{j=0}^3 B^{(j)}_{\kappa_j\lambda}+\mathcal O(\lambda)\right),$$
where $B^{(j)}_{\kappa_j\lambda}$ are defined as in \eqref{B_sign}. Using this identity, together with \eqref{Y_cond} and Lemma~\ref{stationary_lem} we obtain that:
 \begin{multline*}
 \lim_{\lambda\to \infty} \lambda^{\frac{n-3}{2}}\int_{M}\Tr(hg^{-1})u_1u_2\langle \nabla^gu_3,\nabla^gu_0 \rangle_g\,dV_g\\
 = -2^{-3}\lim_{\lambda\to \infty} \lambda^{\frac{n+1}{2}}\Re\,\int_{M}\Tr(hg^{-1})e^{i\lambda S}\langle \nabla^g\Phi^{(3)},\nabla^g\Phi^{(0)} \rangle_g\,dV_g\\
 =-2^{-3}c_0 \Tr(g^{-1}(y)h(y))\langle \nabla^g\Phi^{(3)}(y),\nabla^g\Phi^{(0)}(y) \rangle_{g(y)},
 \end{multline*}
where we recall that $c_0\neq 0$ is as given by Lemma~\ref{stationary_lem}. Thus, adding the contributions from the other two terms in \eqref{quasi_est_2}, we deduce that:
\begin{multline*}
 \lim_{\lambda\to \infty} \lambda^{\frac{n-3}{2}}\int_{M}\Tr(hg^{-1})(u_1u_2\langle \nabla^gu_3,\nabla^gu_0 \rangle_g+u_2u_3\langle \nabla^gu_1,\nabla^gu_0 \rangle_g\\
 +u_3u_1\langle \nabla^gu_2,\nabla^gu_0 \rangle_g)\,dV_g\\
 =-2^{-3}c_0\Tr(g^{-1}(y)h(y))\left(\sum_{j=1}^3 \langle \nabla^g\Phi^{(3)}(y),\nabla^g\Phi^{(0)}(y)\rangle_{g(y)}\right)\\
=-2^{-3}c_0\Tr(g^{-1}(y)h(y))\,\langle \underbrace{\sum_{j=1}^3 \nabla^g\Phi^{(3)}(y)}_{-\nabla^g\Phi^{(0)}(y)},\nabla^g\Phi^{(0)}(y)\rangle_{g(y)} \\
=2^{-3}c_0\Tr(g^{-1}(y)h(y))\langle \nabla^g\Phi^{(0)}(y),\nabla^g\Phi^{(0)}(y)\rangle_{g(y)}=0,
\end{multline*}
where we used property (ii) in Lemma~\ref{lem_5} to get the last step and there we applied \eqref{phase_tangent} and the fact that $\gamma_{v_0}$ is a null geodesic.

Finally, we proceed to prove the remaining estimate \eqref{quasi_est_3}. Note that analogously to the proof of \eqref{quasi_est_2}, there holds:
\begin{multline*}
\lim_{\lambda\to \infty} \lambda^{\frac{n-3}{2}} \int_{M}u_1u_2\langle \nabla^gu_3,\nabla^gu_0 \rangle_h\,dV_g\\
=2^{-3}\Re \left(\lim_{\lambda\to \infty}\lambda^{\frac{n-3}{2}}\int_M\mathcal U_{\kappa_1\lambda}^{(1)}\mathcal U_{\kappa_2\lambda}^{(2)}\langle \nabla^g\mathcal U_{\kappa_3\lambda}^{(3)},\nabla^g\mathcal U_{\kappa_0\lambda}^{(0)}\rangle_h\,dV_g\right).
\end{multline*}
Now, using the expression
$$ \mathcal U_{\kappa_1\lambda}^{(1)}\mathcal U_{\kappa_2\lambda}^{(2)}\langle \nabla^g\mathcal U_{\kappa_3\lambda}^{(3)},\nabla^g\mathcal U_{\kappa_0\lambda}^{(0)}\rangle_g= e^{i\lambda S(x)}\left(-\lambda^2\langle \nabla^g\Phi^{(3)},\nabla^g\Phi^{(0)}\rangle_h\prod_{j=0}^3 B^{(j)}_{\kappa_j\lambda}+\mathcal O(\lambda)\right),$$
together with Lemma~\ref{stationary_lem} and the key identity \eqref{Y_cond}, we obtain:
 \begin{multline*}
 \lim_{\lambda\to \infty} \lambda^{\frac{n-3}{2}}u_1u_2\langle \nabla^gu_3,\nabla^gu_0 \rangle_h\,dV_g\\
 = -2^{-3}\lim_{\lambda\to \infty} \lambda^{\frac{n+1}{2}}\Re\,\int_{M}e^{i\lambda S}\langle \nabla^g\Phi^{(3)},\nabla^g\Phi^{(0)} \rangle_h\,dV_g\\
 =-2^{-3}c_0 \langle \nabla^g\Phi^{(3)}(y),\nabla^g\Phi^{(0)}(y) \rangle_{h(y)},
 \end{multline*}
where we recall that $c_0\neq 0$ is as given by Lemma~\ref{stationary_lem}. Finally, adding the analogous contributions from the remaining two terms in \eqref{quasi_est_3}, we obtain
\begin{multline*}
 \lim_{\lambda\to \infty} \lambda^{\frac{n-3}{2}}\int_{M}(u_1u_2\langle \nabla^gu_3,\nabla^gu_0 \rangle_h+u_2u_3\langle \nabla^gu_1,\nabla^gu_0 \rangle_h\\
 +u_3u_1\langle \nabla^gu_2,\nabla^gu_0 \rangle_h)\,dV_g=-2^{-3}c_0\,\langle \sum_{j=1}^3\underbrace{\nabla^g\Phi^{(j)}(y)}_{-\nabla^g\Phi^{(0)}(y)},\nabla^g\Phi^{(0)}(y)\rangle_{h(y)}\\
 =2^{-3}c_0\langle \nabla^g\Phi^{(0)}(y),\nabla^g\Phi^{(0)}(y)\rangle_{h(y)}\neq 0,
\end{multline*}
where we used property (ii) in the definition of the tensor $G(x,z)$ in the last step. This concludes the proof of the Theorem~\ref{thm_anal_data_2}.

\section{On globally hyperbolic manifolds}
\label{geometry_prelim}

\HOX{$n \geq 2$ is used only in Lemma \ref{lem_linalg_pert}}
We start the geometric part of our analysis, first considering (in this section) the geometric notations and results that will be used to prove Theorem 1.3. As before, we assume $(M,g)$ to be a globally hyperbolic Lorentzian manifold of dimension $1+n$ with $n \geq 2$.
We write $\le$ and $\ll$ for the causal and chronological relations, and $\tau$ for the time separation function. 
Recall that $\le$ is closed, $\ll$ is open and $\tau$ is continuous, 
see e.g. \cite[Lemmas 3 (p. 403), 21--22 (p. 412)]{O'Neill}.
Occasionally we will consider causal relations on a subset $\Omega \subset M$, and we say that $x \le y$ in $\Omega$ if there is a causal future pointing path from $x$ to $y$, staying in $\Omega$, or if $x = y$. Analogously, $x \ll y$ in $\Omega$ if there is a timelike future pointing path from $x$ to $y$, staying in $\Omega$.

The next short cut argument, see \cite[Prop. 46 (p. 294)]{O'Neill}, will be very useful in what follows.

\begin{lemma}\label{lem_shortcut}
If there is a future pointing causal path from $x$ to $y$ on $M$ that is not a null pregeodesic then $x \ll y$.
\end{lemma}

To simplify the notations we often lift functions and relations from $M$ to $TM$ by using the natural projection $\pi : TM \to M$. For example, we write $v \le w$ if $\pi(v) \le \pi(w)$, and $\tau(v,w) =\tau(\pi(v), \pi(w))$, for $v,w \in TM$.
The bundle of lightlike vectors is denoted by $LM$, and $L^+ M$ and $L^-M$ are the future and past pointing subbundles. 
We define the causal bundle (with boundary)
    \begin{align*}
C M = \{v \in TM : \text{ $v$ is causal}\},
    \end{align*}
and write again $C^+ M$ and $C^- M$ for the future and past pointing subbundles.  
When $K \subset M$, we write $LK  = \{(x,\xi) \in LM : x \in K\}$ and use the analogous notation for other bundles as well. 

We denote by $\gamma_{v} : (a,b) \to M$ the inextendible geodesic on $M$ with the initial data $v \in CM$
 and write 
    \begin{align*}
\beta_v : (a,b) \to TM, \quad \beta_{v}(s) = (\gamma_{v}(s),\dot\gamma_{v}(s)).
    \end{align*}
Then $\beta_v(0) = v$ and $a < 0 < b$.

\subsection{Compactness results}

For $p, q \in M$ the causal future and past of $p$ and $q$, respectively, are 
    \begin{align*}
J^+(p) = \{x \in M: p \le x\}, \quad
J^-(q) = \{x \in M : x \le q\}.
    \end{align*}
The causal diamonds
    \begin{align}\label{def_D_closed}
J^+(p) \cap J^-(q) = \{ x \in M : p \le x \le q \}
    \end{align}
are compact. 
More generally, we write $J^\pm(S) = \bigcup_{x \in S} J^\pm(x)$ for a set $S \subset M$.
If $K_1, K_2 \subset M$ are compact, then also
$J^+(K_1) \cap J^-(K_2)$ is compact. Indeed, writing  $K = K_1 \cup K_2$, this follows from $J^+(K) \cap J^-(K)$ being compact, and both $J^+(K_1)$ and $J^-(K_2)$ being closed, see \cite[Th. 2.1 and Prop. 2.3]{HM}.

The fact that $(M,g)$ is not assumed to be geodesically complete causes some technical difficulties. We will typically handle these issues by working in a compact subset. We have the following variation of \cite[Lem. 9.34]{Beem}:

\begin{lemma}\label{lem_limits_defined}
Let $K \subset M$ be compact and suppose that $v_j \to v$ in $C^+ K$, $s_j \to s \geq 0$ in $\R$
and that $\gamma_{v_j}(s_j) \in K$.
Then the inextendible geodesic $\gamma_v : (a,b) \to M$ satisfies $b > s$.
\end{lemma}
\begin{proof}
As $K$ is compact, by passing to a subsequence, still denoted by $(v_j, s_j)$, we may assume that $\gamma_{v_j}(s_j) \to x$ in $K$. Let $\tilde x \in M$ satisfy $x \ll \tilde x$. 
To get a contradiction, suppose that $b \le s$. Let $0 < t < b$.
Then $\gamma_v(t) = \lim_{j \to \infty} \gamma_{v_j}(t)$
and for large $j$ it holds that $\gamma_{v_j}(t) \le \gamma_{v_j}(s_j) \ll \tilde x$.
As the relation $\le$ is closed, it follows that $\gamma_v(t) \le \tilde x$ for $0 < t < b$.
Now the future inextendible causal curve $\gamma_v(t)$, $0 < t < b$, never leaves the compact set $J^+(p) \cap J^-(q)$ where $p = \pi(v)$ and $q = \tilde x$. This is a contradiction with \cite[Lem. 13, p. 408]{O'Neill}. 
\end{proof}

\begin{lemma}\label{lem_exit}
Let $K \subset M$ be compact. 
The exit function 
    \begin{align*}
R(v) = \sup \{ s \geq 0 : \gamma_{v}(s) \in K\}, \quad v \in C^+ K.
    \end{align*}
is finite and upper semi-continuous.
\end{lemma}
\begin{proof}
Finiteness follows from \cite[Lem. 13 (p. 408)]{O'Neill}. Suppose that $v_j \to v$ in $C^+ K$ 
and that $t_j := R(v_j) \to t$ for some $t \geq 0$. 
The upper semi-continuity $t \le R(v)$ follows from the convergence $\gamma_{v_{j}}(t_{j}) \to \gamma_v(t)$ in $K$, 
that again follows from Lemma \ref{lem_limits_defined}.
\end{proof}

\begin{lemma}\label{lem_s_proper}
Suppose that $v_j \to v$ in $C^+M$.
If a sequence $s_j \geq 0$, $j \in \N$, satisfies 
$\gamma_{v_j}(s_j) \to y$ for some $y \in M$, then $s_j$
converges.
\end{lemma}
\begin{proof}
We write $\pi(v) = x$, $\pi(v_j) = x_j$ and $\gamma_{v_j}(s_j) = y_j$.
Let $X$ and $Y$ be bounded neighborhoods of $x$ and $y$, respectively, and write $K= J^+(\overline X) \cap J^-(\overline Y)$.
Then we have $x_j, y_j \in K$ for large $j$.
Now Lemma \ref{lem_exit} implies that $s_j < R(v_j) \le R(v) + 1$ for large $j$, where $R$ is the exit function of $K$.
Write $t^+ = \limsup_{j \to \infty} s_j$ and $t^- = \liminf_{j \to \infty} s_j$. These are both finite.
There are subsequences $s_{j_k}^\pm$ converging to $t^\pm$ and 
$\gamma_{v_{j_k}}(s_{j_k}^\pm) \to \gamma_v(t^\pm) = y$.
Now $t^- = t^+$ by global hyperbolicity. 
\end{proof}

The analogues of Lemmas \ref{lem_limits_defined}--\ref{lem_s_proper} hold also for past pointing vectors.

\subsection{Cut function}

The cut function is defined by
    \begin{align*}
\rho(v) = \sup \{ s > 0 : \tau(v,\beta_{v}(s)) = 0 \}, \quad v \in L^+ M.
    \end{align*}
We define $\rho(v)$ also for $v \in L^- M$ by the above expression but with respect to the opposite time orientation.  
It follows from the definition of $\rho$ 
and Lemma \ref{lem_shortcut}
that if $\gamma_v(s)$ is well-defined for some $s > \rho(v)$, then there is a timelike path from $\gamma_v(0)$ to $\gamma_v(s)$. 
On the other hand, \cite[Lem. 9.13]{Beem} implies the following:

\begin{lemma}\label{lem_uniq_to_cut}
The geodesic segment along $\gamma_v$, with $v \in L^+M$, 
is the only causal path from $\gamma_v(0)$ to $\gamma_v(s)$ for $s < \rho(v)$ up to a reparametrization. 
\end{lemma}

The following two lemma is a variant of \cite[Prop. 9.7]{Beem}. 
\begin{lemma}\label{lem_rho_semicont}
The cut function $\rho : L^+ M \to [0,\infty]$ is lower semi-continuous.
\end{lemma}
\begin{proof}
Suppose that $v_j \to v$ in $L^+M$ 
and write $t_j = \rho(v_j)$.
We need to show that if $t_j \to t$ for some 
$t \geq 0$ then
$t \geq \rho(v)$.
To get a contradiction, suppose that the opposite holds. 
Then there is $\delta > 0$ such that $t + \delta < \rho(v)$.
In particular, $\gamma_v(t+\delta)$ is well-defined, and this implies that also $\gamma_{v_j}(t_j + \delta)$ is well-defined for large $j$.
Writing $x_j = \pi(v_j)$ and $y_j = \gamma_{v_j}(t_j + \delta)$,
there holds $\tau(x_j, y_j) > 0$ since $t_j  + \delta> \rho(v_j)$. 
We write also $(x,\xi) = v$ and $y = \gamma_v(t + \delta)$.
Then $x_j \to x$ and $y_j \to y$.
Let $X$ be a bounded neighborhood of $x$ and define $K= \overline X$.

Let us choose an auxiliary Riemannian metric on $M$ and denote by $SM$ the unit sphere bundle with respect to that metric.
By \cite[Prop. 19, p. 411]{O'Neill} there is a timelike geodesic from $x_j$ to $y_j$. We may reparametrize this geodesic to obtain timelike 
$\xi_j$ in $C^+_{x_j} K \cap S_{x_j} M$ and $s_j > 0$ satisfying 
$\gamma_{x_j,\xi_j}(s_j) = y_j$.
As $C^+ K \cap SM$ is compact, by passing to a subsequence, we may assume that $\xi_j \to \tilde\xi$ for some $\tilde\xi \in C^+_{x} K \cap S_x M$. Lemma \ref{lem_s_proper} implies that $s_j \to s$ for some $s \geq 0$.

If there is no $c \in \R$ such that $\tilde \xi = c \xi$, then 
there are two distinct causal geodesics from $x$ to $y$.
This is a contradiction in view of Lemma \ref{lem_uniq_to_cut} since $x = \pi(v)$, $y = \gamma_v(t+\delta)$ and $t + \delta < \rho(v)$. 

Suppose now that there is $c \in \R$ such that $\tilde \xi = c \xi$. Then $(x_j,c^{-1}\xi_j) \to v$ and
 $cs_j \to t + \delta$.
None of the points $\gamma_v(r)$, $0 \le r \le t + \delta$, is conjugate to $x$ along $\gamma_v$ by \cite[Th. 10.72]{Beem}, and the map $s \mapsto \gamma_v(s)$ is injective due to global hyperbolicity. Hence there is a neighborhood $U$ of $[0, t+\delta] v$
such that $\pi \times \exp$ is injective on $U$.
But $t_j v_j, (x_j, s_j \xi_j) \in U$ for large $j$ and both are mapped to $(x_j, y_j)$. This is a contradiction since the former lightlike and the latter is timelike.
\end{proof}

The following lemma is a variant of \cite[Prop. 9.5]{Beem}.

\begin{lemma}\label{lem_rho_cont_aux}
Let $v_j \to v$ in $L^+M$ and $\rho(v_j) \to t$ in $\R$.
Suppose that $\gamma_v(t)$ is well-defined, that is to say, $t\leq R(v)$. 
Then $\rho(v) = t$.
\end{lemma}
\begin{proof}
Lower semi-continuity of $\rho$ implies that $\rho(v) \le t$. To get a contradiction suppose there is $\delta > 0$ such that 
$\rho(v) + \delta < t$. Then $\rho(v) + \delta < \rho(v_j)$ for large $j$.
We are lead to the contradiction 
    \begin{align*}
\tau(v, \beta_{v}(\rho(v) + \delta)) = \lim_{j \to \infty} \tau(v_j, \beta_{v_j}(\rho(v) + \delta)) = 0.
    \end{align*} 

\end{proof}


The analogues of Lemmas \ref{lem_uniq_to_cut}--\ref{lem_rho_cont_aux} hold also for past pointing vectors.
Moreover, the cut function has the following symmetry.

\begin{lemma}\label{lem_rho_symmetry}
Let $v \in L^+M$ and suppose that $\gamma_v(\rho(v))$ is well-defined. 
Then 
    \begin{align*}
\rho(-\beta_v(\rho(v))) = \rho(v).
    \end{align*}
\end{lemma}
\begin{proof}
Write $w = -\beta_v(\rho(v))$. To get a contradiction suppose that $\rho(w) < s < \rho(v)$.
Then there is a past pointing timelike path from $\gamma_w(0)$ to $\gamma_w(s)$, a contradiction with $\tau(v,w) = 0$.
To get a contradiction suppose that $\rho(v) < \rho(w)$.
For small $\epsilon > 0$ the vector $\tilde w = -\beta_v(\rho(v) + \epsilon)$ is well-defined. Moreover, lower semi-continuity of $\rho$ implies that $\rho(v) + \epsilon < \rho(\tilde w)$ for small enough $\epsilon > 0$.
Lemma \ref{lem_uniq_to_cut} implies then that $\gamma_{\tilde w}$
is the only causal path from $\gamma_{\tilde w}(0)$ to $\gamma_{\tilde w}(\rho(v) + \epsilon) = \pi(v)$. Therefore $\tau(v, \tilde w) = 0$, a contradiction with the definition of $\tilde w$.
\end{proof}

The above lemma implies that if $\gamma_v(s)$, with $s < 0$, and $\gamma_v(\rho(v))$ are both well-defined, then there is a timelike path from $\gamma_v(s)$ to $\gamma_v(\rho(v))$.

\subsection{Optimizing geodesics and earliest observation functions}\label{subsec: optimizing geodesic}

We say that a future pointing causal path $\gamma$ from $x$ to $y$ on $M$ is optimizing if $\tau(x,y) = 0$. Equivalently,
$\gamma$ from $x$ to $y$ is optimizing if it is a segment along some inextendible $\gamma_v$ with $v \in L^+ M$, $\gamma_v(0) = x$ and $y = \gamma_v(s)$ for some $s \leq \rho(v)$.
In the case $s = \rho(v)$ there may be other optimizing geodesics from $x$ to $y$, corresponding to different initial directions at $x$.
One more simple, but nonetheless useful, observation is that if there is an optimizing path from $x$ to $y$, then all causal paths from $x$ to $y$ are optimizing.

{\mltext Below we consider the paths
$\mu_{\i}:[t_1^-,t_1^+]\to M$ and $\mu_{\o}:[s_1^-,s_1^+]\to M,$ and to simplify notations, we assume that $$t_1^-=s_1^-=-1,\quad t_1^+=s_1^+=1.$$}

We define for a timelike future pointing path $\mu : [-1,1] \to M$ the earliest observation functions
    \begin{align*}
f_\mu^+(x) &=\inf\{s\in(-1,1) : \text{ $\tau(x,\mu(s))>0$ or $s = 1$} \}, \quad &x \in M,\\
f_\mu^-(x) &=\sup\{s\in(-1,1) : \text{ $\tau(\mu(s),x)>0$ or $s = -1$} \}, \quad &x \in M.
    \end{align*}
These functions are continuous \cite[Lemma 2.3 (iv)]{KLU}.

\begin{lemma}\label{lem_optim_exists}
Let $\mu : [-1,1] \to M$ be a timelike  future pointing path. Suppose that $x \le \mu(1)$ satisfies $x \not\le \mu(-1)$, in other words, $x \notin J^-(\mu(-1))$.
Then there is $s \in (-1, 1]$ such that either there is an optimizing causal geodesic from $x$ to $\mu(s)$ or $x = \mu(s)$.
In both the cases $s = f_\mu^+(x)$.
\end{lemma}
\begin{proof}
We set $s = f_\mu^+(x)$ and $y = \mu(s)$.
If $s = 1$ then $x \le y$ by the assumption $x \le \mu(1)$.
On the other hand, if $s < 1$ also then $x \le y$ since the causal relation $\le$ is closed. Hence there is a causal path from $x$ to $y$ or $x=y$.
It remains to show that in the former case, the path is optimizing.
There holds $s > -1$ since $x \not\le \mu(-1)$.
This again implies that $\tau(x,y) = 0$.
\end{proof}

A variation of the above proof gives:
\begin{lemma}\label{lem_optim_exists2}
Let $\mu : [-1,1] \to M$ be a timelike  future pointing path. Suppose that $\mu(-1) \le x$ and $\mu(1) \not\le x$.
Then there is $s \in [-1, 1)$ such that either there is an optimizing causal geodesic from $\mu(s)$ to $x$ or $x = \mu(s)$.
In both the cases $s = f_\mu^-(x)$.
\end{lemma}
\begin{proof}
\HOX{Hide the proof in the final version. The proof shows in fact that $\not\le$ could be replaced by $\not<$ with the price that the case $x = \pi(1)$ is possible}
We set $s = f_\mu^-(x)$ and $y = \mu(s)$.
If $s = -1$ then $y \le x$ by the assumption $\mu(-1) \le x$.
On the other hand, if $s > -1$ also then $y \le x$ since the causal relation $\le$ is closed. Hence there is a causal path from $y$ to $x$ or $x=y$.
It remains to show that the path is optimizing in the former case.
There holds $s < 1$ since $\mu(1) \not< x$ and $x \ne y$.
This again implies that $\tau(y,x) = 0$.
\end{proof}


\begin{lemma}\label{lem_f_inc}
Let $\mu : [-1,1] \to M$ be a timelike future pointing path, let $v \in L^+M$,
and write $f(s) = f_\mu^+(\gamma_v(s))$.
Suppose that $\mu$ and $\gamma_{v}$ do not intersect.
Then 
\begin{itemize}
\item[(1)] $f$ is increasing,
\item[(2)] if $-1 < f(s_0) < 1$ for some $s_0$ then $f$ is strictly increasing near $s_0$,
\item[(3)] if $f(s_0) = 1$ and $\gamma_v(s_0) < \mu(1)$ for some $s_0$ then $f$ is strictly increasing for $s < s_0$ near $s_0$.
\end{itemize}
\end{lemma}
\begin{proof}
If there is $s > s_0$ such that $f(s) \leq f(s_0)$ and $f(s) < 1$ then there is a causal path from $\gamma_v(s_0)$ to $\mu(f(s_0))$ via $\gamma_v(s)$ that is not a null pregeodesic and therefore $\tau(\gamma_v(s_0), \mu(f(s_0))) > 0$.
If also $-1 < f(s_0)$ then $\tau(\gamma_v(s_0), \mu(f(s_0))) = 0$, a contradiction. This shows (2), and also that if $f(s_0) = 1$ then $f(s) = 1$ for $s > s_0$.

If $f(s_0) = -1$, then $f(s) = -1$ for $s < s_0$. Indeed, there is a non-optimizing causal path from $\gamma_v(s)$ to $\mu(-1)$ via $\gamma_v(s_0)$. We have shown (1).

Let us now suppose that $f(s_0) = 1$ and there is a causal path from $\gamma_v(s_0)$ to $\mu(1)$. Let $s < s_0$ be near $s_0$. 
There is a causal path from $\gamma_v(s)$ to $\mu(1)$ via $\gamma_v(s_0)$ that is not a null pregeodesic and therefore $\tau(\gamma_v(s_0), \mu(1)) > 0$. But then also $\tau(\gamma_v(s_0), \mu(t)) > 0$ for $t$ close to 1 by continuity. This implies that $f(s) < 1$ and as also $-1 < f(s)$ by continuity, we see that (3) follows from (2).
\end{proof}

We have the following variant of \cite[Lemma 2.3(iv)]{KLU}.

\begin{lemma}
Let $\mu_a : [-1,1] \to M$
be a family of timelike future pointing paths
and suppose that 
$\mu_a(s) = (s, a)$, $a \in U \subset \R^n$,
in some local coordinates. Suppose that $x_j \to x$ in $M$
and $a_j \to a$ in $B(0,\delta)$.
Then $f_{\mu_{a_j}}^+(x_j) \to f_{\mu_a}^+(x)$.
\end{lemma}
\begin{proof}
Let us consider first the case that $s:=f_{\mu_a}^+(x) < 1$.
Then $\tau(x, (s + \epsilon, a)) > 0$ for any small $\epsilon > 0$. Continuity of $\tau$ implies that 
$\tau(x_j, (s + \epsilon, a_j)) > 0$
for large $j$. Hence $\limsup_{j\to\infty}f_{\mu_{a_j}}^+(x_j) \leq s + \epsilon \to s$ as $\epsilon \to 0$. Clearly also $\limsup_{j\to\infty}f_{\mu_{a_j}}^+(x_j) \leq s$
in the case $s = 1$.

To get a contradiction, suppose that 
$\tilde s := \liminf_{j\to\infty}f_{\mu_{a_j}}^+(x_j) < s$.
By passing to a subsequence, we may replace $\liminf$ by $\lim$ above. Moreover,
    \begin{align*}
\tau(x_j, (s, a_j)) \geq \tau(x_j, (\tilde s, a_j)) + \tau((\tilde s, a_j), (s, a_j)),
    \end{align*}
and letting $j \to \infty$, we obtain
    \begin{align*}
\tau(x, (s, a)) \geq \tau((\tilde s, a), (s, a)) > 0.
    \end{align*}
This is in contradiction with $s=f_{\mu_a}^+(x)$ since $s > \tilde s \geq -1$.
\end{proof}

\subsection{Three shortcut arguments}

We denote the image of a path $\mu : I \to M$,
with $I$ an interval in $\R$, by
    \begin{align*}
\overline \mu = \mu(I) = \{\mu(s) : s \in I\}.
    \end{align*}
ans use also the shorthand notations
    \begin{align*}
\overleftarrow{\;\gamma_v} = \{x \in \overline{\gamma_v} : x \le \pi(v) \}, 
\quad
\overrightarrow{\gamma_v} = \{x \in \overline{\gamma_v} : x \ge \pi(v) \}.
    \end{align*}
We emphasize here that $\overline{\gamma_v}$ is defined on the maximal interval $I=(-R(-v),R(v))$ on which the geodesic can be defined in $M$, that is, $\overline{\gamma_v}$ is an inextendible geodesic on $M$. We say that two geodesics $\gamma_v$ and $\gamma_w$ are distinct if $\overline{\gamma_v} \ne \overline{\gamma_w}$.

\begin{lemma}\label{lem_shortcut2}
Let $x_1, x_2, y \in M$ and $v_1, v_2, w \in L^+ M$.
Suppose that $\gamma_{v_j}$ is optimizing from $x_j$ to $y$ for $j=1,2$, and that $\gamma_{v_1}$, $\gamma_{v_1}$ and $\gamma_{w}$ are all distinct. 
Suppose, furthermore, that $y_j \in \overline{\gamma_{v_j}} \cap \overline{\gamma_{w}}$ satisfy $x_j < y_j$ for both $j=1,2$.
Then either $y=y_1=y_2$ or at least one of $y_1$, $y_2$ satisfies $y < y_j$.
\end{lemma}
\begin{proof}
As $y_j, y \in \overline{\gamma_{v_j}}$ there holds either $y < y_j$ or $y_j \le y$.
We suppose that $y_j \le y$ for both $j=1,2$
and show that $y=y_1=y_2$. 

To get a contradiction, suppose that 
$y_1 < y_2$. Then there are two distinct causal paths from $y_1$ to $y$, one along $\gamma_{v_1}$ and the other first along $\gamma_{w}$ and then along $\gamma_{v_2}$, a contradiction with $\gamma_{v_1}$ being optimizing from $x_1 < y_1$ to $y$.
By symmetry, also $y_1 < y_2$ leads to a contradiction.

We have shown that $y_1 = y_2$.
To get a contradiction, suppose that $y_1 < y$.
Then $\gamma_{v_j}$, $j=1,2$, are two distinct causal paths from $y_1$ to $y$, a contradiction with $\gamma_{v_1}$ being optimizing from $x_1 < y_1$ to $y$.
\end{proof}

\begin{lemma}\label{lem_C}
Let $v \in L^+ M$, let $K \subset M$ be compact, and let $x \in \overline{\gamma_v}$ satisfy $x < \pi(v)$.
Then there is a neighborhood $\mathcal U \subset L^+ M$ of $v$
such that for all $w \in \mathcal U$ it holds that if there are 
$y \in \overrightarrow{\gamma_{w}} \cap K$
and $z \in \overrightarrow{\gamma_{v}} \cap K$,
satisfying $y < z$,
and two distinct geodesics from $y$ to $z$, then $\gamma_v$ is not optimizing from $x$ to $z$.
\end{lemma}
\begin{proof}
Suppose that there are $\{w_j\}_{j=1}^{\infty}\subset L^+M$ with $\lim_{j\to\infty}w_j=v$ and
    \begin{align*}
 y_j \in \overrightarrow{\gamma_{w_j}} \cap K, \quad z_j \in \overrightarrow{\gamma_{v}} \cap K,
    \end{align*}
satisfying $y_j < z_j$, and two distinct geodesics from $y_j$ to $z_j$.
Due to compactness, we may pass to a subsequence and assume without loss of generality that $y_j \to y$ and $z_j \to z$ for some $y, z \in K$.
Now $w_j \to v$ and $y_j \to y$ imply that $y \in \overrightarrow{\gamma_{v}}$.
Let us show that $\gamma_v$ is not optimizing from any $\tilde y < y$ to $z$.

We begin by showing that the points $y$ and $z$ must be distinct. To get a contradiction suppose that $y=z$. As $M$ is globally hyperbolic, $y$ has an arbitrarily small neighborhood $U$ such that no causal path that leaves $U$ ever returns to $U$. Thus for large $j$ the two distinct causal geodesics from $y_j$ to $z_j$ are contained in $U$.  
But when $U$ is small, it is contained in a convex neighborhood of $y$, see e.g. \cite[Prop. 7 (p. 130)]{O'Neill}, a contradiction.  

As $y_j < z_j$, the relation $\le$ is closed, and $y \ne z$, we have $y < z$.
Denote by $\eta_j$ the  direction of a geodesic from $y_j$ to $z_j$, normalized with respect to some auxiliary Riemannian metric.
Due to compactness, we may pass to a subsequence and assume without loss of generality that $\eta_j \to \eta$. 

If $\eta$ is not tangent to $\gamma_v$ at $y$ then 
the causal path given by $\gamma_v$ from $\tilde y$ to $y$
and by $\gamma_{y,\eta}$ from $y$ to $z$ is not a null pregeodesic. Hence  $\gamma_v$ is not optimizing from $\tilde y$ to $z$ as required.

Let us now suppose that $\eta$ is tangent to $\gamma_v$ at $y$.
We write $z_j = \gamma_{y, \eta_j}(s_j)$. 
Lemma \ref{lem_s_proper} implies that $s_j \to s$ for some $s \geq 0$, and therefore $z = \gamma_{y, \eta}(s)$.
Moreover, $\rho(y,\eta_j) \leq s_j$ by Lemma \ref{lem_uniq_to_cut}, and by passing once again to a subsequence we may assume that $\rho(y,\eta_j) \to t$ for some $t \leq s$. Now $\gamma_{y,\eta}(t)$ is well-defined and Lemma \ref{lem_rho_cont_aux} implies that $\rho(y,\eta) = t$.
Finally, Lemma \ref{lem_rho_symmetry} implies that $\gamma_v$ is not optimizing from $\tilde y$ to $z = \gamma_{y, \eta}(s)$.

We have shown that $\gamma_v$ is not optimizing from any $\tilde y < y$ to $z$.
To get a contradiction, suppose that $\gamma_v$ is optimizing from $x$ to $z_j$.
Then $\tau(x,z) = \lim_{j \to \infty} \tau(x,z_j) = 0$. In particular, $\gamma_v$ is optimizing from $x$ to $z$,
a contradiction since $x < \pi(v) \le y$.
\end{proof}

\begin{lemma}\label{lem_C2}
Let $v \in L^+ M$, let $K \subset M$ be compact, and let $x \in \overline{\gamma_v}$ satisfy $x < \pi(v)$.
Let $\mu_a : [-1,1] \to M$
be a family of timelike and future pointing paths
and suppose that 
$\mu_a(s) = (s, a)$, $a \in B(0,\delta)$,
in some local coordinates.
Suppose, furthermore, that 
$\overline{\gamma_v} \cap \overline{\mu_0} = \emptyset$ and $f_{\mu_0}^+(x) > -1$.
Then there are a neighborhood $\mathcal U \subset L^+ M$ of $v$ and $0 < \delta' < \delta$
such that for all $w \in \mathcal U$ and all $a \in B(0,\delta')$ it holds that if there are 
$y \in \overrightarrow{\gamma_{v}} \cap K$
and $z \in \overrightarrow{\gamma_{w}} \cap K$,
satisfying $y < z$,
and two distinct geodesics from $y$ to $z$, 
then $f_{\mu_a}^+(z) \geq f_{\mu_a}^+(\gamma_v(t))$ whenever $\gamma_v$ is optimizing from $x$ to $\gamma_v(t)$.
\end{lemma}
\begin{proof}
If $z \not\le \mu_a(1)$ then $f_{\mu_a}^+(z) = 1$ and the conclusion is trivial. Thus we may assume that $z \le \mu_a(1)$.

Suppose that there are $\{w_j\}_{j=1}^{\infty}\subset L^+M$ with $\lim_{j\to\infty}w_j=v$ and
    \begin{align*}
y_j \in \overrightarrow{\gamma_{v}} \cap K, \quad z_j \in \overrightarrow{\gamma_{w_j}} \cap K,
    \end{align*}
satisfying $y_j < z_j$, and two distinct geodesics from $y_j$ to $z_j$.
Due to compactness, we may pass to a subsequence and assume without loss of generality that $y_j \to y$ and $z_j \to z$ for some $y, z \in K$.
Now $w_j \to v$ and $z_j \to z$ imply that $z \in \overrightarrow{\gamma_{v}}$.
As in the proof of Lemma \ref{lem_C}, we see that $\gamma_v$ is not optimizing from any $\tilde y < y$ to $z$.
In particular, $\gamma_v$ is not optimizing from $x$ to $z$.

Let $\xi \in L_x^+ M$ satisfy $\overline{\gamma_{x,\xi}} = \overline{\gamma_v}$.
To get a contradiction, suppose that there are $a_j \to 0$ such that $z_j \le \mu_{a_j}(1)$ and 
$f_{\mu_{a_j}}^+(z_j) \leq f_{\mu_{a_j}}^+(\gamma_{x,\xi}(t_j))$ for some $0 < t_j \leq \rho(x,\xi)$.
As the relation $\le$ is closed, we have $z \le \mu_0(1)$, and in fact, $z < \mu_0(1)$ as $\overline{\gamma_v} \cap \overline{\mu_0} = \emptyset$.
We write $z = \gamma_{x,\xi}(s)$ for some $s > \rho(x,\xi)$.
Using the assumptions $\overline{\gamma_v} \cap \overline{\mu_0} = \emptyset$ and $f_{\mu_0}^+(x) > -1$, Lemma~\ref{lem_f_inc} implies that the function $t \mapsto f_{\mu_0}^+(\gamma_{x,\xi}(t))$ is strictly increasing for $t < s$ near $s$.
Hence 
    \begin{align*}
f_{\mu_0}^+(\gamma_{x,\xi}(\rho(x,\xi))) < f_{\mu_0}^+(\gamma_{x,\xi}(s)) = f^+_{\mu_0}(z).
    \end{align*}
Moreover,    
    \begin{align*}
f_{\mu_{a_j}}^+(z_j) \leq f_{\mu_{a_j}}^+(\gamma_{x,\xi}(t_j)) \leq 
f_{\mu_{a_j}}^+(\gamma_{x,\xi}(\rho(x,\xi))),
    \end{align*}
and letting $j \to \infty$ leads to the contradiction $f_{\mu_0}^+(z)<f_{\mu_0}^+(z)$.
\end{proof}

\subsection{Flowout from a point}

Consider the following set given by the flowout along null rays from a point $x \in M$
    \begin{align}\label{def_C}
C(x) &= \{\beta_{x,\xi}(1) : \xi \in L_x^+ M \}\subset TM.
    \end{align}
It is easy to see that $C(x)$ is a smooth submanifold of dimension $n$ in $TM$.

\begin{lemma}\label{lem_flowout_int}
Let $\mathcal F \subset M$ be finite and non-empty, 
write $C = \bigcup_{x \in \mathcal F} \overline{C(x)}$, and let $v \in C$. 
Then there is a neighborhood $V \subset L^+M$ of $v$ such that any $w \in C \cap V$ satisfies $\overleftarrow{\;\gamma_v} \cap \overleftarrow{\;\gamma_w} \cap \mathcal F \ne \emptyset$.
\end{lemma}
\begin{proof}
Write $K=J^+(\mathcal F) \cap J^-(\mathcal F)$ and let $R$ be the corresponding exit function. 
There are a neighborhood $W_0 \subset L^+M$ of $v$ and
$\epsilon > 0$ such that $\gamma_w(s)$ is well-defined for $w \in W_0$ and $s \in I$ where
$I  = [-R(-v)-\epsilon, 0]$.
As $R$ is upper semi-continuous, there is a neighborhood $W_1 \subset W_0$ of $v$ such that $R(-w) \leq R(-v) + \epsilon$ for all $w \in \overline W_1$. We may assume that $W_1$ is bounded.
Choose an auxiliary Riemannian metric on $M$ and denote by $d$ the distance function with respect to this metric. Define the function
    \begin{align*}
h : \overline W_1 \times I \to \R, \quad
h(w,s) = d(\gamma_w(s), \mathcal F \setminus \overleftarrow{\;\gamma_v}).
    \end{align*}
Then $h$ is uniformly continuous and 
there is $c > 0$ such that
$h(v,s) > c$ for all $s \in I$.
Hence there is a neighborhood $W_2 \subset W_1$ of $v$ such that $h(w,s) > 0$ for all $w \in W_2$ and $s \in I$.
Now $w \in C \cap W_2$ satisfies $\overleftarrow{\;\gamma_w} \cap (\mathcal F \setminus \overleftarrow{\;\gamma_v}) = \emptyset$.
Therefore $w \in C$ implies that $\overleftarrow{\;\gamma_w} \cap (\mathcal F \cap \overleftarrow{\;\gamma_v}) \ne \emptyset$.
\end{proof}

\subsection{Earliest observation sets}
We define the set of earliest observations $E(x)$ of null rays from $x \in M$ to $\Omega \subset M$ by
    \begin{align}\label{def_E}
E(x) &= \{ v \in \overline{C(x)} : 
\text{$\pi(v) \in \Omega$ and there is no $\tilde v \in \overline{C(x)}$ s.t. $\tilde v \ll v$} \}.
    \end{align}
The sets $E(x)$ and $C(x)$ are illustrated in Figure \ref{fig_E} below.

We will assume that $\Omega$ has the following form in some local coordinate map $F:U\subset\R^{1+n}\to M$,
\begin{itemize}
\item[(F)] $\Omega = F((-1,1) \times B(0,\delta))$ for a small $\delta>0$ and that the paths $s\mapsto (s,a)$ are timelike and future pointing for all $s\in (-1,1)$ and $a \in B(0,\delta)$. 
\end{itemize}
Note that the abstract condition (F) is satisfied for example for the set $\Omega=\Omega_\o$ and $\Omega=\Omega_\i$ given by \eqref{foliation}.
\begin{lemma}\label{lem_E_rho}
Let an open set $\Omega \subset M$ satisfy {\rm (F)} and let a point $x \in M$ satisfy $x \notin J^-(F(\{-1\}\times B(0,\delta)))$.
Then
    \begin{align}\label{E_alternative}
E(x) = \{\beta_{x,\xi}(s) : \xi \in L_x^+ M,\ 0 \leq s \leq \rho(x,\xi),\  \gamma_{x,\xi}(s) \in \Omega\}.
    \end{align}
\end{lemma}
\begin{proof}
Denote by $E_0$ the right-hand side of (\ref{E_alternative}), let $v \in E_0$ and write $y=\pi(v)$. If $x=y$ then clearly $v \in E(x)$. 
Let us now consider the case that there are $\xi \in L_x^+ M$ and $0 < s\leq \rho(x,\xi)$ such that
$v=\beta_{x,\xi}(s)$. Clearly, $v\in C(x)$ and $\gamma_{x,\xi}$ is optimizing from $x$ to $y$.
To get a contradiction suppose that $v\notin E(x)$.
Then there exists $\tilde v\in C(x)$ such that 
$\tilde v \ll v$, and writing $\tilde y = \pi(\tilde v)$,
we are lead to the contradiction $x \le \tilde y \ll y$
with $\gamma_{x,\xi}$ being optimizing from $x$ to $y$.
This shows that $v \in E(x)$.

On the other hand, let $v\in E(x)$ and $y=\pi(v)$. If $x=y$ then clearly $v \in E_0$.
Let us now consider the case that there are 
are $\xi\in L_x^+ M$ and $r > 0$ such that
$v=\beta_{x,\xi}(r)$. Working in the local coordinates (F), there is
$F(s_0,a) \in (-1,1) \times B(0,\delta)$ such that $y = F(s_0,a)$.
We define $\mu_a : [-1,1] \to M$ by $\mu_a(s) = F(s,a)$.
Now $x < y \le \mu_a(1)$ and we assumed also $x \not\le \mu_a(-1)$. Lemma \ref{lem_optim_exists}
implies then that $x \in \overline{\mu_a}$ or there is an optimizing geodesic from $x$ to $(s_1, a)$ for some $s_1 \in (-1,1]$. 
The former case is not possible, since $x < y$ and $x,y \in \overline{\mu_a}$ imply $x \ll y$, and this is a contradiction with $v \in E(x)$.
Hence there is an optimizing geodesic $\tilde \gamma$ from $x$ to $X(s_1, a)$ for some $s_1 \in (-1,1]$. 
We have $s_0 = s_1$ since $s_0 < s_1$ is a contradiction with $\tilde \gamma$ being optimizing, and $s_0 > s_1$ is a contradiction with $v \in E(x)$. Hence $\tilde \gamma$ is optimizing from $x$ to $y$, and so is $\gamma_{x,\xi}$.
Moreover, $\gamma_{x,\xi}$ being optimizing from $x$ to $y = \pi(v)$ implies that $v \in E_0$.
\end{proof}

\begin{lemma}\label{lem_E_pi}
Let $\Omega \subset M$ and $x \in M$ be as in Lemma \ref{lem_E_rho}.
Then
    \begin{align}\label{E_alternative3}
E(x) = \{(y,\eta) \in L^+ \Omega :\ \text{there is $\epsilon > 0$ such that $\gamma_{y,\eta}(s) \in \pi(E(x))$}&
\\\notag\text{for all $s \in [0, \epsilon]$ or for all $s \in [-\epsilon,0]$}&\}.
    \end{align}
\end{lemma}
\begin{proof}
Denote by $E_0$ the right-hand side of (\ref{E_alternative3}) and let $(y,\eta) \in E(x)$.
If $x = y$ then $\gamma_{y,\eta}(s) \in \pi(E(x))$
for $s \geq 0$ close to zero, since $\Omega$ is open and (\ref{E_alternative}) holds.
Analogously, if $x \ne y$ then
$\gamma_{y,\eta}(s) \in \pi(E(x))$ for $s \leq 0$ close to zero.
Hence $E(x) \subset E_0$.

We show that $(y,\eta) \in E_0$ implies $(y,\eta) \in E(x)$ only in the case that there is $\epsilon > 0$ such that $\gamma_{y,\eta}(s) \in \pi(E(x))$ for all $s \in [-\epsilon,0]$. The other case is analogous. We write $\tilde y = \gamma_{y,\eta}(-\epsilon)$. Then (\ref{E_alternative}) implies that $\tau(x,y) = 0$
and that there is $\xi \in L^+_x M$ such that $\tilde y = \gamma_{x,\xi}(s)$ for some $s \in (0,\rho(x,\xi)]$. The path from $x$ to $y$ along $\gamma_{x,\xi}$ from $x$ to $\tilde y$ and along $\gamma_{y,\eta}$ from $\tilde y$ to $y$ is a null pregeodesic since $\tau(x,y) = 0$. Therefore $(y,\eta) = \beta_{x,\xi}(r)$ for some $r \in (0,\rho(x,\xi)]$ and $(y,\eta) \in E(x)$.
\end{proof}

\begin{lemma}\label{lem_E_future}
Let $\Omega \subset M$ and $x \in M$ be as in Lemma \ref{lem_E_rho}.
Suppose that a set $C \subset \Omega$ satisfies 
$\pi(E(x)) \subset C \subset J^+(x)$.
Then  
    \begin{align}\label{E_alternative2}
\pi(E(x)) &= \{y \in C : \text{ 
there is no $\tilde y \in C$ s.t. $\tilde y \ll y$ in $\Omega$} \}.
    \end{align}
\end{lemma}
\begin{proof}
Denote by $E_0$ the right-hand side of (\ref{E_alternative2}) and let $y \in \pi(E(x))$. 
If $x = y$ then there is no $\tilde y \in C$ such that $\tilde y \ll y$ since $x \le \tilde y$ by the assumption $C \subset J^+(x)$. Hence $y \in E_0$. 
Suppose now that $y \ne x$.
By Lemma \ref{lem_E_rho}
there are $\xi \in L_x^+ M$ and $0 < s\leq \rho(x,\xi)$ such that
$y=\gamma_{x,\xi}(s)$. Moreover, $y \in C$ and $\gamma_{x,\xi}$ is optimizing from $x$ to $y$.
To get a contradiction suppose that there exists $\tilde y\in C$ such that 
$\tilde y \ll y$. As $C \subset J^+(x)$,
we are lead to the contradiction $x \le \tilde y \ll y$
with $\gamma_{x,\xi}$ being optimizing from $x$ to $y$.
This shows that $y \in E_0$.

Let $y \in E_0$. If $x = y$ then $y \in \pi(E(x))$. Let us now assume that $x \ne y$.
In the local coordinates (F) we may write $y = F(s_0,a)$ for some $(s_0,a) \in (-1,1) \times B(0,\delta)$ and define the path $\mu_a(s) = F(s,a)$. There holds $x < y \le \mu_a(1)$ since $C \subset J^+(x)$.
As in the proof of Lemma \ref{lem_E_rho}, there is an optimizing geodesic $\tilde \gamma$ from $x$ to $\tilde y := F(s_1, a)$ for some $s_1 \in (-1,1]$. 
Lemma \ref{lem_E_rho} implies that $\tilde y \in \pi(E(x))$.
Now $s_0 < s_1$ is a contradiction with $\tilde \gamma$ being optimizing since $x \le y \ll \tilde y$ in this case, and $s_0 > s_1$ is a contradiction with $y \in E_0$ since 
$\tilde y \in \pi(E(x)) \subset C$ and $\tilde y \ll y$ in this case.
Therefore $s_0 = s_1$ and $y = \tilde y \in \pi(E(x))$.
\end{proof}

\begin{lemma}\label{lem_E_future2}
Let $\Omega \subset M$, $x \in M$
and $C \subset \Omega$ be as in Lemma~\ref{lem_E_future}.
Then the path $\mu_a : [-1,1] \to M$, defined by $\mu_a(s) = F(s,a)$ with $a \in B(0,\delta)$ in the local coordinates (F), satisfies
    \begin{align}\label{f_alternative}
f_{\mu_a}^+(x) &= \inf \{s \in [-1,1] : \text{
$F(s,a) \in C$ or $s=1$}\}.
    \end{align}
\end{lemma}
\begin{proof}
Suppose for the moment that $x \le (1,a)$.
We write $s_0$ and $s_1$ for the left and right-hand sides of (\ref{f_alternative}), respectively.
Lemma \ref{lem_optim_exists} implies that $s_0 > -1$ and that either
$x = F(s_0, a)$ or 
there is an optimizing geodesic from $x$ to $F(s_0, a)$. 

Case $x = F(s_0, a)$ and $s_0 < 1$. Then $x \in \Omega$
and $x \in \pi(E(x)) \subset C$. Now $s_1 \leq s_0$ 
and $s_1 < s_0$ is a contradiction with $C \subset J^+(x)$.
Hence $s_0 = s_1$ in this case. 

Case $x = F(s_0, a)$ and $s_0 = 1$. Then $J^+(x) \cap \Omega = \emptyset$ and also $s_1 = 1$. 

Case that $s_0 < 1$ and there is an optimizing geodesic $\gamma$ from $x$ to $y_0 := F(s_0, a)$. 
As $\gamma$ is optimizing from $x$ to $y_0$, Lemma \ref{lem_E_rho} implies that $y_0 \in \pi(E(x)) \subset C$. Hence $s_1 \leq s_0$.
Moreover, $y_1 := F(s_1, a) \in \overline C$
and $x \le y_1$ since $C \subset J^+(x)$ and the causal relation $\le$ is closed.
Finally $s_1 < s_0$ leads to the contradiction $x \le y_1 \ll y_0$ with $\gamma$ being optimizing from $x$ to $y_0$.

Case that there is an optimizing geodesic $\gamma$ from $x$ to $F(1, a)$ or $x \not\le F(1,a)$.
Then $x \not\le F(s,a)$ for all $-1 < s < 1$.
In other words, $F(s, a) \notin J^+(x)$ for all $-1 < s < 1$,
and $s_0 = s_1 = 1$.
\end{proof}

\subsection{On the span of three lightlike vectors}
\label{nullvectors_sec}

We start with a simple lemma about the linear span of two lightlike vectors on Lorentzian manifolds.
\begin{lemma}\label{span_two_null}
Let $y$ be a point on a Lorentzian manifold $(M,g)$ of dimension $1+n$ with $n \geq 2$. Let $\xi_1,\xi_2,\xi_3 \in T_y M\setminus 0$ be lightlike vectors such that they are not all multiples of each other. Then, 
$$ c_1\xi_1+c_2\xi_2+c_3\xi_3=0 \implies c_1=c_2=c_3=0.$$
\end{lemma}

\begin{proof}
It suffices to work in the normal coordinate system at the point $y$ where the metric evaluated at the point $y$ is the Minkowski metric. After an scaling and without loss of generality we can write $\xi_j= (1,\xi'_j)$, $j=1,2,3$ for vectors $\xi'_j \in \R^{n}$ that satisfy $|\xi'_j|=1$. Thus, it follows that
$$ c_1+c_2+c_3=0,\quad \text{and}\quad c_1\xi'_1+c_2\xi'_2+c_3\xi'_3=0.$$
Therefore $|c_j\xi'_j+c_k\xi'_k|=|c_j||\xi'_j|+|c_k||\xi'_k|$, for all $j,k=1,2,3$. Since the vectors are not all parallel, it follows that $c_1=c_2=c_3=0$.
\end{proof}

Next, we consider the linear span of three lightlike vectors. The following lemma is taken from \cite[Lemma 1]{CLOP}.
\begin{lemma}\label{Lauri's Lemma}
Let $y$ be a point on a Lorentzian manifold $(M,g)$ of dimension $1+n$ with $n \geq 2$. Let $\xi_1, \eta \in T_y M \setminus 0$ be lightlike. In any neighborhood of $\xi_1$ in $T_y M$, there exist two lightlike vectors $\xi_2, \xi_3$ such that $\eta$ is in $\linspan(\xi_1, \xi_2, \xi_3)$.  
\end{lemma}
We will also need a variation of the above lemma as follows.
\begin{lemma}\label{lem_linalg_pert}
Let $x \in M$, $\xi_0, \xi_1 \in L^+_x M$
and let $U \subset L^+_x M$ be a neighborhood of $\xi_1$.
Suppose that $\xi_0 \notin \linspan(\xi_1)$.
There are a neighborhood $V \subset T_x M$ of $\xi_0$
and $\xi_2 \in U$ 
such that for any $\eta \in V$
there is $\xi_3 \in U$ such that $\eta \in \linspan(\xi_1,\xi_2,\xi_3)$ and $\eta \notin \linspan(\xi_j)$, $j=1,2,3$.
\end{lemma}
\begin{proof}
We choose normal coordinates centred at $y$. Then $g$ is the Minkowski metric on the fibre $T_y M = \R^{1+n}$.
The statement is invariant with respect to non-vanishing rescaling of $\xi_0$ and $\xi_1$, and we assume without loss of generality that $\xi_j = (1,\xi_j')$ with $\xi_j'$ a unit vector in $\R^n$, $j=0,1$.
We choose an orthonormal basis $e_1,\dots,e_n$ of $\R^n$ such that $e_1 = \xi_1'$ and $\xi_0' \in \linspan(e_1,e_2)$.
Then in this basis it holds for some $a,b \in \R$ that  
    \begin{align}\label{xi_eta_form}
\xi_1 = (1,1,0,\underbrace{0,\dots,0}_{\text{$n-2$ times}}), \quad \xi_0 = (1, a, b, \underbrace{0,\dots,0}_{\text{$n-2$ times}}).
    \end{align}
Choose a small enough $r > 0$ so that both the vectors
    \begin{align*}
\xi_+ = (1, \sqrt{1-r^2}, r, \underbrace{0,\dots,0}_{\text{$n-2$ times}}),
\quad 
\xi_- = (1, \sqrt{1-r^2}, -r, \underbrace{0,\dots,0}_{\text{$n-2$ times}}),
    \end{align*}
are in $U$ and $\xi_0 \notin \linspan(\xi_\pm)$. We set $\xi_2 = \xi_+$.

Let $\delta = (\delta_0, \delta_1, \delta_2) \in \R^3$ and $\epsilon \in \R^{n-2}$ 
be close to the respective origins.
Consider the following perturbation of $\xi_0$
    \begin{align*}
\eta = (1 + \delta_0, a + \delta_1, b + \delta_2, \epsilon),
    \end{align*}
 and $\xi_3$ of the form
    \begin{align*}
\xi_3 = (1,  \sqrt{1-r^2 - c^2 |\epsilon|^2}, -r, c\epsilon),
    \end{align*}
where $c \in \R$. 
The system $c_1 \xi_1 + c_2 \xi_2 + c_3 \xi_3 = \eta$ for $c_1, c_2, c_3 \in \R$ reads in matrix form 
    \begin{align*}
\begin{pmatrix}
 1 & 1 & 1 & 1 +\delta_0 \\
 1 & \sqrt{1-r^2} & \sqrt{1-r^2 - c^2 |\epsilon|^2} & a + \delta_1 \\
 0 & r & -r & b + \delta_2 \\
 0 & 0 & c\epsilon & \epsilon \\
\end{pmatrix}
    \end{align*}
and two steps of the Gaussian elimination algorithm reduces this to 
    \begin{align*}
\begin{pmatrix}
 1 & 1 & 1 & 1 +\delta_0 \\
 0 & 1 & w & z \\
 0 & 0 & x & y \\
 0 & 0 & c\epsilon & \epsilon \\
\end{pmatrix}
    \end{align*}
where the specific form of $y = y(\delta,\epsilon)$, $z = z(\delta,\epsilon)$ and $w=w(c,\epsilon)$ is not important to us, and 
    \begin{align*}
x = x(c,\epsilon) = \frac{r}{\sqrt{1-r^2}-1} \left(2-\sqrt{1-r^2}-\sqrt{1-r^2-c^2|\epsilon|^2}\right).
    \end{align*}
As $x \ne 0$, the above system has a solution if and only if 
    \begin{align}\label{linalg_aux}
x(c,\epsilon)-cy(\delta,\epsilon) = 0.
    \end{align}

To get a contradiction, suppose that $y(0,0) = 0$. 
Then $\xi_0 \in \linspan(\xi_1, \xi_2)$, and as $\xi_0, \xi_1$ and $\xi_2$ are all lightlike Lemma~\ref{span_two_null} applies to show that $\xi_0 \in \linspan(\xi_1)$ or $\xi_0 \in \linspan(\xi_2)$. 
But $\xi_0 \notin \linspan(\xi_j)$ for both $j=1,2$, a contradiction.

We write $F(c,\delta,\epsilon)$ for the left-hand side of (\ref{linalg_aux}), and $y_0 = y(0,0)$ and $x_0 = x(c,0)$, the latter being independent from $c$.
Setting $c_0 = x_0 / y_0$, we see that $F(c_0, 0, 0) = 0$
and $\p_c F(c_0, 0, 0) = -y_0 \ne 0$. By the implicit function theorem there is a neighborhood $V_0 \subset \R^{3+(n-2)}$ of the origin and a smooth map $h : V_0 \to \R$ such that $h(0,0) = c_0$ and $F(h(\delta, \epsilon), \delta, \epsilon) = 0$ for all $(\delta, \epsilon) \in V_0$.
As $\xi_3 = \xi_-$ when $\epsilon = 0$, by making $V_0$ smaller if necessary, we have that $\xi_3$, with $c=h(\delta, \epsilon)$, is in $U$ for all $(\delta, \epsilon) \in V_0$.

By making $V_0$ smaller if necessary, we can also guarantee that $\eta \notin \linspan(\xi_j)$ for $j=1,2,3$, since $\xi_0 \notin \linspan(\xi_j)$, $j=1,2$, and $\xi_0 \notin \linspan(\xi_-)$. 
\end{proof}

\section{Recovery of earliest observation sets and proof of Theorem~\ref{t2}}
\label{earliest_sec}

The aim of this section is to prove Theorem~\ref{t2}. As in the hypothesis of the theorem, {\mltext we consider} time-like paths $\mu_\i:[t_0^-,t_0^+]\to M$ and $\mu_\o:[s_0^-,s_0^+]\to M$ and define the source and observe regions $\Omega_\i$ and $\Omega_\o$ by \eqref{foliation}. We will assume as in the hypothesis of the theorem that $\delta$ is sufficiently small so that \eqref{fol_cond} holds. 
\begin{remark}
\label{rmk_fol}
Throughout the remainder of the paper and without loss of generality, we will assume that $s_1^\pm = \pm 1$ and that $t_1^\pm=\pm 1$. With this convenient choice of notation, we have 
$$\Omega_\i=F_\i((-1,1)\times B(0,\delta)),\quad \text{and} \quad \Omega_\o=F_\o((-1,1)\times B(0,\delta)),$$ 
and therefore the abstract foliation condition (F), that was studied in the previous section, holds for these sets.
\end{remark}

We recall from Definition~\ref{def: good relation} that $\rel \in L^+ \Omega_\o \times (L^+ \Omega_\i)^3$
is a three-to-one scattering relation if it has following two properties:
\begin{itemize}
\item[(R1)] If $(v_0,v_1, v_2,v_3) \in \rel$,
then there is $y \in \overleftarrow{\;\gamma_{v_0}} \cap \bigcap_{j=1}^3  \overrightarrow{\gamma_{v_{j}}}$.
\item[(R2)] Assume that $\gamma_{v_{j}}$, $j=0,1,2,3$, are distinct
and there exists $y \in \overleftarrow{\;\gamma_{v_{0}}} \cap \bigcap_{j=1}^3  \overrightarrow{\gamma_{v_{j}}}.$ Moreover, assume that  
$y=\gamma_{v_{0}}(s_0)$ with $s_0\in (-\rho(v_0),0]$ and $y=\gamma_{v_{j}}(s_j)$  for all $j=1,2,3$,
with
 $s_j\in [0,\rho(v_j))$.
Denote $\xi_j=\dot \gamma_{v_{j}}(s_j)$  for $j=0,1,2,3$
and assume that $\xi_0 \in \linspan(\xi_1, \xi_2, \xi_3)$. Then, 
 it holds that $(v_{0},v_{1}, v_{2},v_{3}) \in \rel$.
\end{itemize}

\subsection{Lower and upper bounds for conical pieces}

We define a conical piece associated to a three-to-one scattering relation $\rel$
and $v_{1}, v_{2} \in L^+\Omega_\i$ by
    \begin{align*}
\CP(v_{1}, v_{2})
&= \{ v_{0} : 
\text{ there is $v_{3} \in L^+\Omega_\i$ such that
$(v_{0},v_{1}, v_{2},v_{3}) \in \rel$} \}.
    \end{align*}

\begin{lemma}\label{lem_smooth_piece}
Let $v_{1} \in L^+ \Omega_\i$, $v_{0} \in L^+ \Omega_\o$ and write $\pi(v_{1}) = x$ and $\pi(v_{0}) = z$.
Suppose that $\gamma_{v_{1}}$ is optimizing from $x$ to a point $y$ in $M$ and 
that $\gamma_{v_{0}}$ is optimizing from $y$ to $z$.
Suppose furthermore that $\gamma_{v_{1}}$ and $\gamma_{v_{0}}$ do not intersect at $x$ or at $z$. Then there is a vector $v_{2} \in L^+ \Omega_\i$, arbitrarily close to $v_{1}$, and a set
$C \subset \CP(v_{1}, v_{2})$ 
such that $C$ is a neighborhood of $v_{0}$ in $C(y)$.
\end{lemma}
\begin{proof}
Let $s_{j} \in \R$ satisfy $\gamma_{v_{j}}(-s_j) = y$, $j=0,1$, and write 
$\xi_j = \dot \gamma_{v_{j}}(-s_j)$.
As $\gamma_{v_{1}}$ is optimizing from $x$ to $y$
and $\gamma_{v_{0}}$ is optimizing from $y$ to $z$,
these two geodesics do not 
intersect at any point $\tilde y \ne y$ satisfying $x < \tilde y < z$.
We have also assumed that they do not intersect at $x$ or $z$. Hence there is such a neighborhood $V_0 \subset L_y^+ M$ of $\xi_0$ that $y$ is the unique point satisfying $y \in \overleftarrow{\;\gamma_{\tilde v_{0}}} \cap \overrightarrow{\gamma_{v_{1}}}$ where $\tilde v_{0} = \beta_{y,\eta}(s_0)$ and $\eta \in V_0$.

Observe that $\xi_0 \notin \linspan(\xi_1)$ since $\gamma_{v_{1}}$ and $\gamma_{v_{0}}$ do not intersect at $x$. Let $U \subset L_y^+ M$ be a neighborhood of $\xi_1$ such that $\gamma_{y,\xi}(s_1) \in \Omega_\i$ for all $\xi \in U$.
By Lemma \ref{lem_linalg_pert} there are a neighborhood $V \subset V_0$ of $\xi_0$
and $\xi_2 \in U$
such that for any $\eta \in V$
there is $\xi_3 \in U$ such that $\eta \in \linspan(\xi_1,\xi_2,\xi_3)$ and $\eta \notin \linspan(\xi_j)$, $j=1,2,3$.
Writing
 $v_{j} = \beta_{y,\xi_j}(s_j)$, $j=2,3$, we have $(\tilde v_{0}, v_{1}, v_{2}, v_{3}) \in \rel$ for $\eta \in V$, due to (R2).
To conclude, 
we use the fact that the image of $V$ under the map $\eta \mapsto \beta_{y, \eta}(s_0)$
is a smooth submanifold of dimension $n$ in $T M$.
\end{proof}

\begin{lemma}\label{lem_CP_upper}
Let $v_{1}, v_{2} \in L^+\Omega_\i$
satisfy $\overline{\gamma_{v_{1}}} \ne \overline{\gamma_{v_{2}}}$.
We write 
    \begin{align}\label{def_K}
K = J^+(\overline{\Omega_\i}) \cap J^-(\overline{\Omega_\o}).
    \end{align}
Then the set 
$\mathcal F = \overrightarrow{\gamma_{v_{1}}} \cap \overrightarrow{\gamma_{v_{2}}} \cap K$
is finite, and 
    \begin{align*}
\CP(v_{1}, v_{2}) \subset \bigcup_{x \in \mathcal F} \overline{C(x)}.
    \end{align*}
\end{lemma}
\begin{proof}
Observe that $K$ is compact since both $\Omega_\i$ and $\Omega_\o$ are bounded.
If $\mathcal F$ is not finite, then it has an accumulation point, and using Lemma \ref{lem_s_proper} we obtain the contradiction $\overline{\gamma_{v_{1}}} = \overline{\gamma_{v_{2}}}$.
The second claim follows immediately from (R1).
\end{proof}

\subsection{Relating earliest observation sets to a three-to-one scattering relation}

Recall that the observation set $\Omega_\o$ satisfies (F) and take $\Omega = \Omega_\o$
in the definition (\ref{def_E}) of 
the earliest observation sets $E(y)$, $y \in M$.
We will next relate $E(y)$ to a set constructed from $\rel$.
To this end, define the set $E(v,w)$ for $v,w \in L^+\Omega_\i$ as follows: let $C(v,w)$ be the closure in $T \Omega_\o$ of the union $\bigcup_{C \in \mathcal C(v,w)} C$
where 
    \begin{align*}
\mathcal C(v,w) = \{ C :
&\text{ $C \subset \Omega_\o$ is a smooth manifold of dimension $n$ s.t.}
\\&\text{ $C \subset \CP(v,\tilde v) \cap \CP(w,\tilde w)$ for some $\tilde v, \tilde w \in L^+\Omega_\i$}
\\&\text{ satisfying 
$\overline{\gamma_v} \ne \overline{\gamma_{\tilde v}}$
and $\overline{\gamma_w} \ne \overline{\gamma_{\tilde w}}$} \},
    \end{align*}
then we set 
    \begin{align*}
\tilde E(v,w) &= \{ u \in C(v,w) : \text{
there is no $\tilde u \in C(v,w)$ s.t. $\tilde u \ll u$ in $\Omega_\o$}\},
    \end{align*}
and
    \begin{align*}
E(v,w) = \{(z,\zeta) :\ \text{there is $\epsilon > 0$ such that $\gamma_{z,\zeta}(s) \in \pi(\tilde E(v,w))$}&
\\\text{for all $s \in [0, \epsilon]$ or for all $s \in [-\epsilon,0]$}&\}.
    \end{align*}

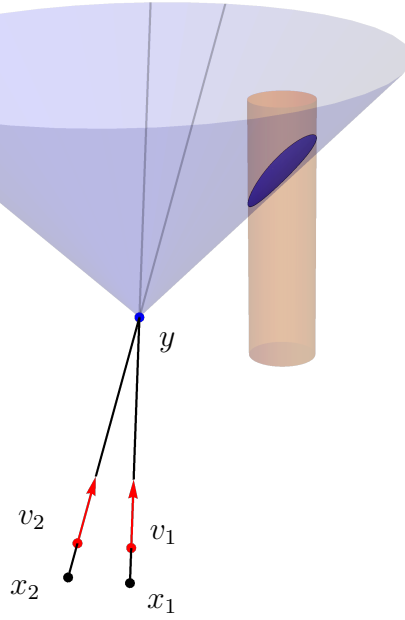
\begin{figure}
\def\svgwidth{9cm}
\begingroup%
  \makeatletter%
  \providecommand\color[2][]{%
    \errmessage{(Inkscape) Color is used for the text in Inkscape, but the package 'color.sty' is not loaded}%
    \renewcommand\color[2][]{}%
  }%
  \providecommand\transparent[1]{%
    \errmessage{(Inkscape) Transparency is used (non-zero) for the text in Inkscape, but the package 'transparent.sty' is not loaded}%
    \renewcommand\transparent[1]{}%
  }%
  \providecommand\rotatebox[2]{#2}%
  \newcommand*\fsize{\dimexpr\f@size pt\relax}%
  \newcommand*\lineheight[1]{\fontsize{\fsize}{#1\fsize}\selectfont}%
  \ifx\svgwidth\undefined%
    \setlength{\unitlength}{575.9999856bp}%
    \ifx\svgscale\undefined%
      \relax%
    \else%
      \setlength{\unitlength}{\unitlength * \real{\svgscale}}%
    \fi%
  \else%
    \setlength{\unitlength}{\svgwidth}%
  \fi%
  \global\let\svgwidth\undefined%
  \global\let\svgscale\undefined%
  \makeatother%
  \begin{picture}(1,1)%
    \lineheight{1}%
    \setlength\tabcolsep{0pt}%
    \put(0,0){\includegraphics[width=\unitlength,page=1]{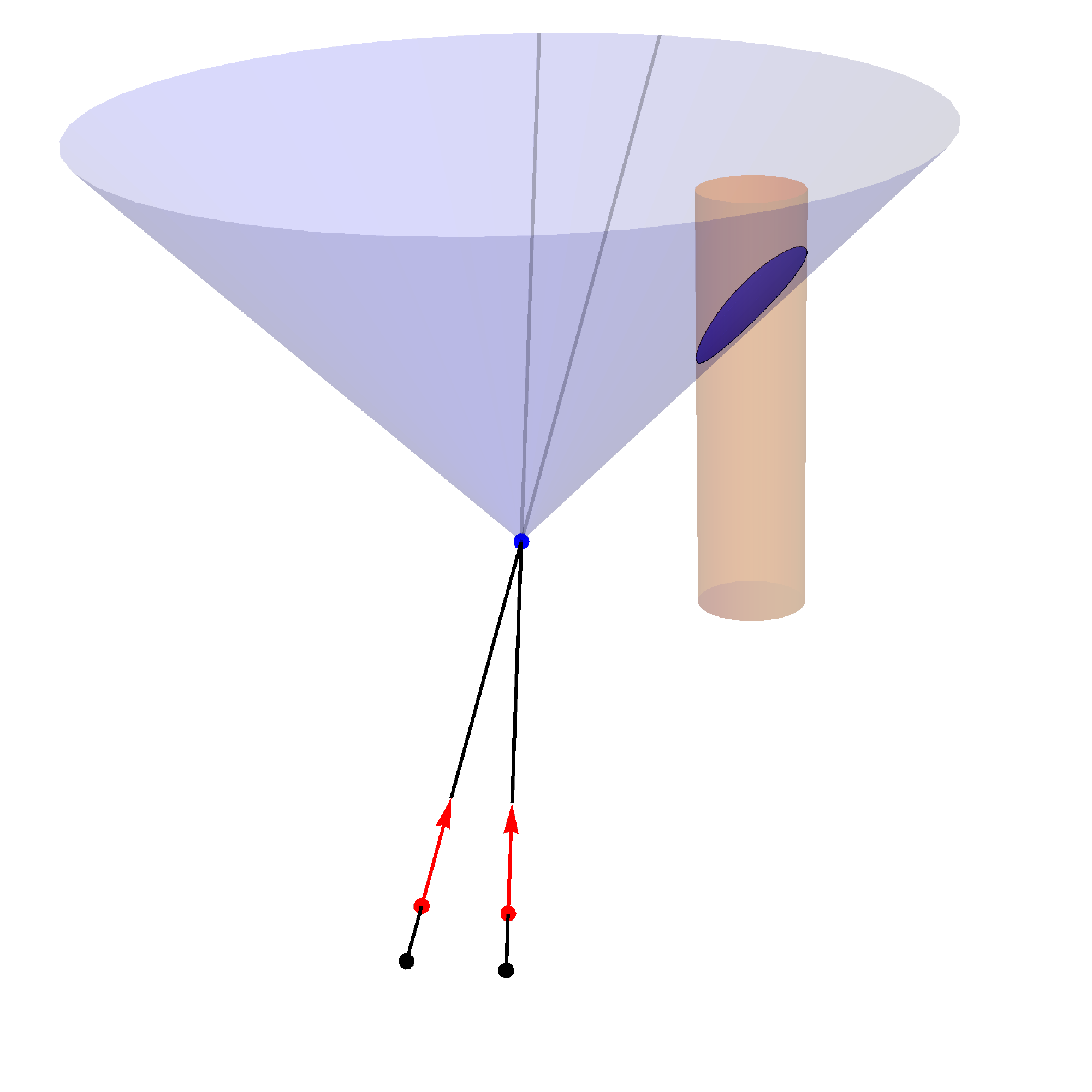}}%
    \put(0.48845318,0.07504559){\color[rgb]{0,0,0}\makebox(0,0)[lt]{\lineheight{1.25}\smash{\begin{tabular}[t]{l}$x_1$\end{tabular}}}}%
    \put(0.49255546,0.17737897){\color[rgb]{0,0,0}\makebox(0,0)[lt]{\lineheight{1.25}\smash{\begin{tabular}[t]{l}$v_1$\end{tabular}}}}%
    \put(0.28600404,0.09564792){\color[rgb]{0,0,0}\makebox(0,0)[lt]{\lineheight{1.25}\smash{\begin{tabular}[t]{l}$x_2$\end{tabular}}}}%
    \put(0.29786187,0.19908741){\color[rgb]{0,0,0}\makebox(0,0)[lt]{\lineheight{1.25}\smash{\begin{tabular}[t]{l}$v_2$\end{tabular}}}}%
    \put(0.50636037,0.46291526){\color[rgb]{0,0,0}\makebox(0,0)[lt]{\lineheight{1.25}\smash{\begin{tabular}[t]{l}$y$\end{tabular}}}}%
  \end{picture}%
\endgroup%

\caption{
Geometric setting of Lemma \ref{lem_E_good} in the $1+2$-dimensional Minkowski space. The time axis is vertical.
Set $\Omega_\o$ is the orange solid cylinder,
projection $\pi(C(y))$ is the light blue cone,
and $\pi(E(y))$ is drawn in dark blue.
Point $y$ in blue, points $x_1, x_2$ in black, and vectors $v_1, v_2 \in L^+ \Omega_\i$ in red. Geodesics $\gamma_{v_1}$ and $\gamma_{v_2}$ are the black lines. 
}\label{fig_E}
\end{figure}

The following lemma describes the basic idea that we will use to construct earliest observation sets given a three-to-one scattering relation $\rel$. The geometric setting of the lemma is shown in Figure \ref{fig_E}.

\begin{lemma}\label{lem_E_good}
Let $v_1, v_2 \in L^+ \Omega_\i$
and suppose that there are $x_j < \pi(v_j)$
such that $\gamma_{v_j}$ is optimizing from $x_j$ to a point $y \in \overrightarrow{\gamma_{v_1}} \cap \overrightarrow{\gamma_{v_2}}$ for $j=1,2$.
Suppose, furthermore, that $\overline{\gamma_{v_1}} \ne \overline{\gamma_{v_2}}$ and that $y \notin J^-(F_\o(\{-1\}\times B(0,\delta)))$.
Then $E(y) \subset C(v_1,v_2)$, 
$\pi(C(v_1, v_2)) \subset J^+(y)$
and $E(y) = E(v_1, v_2)$.
\end{lemma}
\begin{proof}
Let $u \in E(y)$. We will show that $u \in C(v_1,v_2)$. 
Note that $\gamma_{u}$ is optimizing from $y$ to $z := \pi(u)$ by Lemma~\ref{lem_E_rho}. 
Suppose for the moment that neither $\gamma_{v_1}$ nor $\gamma_{v_2}$ intersects $\gamma_u$ at $z$. 
As $\gamma_{v_j}$ is optimizing from $x_j$ to $y$, the geodesics $\gamma_u$ and $\gamma_{v_j}$ can not intersect at $\pi(v_j)$, and
Lemma \ref{lem_smooth_piece} implies that 
there are $\tilde v_j \in L^+ \Omega_\i$ and
$C_j \subset \CP(v_j, \tilde v_j)$ 
such that $C_j$ is a neighborhood of $u$ in $C(y)$.
As $C(y)$ is a smooth manifold of dimension $n$,
the intersection $C_1 \cap C_2$ is a smooth manifold of dimension $n$ containing $u$.
Hence $u \in C(v_1,v_2)$.

Let us now consider the case that $\gamma_{v_1}$ or $\gamma_{v_2}$ intersects $\gamma_u$ at $z$. Choose $\eta \in L_y^+ M$ and $t > 0$ so that $\overline{\gamma_{y,\eta}} = \overline{\gamma_u}$ and $\gamma_{y,\eta}(t) = z$.
We have $t \leq \rho(y,\eta)$ and there are $\eta_k \in L_y^+ M$ and $0 < t_k < \rho(y,\eta_k)$ such that writing $u_k = \beta_{y,\eta_k}(t_k)$ and $z_k = \pi(u_k)$ there holds $u_k \to u$ and
neither $\gamma_{v_1}$ nor $\gamma_{v_2}$ intersects $\gamma_{u_k}$ at $z_k$.
The argument above implies that $u_k \in C(v_1,v_2)$. As $C(v_1,v_2)$ is closed, also $u \in C(v_1,v_2)$. 

Let $u \in C(v_1, v_2)$. We will show that $\pi(u) \in J^+(y)$.
Consider first the case that $u \in C$ for some $C \in \mathcal C(v_1,v_2)$ and $\overline{\gamma_{v_j}} \ne \overline{\gamma_{u}}$ for $j=1,2$.
Then there are $y_j \in  \overleftarrow{\;\gamma_{u}} \cap\overrightarrow{\gamma_{v_j}}$.
As $x_j < \pi(v_j) \le y_j$, Lemma~\ref{lem_shortcut2} implies that either $y=y_1=y_2$ or at least one of $y_1$ and $y_2$ satisfies $y < y_j$. 
In both the cases $y \le \pi(u)$, that is, $\pi(u) \in J^+(y)$.
Consider now the case that there is sequence $u_k$, $k \in \N$,
such that $u_k \in C_k$ for some $C_k \in \mathcal C(v_1,v_2)$
and $u_k \to u$. 
The sequence $u_k$ can be chosen so that also 
$\overline{\gamma_{v_j}} \ne \overline{\gamma_{u_k}}$ for $j=1,2$ and $k=1,2,\dots$. 
We obtain $y \le \pi(u)$ also in this case since the relation $\le$ is closed. 

We have shown, in particular, that $\pi(E(y)) \subset \pi(C(v_1,v_2)) \subset J^+(y)$.
Lemma \ref{lem_E_future} implies now that $\pi(\tilde E(v_1,v_2)) = \pi(E(y))$. 
Finally $E(y) = E(v_1, v_2)$ follows immediately from Lemma \ref{lem_E_pi}.
\end{proof}

\subsection{Local test for optimality before intersection}

In Lemma \ref{lem_E_good} the geodesic $\gamma_{v_1}$ needs to be optimizing from $x_1$ to $y$. 
We will give a construction that allows us to tell apart the optimizing cases
and non-optimizing, but close to optimizing cases, given $\rel$. We begin with an auxiliary lemma.

\begin{lemma}\label{lem_C_two_geods}
Let $v_1, v_2 \in L^+ \Omega_\i$ and let $u_1 \in C \in \mathcal C(v_1, v_2)$. Then there is $u_2 \in C$ 
such that 
$\overline{\gamma_{u_1}} \ne \overline{\gamma_{u_2}}$
and that 
$\overleftarrow{\;\gamma_{u_1}} \cap \overleftarrow{\;\gamma_{u_2}} \cap \overrightarrow{\gamma_{v_j}} \ne \emptyset$ for both $j=1,2$.
\end{lemma}
\begin{proof}
As $C$ is a smooth manifold of dimension $n$ there is $u_2$ satisfying $\overline{\gamma_{u_1}} \ne \overline{\gamma_{u_2}}$ in any neighborhood of $u_1$ in $C$. Lemma \ref{lem_CP_upper} implies that for some $\tilde v_1 \in L^+ \Omega_\i$ and finite $\mathcal F \subset \overrightarrow{\gamma_{v_1}}$ there holds
    \begin{align*}
C \subset \CP(v_1, \tilde v_1) \subset \bigcup_{x \in \mathcal F} \overline{C(x)}.
    \end{align*}
Then Lemma \ref{lem_flowout_int} implies that 
$\mathcal F \cap \overleftarrow{\;\gamma_{u_1}} \cap \overleftarrow{\;\gamma_{u_2}}\ne \emptyset$ when $u_2 \in C$ is close enough to $u_1$.
The proof that $\overleftarrow{\;\gamma_{u_1}} \cap \overleftarrow{\;\gamma_{u_2}} \cap \overrightarrow{\gamma_{v_2}} \ne \emptyset$ is analogous.
\end{proof}

\begin{lemma}\label{lem_optim}
Let $v_1, v_2 \in L^+ \Omega_\i$ and $x_1 \in M$ be as in Lemma \ref{lem_E_good}.
Then there do not exist $\tilde v_1, \tilde v_2 \in L^+ \Omega_\i$ and non-empty $C \in \mathcal C(\tilde v_1, \tilde v_2)$
such that $C \subset E(v_1, v_2)$ and $x_1 \ll \tilde x_1$ for some $\tilde x_1 \in \overleftarrow{\;\gamma_{\tilde v_1}}$.
\end{lemma}
\begin{proof}
To get a contradiction we suppose that there are 
$\tilde v_1, \tilde v_2 \in L^+ \Omega_\i$ and non-empty $C \in \mathcal C(\tilde v_1, \tilde v_2)$
such that $C \subset E(v_1, v_2)$ and $x_1 \ll \tilde x_1$ for some $\tilde x_1 \in \overleftarrow{\;\gamma_{\tilde v_1}}$.
By Lemma \ref{lem_C_two_geods} there are $u_1, u_2 \in C$
and $\tilde y \in M$ such that 
$\overline{\gamma_{u_1}} \ne \overline{\gamma_{u_2}}$
and 
$\tilde y \in \overleftarrow{\;\gamma_{u_1}} \cap \overleftarrow{\;\gamma_{u_2}} \cap \overrightarrow{\gamma_{\tilde v_1}}$.
Lemma \ref{lem_E_good} implies that $E(v_1, v_2) = E(y)$, and as $C \subset E(v_1, v_2)$, 
there holds $y \in \overleftarrow{\;\gamma_{u_1}} \cap \overleftarrow{\;\gamma_{u_2}} \cap \overrightarrow{\gamma_{v_1}}$.
As $y, \tilde y \in \overline{\gamma_{u_1}}$
we have $\tilde y \le y$ or $y < \tilde y$.

Case $y < \tilde y$. The causal path from $y$ to $\pi(u_1)$, given by $\gamma_{u_2}$ from $y$ to $\tilde y$ and $\gamma_{u_1}$ from $\tilde y$ to $\pi(u_1)$, 
is not a null pregeodesic. This is a contradiction with $u_1 \in E(y)$. 

Case $\tilde y \le y$. There is a causal path from $x_1$ to $\tilde x_1$, and there is a causal path from $\tilde x$ to $y$
given by $\gamma_{\tilde v_1}$ from $\tilde x_1$ to $\tilde y$
and by $\gamma_{u_1}$ from $\tilde y$ to $y$.
Therefore there is a causal path from $x_1$ to $y$ via $\tilde x_1$ and $x_1 \ll \tilde x_1$, a contradiction with $\gamma_v$ being optimizing from $x_1$ to $y$.
\end{proof}

\begin{lemma}\label{lem_almost_optim}
Let $v_1, v_2 \in L^+ \Omega_\i$ and $x_1, y \in M$ be as in Lemma \ref{lem_E_good}, and let $$\mu_\i : [-1,1] \to \Omega_\i$$ be a timelike and future pointing path.
Suppose that $\mu_\i(s) \in \overline{\gamma_{v_1}}$
for some $s \in [-1,1)$ and $\mu_\i(s) \ll y$.
Suppose, furthermore, that $\mu_\i(1) \not\le y$ and $E(v_1, v_2) \ne \emptyset$.
Then there are $\tilde v_1, \tilde v_2 \in L^+ \Omega_\i$ and non-empty $C \in \mathcal C(\tilde v_1, \tilde v_2)$
such that $C \subset E(v_1, v_2)$ and $\mu_\i(\tilde s) \in \overleftarrow{\;\gamma_{\tilde v_1}}$ for some $\tilde s > s$.
\end{lemma}
\begin{proof}
By Lemma \ref{lem_E_good} we have $E(v_1, v_2) = E(y)$.
As $E(y)$ is non-empty manifold with (non-smooth) boundary,
there is $u \in E(y)$ such that $E(y)$ is a smooth manifold near $u$. Then $\gamma_u$ is optimizing from $y$ to $z:=\pi(u)$ and also slightly past $z$.

By Lemma \ref{lem_optim_exists2} there is $\tilde s \in (s,1)$ such that either there is an optimizing geodesic from $\mu_\i(\tilde s)$ to $y$ or $y = \mu_\i(\tilde s)$. 

Let us consider the former case first. Choose $\tilde v_1$ in the tangent bundle of that geodesic so that $\tilde x:=\pi(\tilde v_1) \in \Omega_\i$
and $\mu_\i(\tilde s) < \tilde x < y$.
By Lemma \ref{lem_smooth_piece}
there are $\tilde v_2 \in L^+ \Omega_\i$ and
$C \subset \CP(\tilde v_{1}, \tilde v_{2})$ 
such that $C$ is a neighborhood of $u$ in $C(y)$.
But $C(y) \subset E(y)$ near $u$.

Let us now suppose that $y = \mu_\i(\tilde s)$.
Choose $\xi_1 \in L^+_y M$ such that $\overline{\gamma_{y,\xi_1}} \ne \overline{\gamma_{u}}$.
By Lemma \ref{lem_linalg_pert} there is a neighborhood $V \subset L_y^+ M$ of $\xi_0$
and $\xi_2 \in L_y^+ M$
such that for any $\eta \in V$
there is $\xi_3 \in L_y^+ M$ such that $\eta \in \linspan(\xi_1,\xi_2,\xi_3)$ and $\eta \notin \linspan(\xi_j)$, $j=1,2,3$.
We write $\tilde v_j = (y, \xi_j)$, $j=1,2$.
Observe that $y$ is the only point in $\overleftarrow{\;\gamma_{\tilde u}} \cap \overrightarrow{\gamma_{\tilde v_1}}$ for $\tilde u$ close to $u$ since $\gamma_u$ is optimizing from $y$ to a point past $\pi(u)$.
As in the proof of Lemma~\ref{lem_smooth_piece}, we see that there is $C \subset \CP(\tilde v_1, \tilde v_2)$ such that $C$ is a neighborhood of $u$ in $C(y)$.
\end{proof}

\subsection{Global recovery}

\begin{figure}
\def\svgwidth{8.5cm}
\begingroup%
  \makeatletter%
  \providecommand\color[2][]{%
    \errmessage{(Inkscape) Color is used for the text in Inkscape, but the package 'color.sty' is not loaded}%
    \renewcommand\color[2][]{}%
  }%
  \providecommand\transparent[1]{%
    \errmessage{(Inkscape) Transparency is used (non-zero) for the text in Inkscape, but the package 'transparent.sty' is not loaded}%
    \renewcommand\transparent[1]{}%
  }%
  \providecommand\rotatebox[2]{#2}%
  \newcommand*\fsize{\dimexpr\f@size pt\relax}%
  \newcommand*\lineheight[1]{\fontsize{\fsize}{#1\fsize}\selectfont}%
  \ifx\svgwidth\undefined%
    \setlength{\unitlength}{575.9999856bp}%
    \ifx\svgscale\undefined%
      \relax%
    \else%
      \setlength{\unitlength}{\unitlength * \real{\svgscale}}%
    \fi%
  \else%
    \setlength{\unitlength}{\svgwidth}%
  \fi%
  \global\let\svgwidth\undefined%
  \global\let\svgscale\undefined%
  \makeatother%
  \begin{picture}(1,1)%
    \lineheight{1}%
    \setlength\tabcolsep{0pt}%
    \put(0,0){\includegraphics[width=\unitlength,page=1]{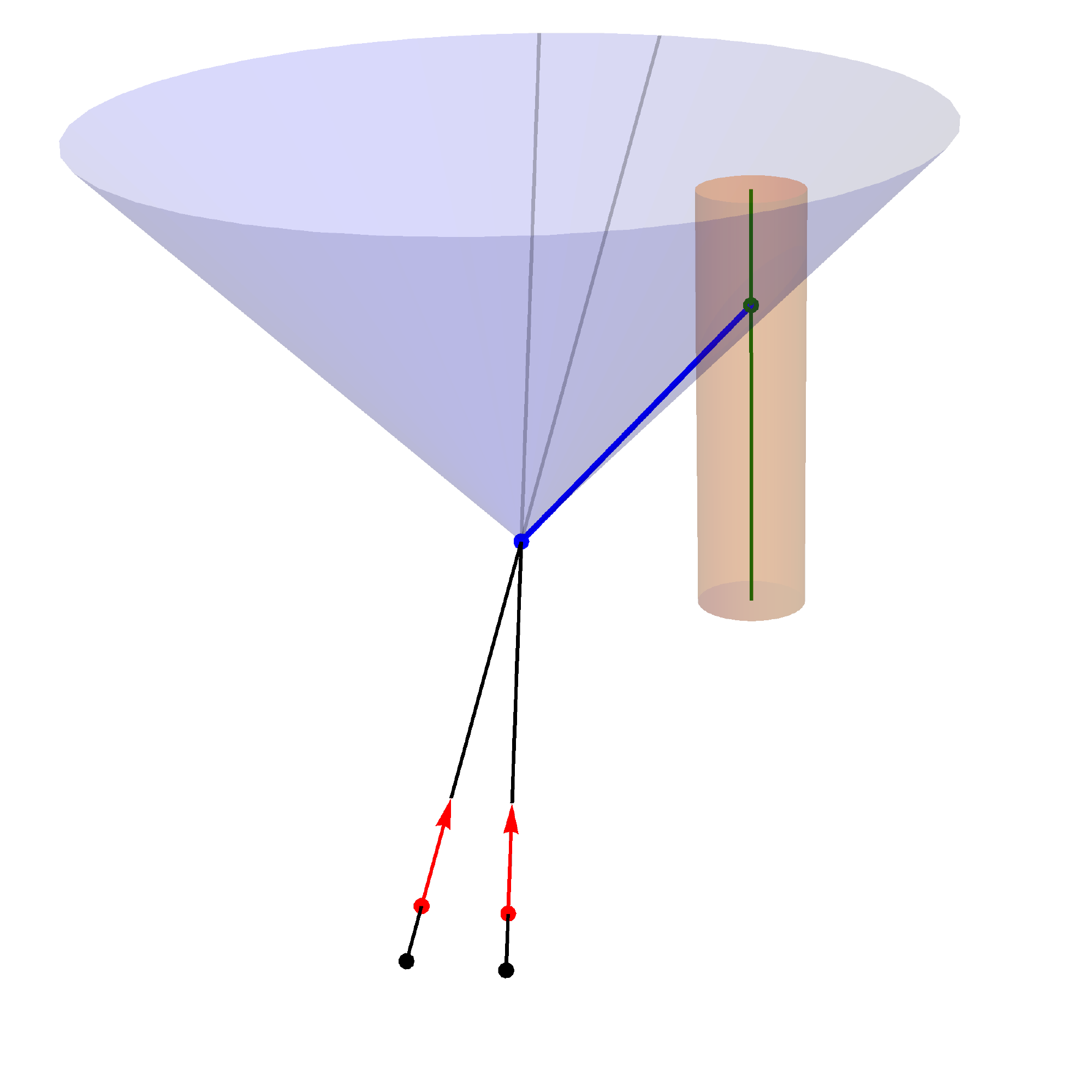}}%
    \put(0.48845318,0.07504559){\color[rgb]{0,0,0}\makebox(0,0)[lt]{\lineheight{1.25}\smash{\begin{tabular}[t]{l}$x_1$\end{tabular}}}}%
    \put(0.49255546,0.17737897){\color[rgb]{0,0,0}\makebox(0,0)[lt]{\lineheight{1.25}\smash{\begin{tabular}[t]{l}$v_1$\end{tabular}}}}%
    \put(0.28600404,0.09564792){\color[rgb]{0,0,0}\makebox(0,0)[lt]{\lineheight{1.25}\smash{\begin{tabular}[t]{l}$x_2$\end{tabular}}}}%
    \put(0.29786187,0.19908741){\color[rgb]{0,0,0}\makebox(0,0)[lt]{\lineheight{1.25}\smash{\begin{tabular}[t]{l}$v_2$\end{tabular}}}}%
    \put(0.50636037,0.46291526){\color[rgb]{0,0,0}\makebox(0,0)[lt]{\lineheight{1.25}\smash{\begin{tabular}[t]{l}$y$\end{tabular}}}}%
  \end{picture}%
\endgroup%

\hspace{-3cm}
\def\svgwidth{8.5cm}
\begingroup%
  \makeatletter%
  \providecommand\color[2][]{%
    \errmessage{(Inkscape) Color is used for the text in Inkscape, but the package 'color.sty' is not loaded}%
    \renewcommand\color[2][]{}%
  }%
  \providecommand\transparent[1]{%
    \errmessage{(Inkscape) Transparency is used (non-zero) for the text in Inkscape, but the package 'transparent.sty' is not loaded}%
    \renewcommand\transparent[1]{}%
  }%
  \providecommand\rotatebox[2]{#2}%
  \newcommand*\fsize{\dimexpr\f@size pt\relax}%
  \newcommand*\lineheight[1]{\fontsize{\fsize}{#1\fsize}\selectfont}%
  \ifx\svgwidth\undefined%
    \setlength{\unitlength}{575.9999856bp}%
    \ifx\svgscale\undefined%
      \relax%
    \else%
      \setlength{\unitlength}{\unitlength * \real{\svgscale}}%
    \fi%
  \else%
    \setlength{\unitlength}{\svgwidth}%
  \fi%
  \global\let\svgwidth\undefined%
  \global\let\svgscale\undefined%
  \makeatother%
  \begin{picture}(1,1)%
    \lineheight{1}%
    \setlength\tabcolsep{0pt}%
    \put(0,0){\includegraphics[width=\unitlength,page=1]{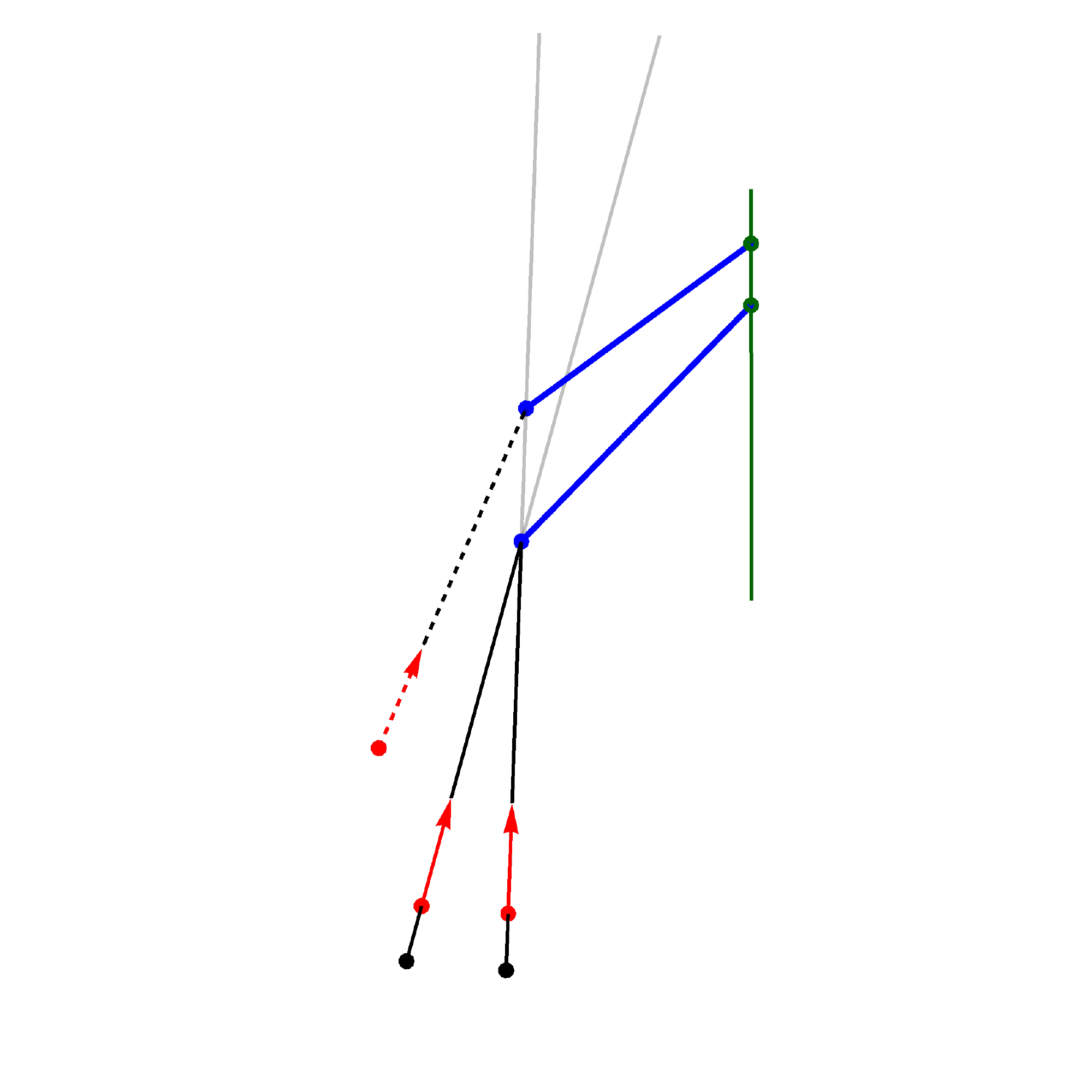}}%
    \put(0.71290816,0.70361924){\color[rgb]{0,0,0}\makebox(0,0)[lt]{\lineheight{1.25}\smash{\begin{tabular}[t]{l}$z$\end{tabular}}}}%
    \put(0.71290816,0.76611924){\color[rgb]{0,0,0}\makebox(0,0)[lt]{\lineheight{1.25}\smash{\begin{tabular}[t]{l}$\tilde z$\end{tabular}}}}%
    \put(0.28888159,0.35452477){\color[rgb]{0,0,0}\makebox(0,0)[lt]{\lineheight{1.25}\smash{\begin{tabular}[t]{l}$w_2$\end{tabular}}}}%
  \end{picture}%
\endgroup%

\caption{
Function $f_0(v_1, v_2)$ in the $1+2$-dimensional Minkowski space. 
{\em Left.} The time-like path $\mu_0 = \mu_\o$ is the green line segment,
and the point $z=\mu_0(f_0(v_1,v_2))$ is also green.
The thick blue line segment is the optimizing geodesic from $y$ to $z$. It is contained in the cone $\pi(C(y))$.
The time axis is vertical, and the sets $\Omega_\o$, $\pi(C(y))$, $\pi(E(y))$, as well as, the points $y, x_1, x_2$, vectors $v_1, v_2$ and geodesics $\gamma_{v_1}, \gamma_{v_2}$ as in Figure~\ref{fig_E}. 
{\em Right.} Inequality $f_0(v_1,v_2) \le f_0(v_1, w_2)$ as in Lemma \ref{lem_f_large}. 
The vector $w_2$ is dashed red and the points $z=\mu_0(f_0(v_1,v_2))$ and $\tilde z=\mu_0(f_0(v_1,w_2))$ are green.
The points $y, x_1, x_2$, vectors $v_1, v_2$, geodesics $\gamma_{v_1}, \gamma_{v_2}$ and path $\mu_0$ as on left.
}\label{fig_fa}
\end{figure}
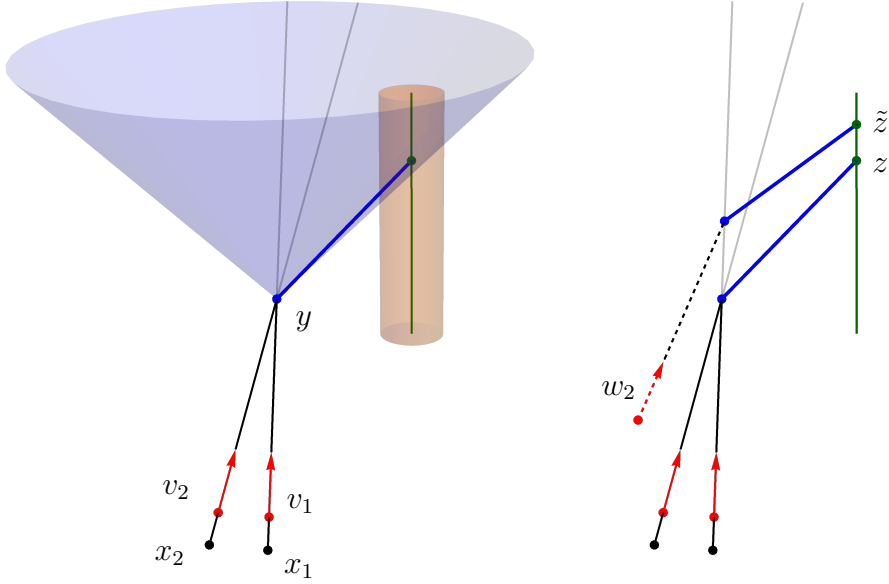

Let $\mu_a$ be as in Lemma \ref{lem_E_future2}, with $\Omega=\Omega_\o$, and define for $v_1, v_2 \in L^+ \Omega_\i$,
    \begin{align*}
f_a(v_1,v_2) = \inf \{s \in [-1,1] : \text{
$\mu_a(s) \in \pi(C(v_1,v_2))$ or $s=1$}\}.
    \end{align*}
Figure \ref{fig_fa} illustrates the function $f_a(v_1, v_2)$.

\begin{lemma}\label{lem_fout}
Let $v_1, v_2 \in L^+ \Omega_\i$ and $y \in M$ be as in Lemma \ref{lem_E_good}.
Then $$f_a(v_1,v_2) = f_{\mu_a}^+(y).$$
\end{lemma}
\begin{proof}
Follows immediately from Lemmas \ref{lem_E_future2} and \ref{lem_E_good}. 
\end{proof}

\begin{lemma}\label{lem_f_large}
Let $v_1 \in L^+ \Omega_\i$ and let $x_1 \in \overline{\gamma_{v_1}}$ satisfy $x_1 < \pi(v_1)$.
Suppose that $\overline{\gamma_{v_1}} \cap \overline{\mu_0} = \emptyset$ and $f_{\mu_0}^+(x_1) > -1$.
Then there is a neighborhood $\mathcal U \subset L^+ \Omega_\i$ of $v_1$ such that all 
$v_2 \in L^+ \Omega_\i$ and $w_2 \in \mathcal U$ satisfy the following. If
there is $x_2 < \pi(v_2)$ such that $\gamma_{v_j}$ is optimizing from $x_j$ to a point 
$y \in \overrightarrow{\gamma_{v_1}} \cap \overrightarrow{\gamma_{v_2}}$ for $j=1,2$,
$\overline{\gamma_{v_1}} \ne \overline{\gamma_{v_2}}$, $y \notin J^-(F_\o(\{-1\}\times B(0,\delta)))$,
and $\overrightarrow{\gamma_{w_2}} \cap \overrightarrow{\gamma_{v_1}} \cap J^-(y) = \emptyset$,
then $f_0(v_1,v_2) \leq f_0(v_1,w_2)$.
\end{lemma}
\begin{proof}
Let $\mathcal U$ be small enough so that it is contained in the two neighborhoods given by Lemmas \ref{lem_C} and \ref{lem_C2}, respectively.
When applying Lemmas \ref{lem_C} and \ref{lem_C2} we take $v = v_1$, $x = x_1$ and $K = J^+(\overline{\Omega_\i}) \cap J^-(\overline{\Omega_\o})$. Moreover, let $\delta' > 0$ be as in Lemma \ref{lem_C2}.

Let $w_2 \in \mathcal U$, $C \in \mathcal C(v_1,w_2)$ and let $u_1 \in C$ satisfy $\pi(u) \in \overline{\mu_a}$ for  $a \in B(0,\delta')$.
We write $\pi(u_1) = F(s, a)$ in the local coordinates (F)
and begin by showing that 
    \begin{align}\label{fa_bound}
f_a(v_1,v_2) \leq s.
    \end{align}
By Lemma~\ref{lem_C_two_geods} there are
$u_2 \in C$, satisfying
$\overline{\gamma_{u_1}} \ne \overline{\gamma_{u_2}}$,
and
    \begin{align*}
y_1 \in \overleftarrow{\;\gamma_{u_1}} \cap \overleftarrow{\;\gamma_{u_2}} \cap \overrightarrow{\gamma_{v_1}}, 
\quad 
y_2 \in \overleftarrow{\;\gamma_{u_1}} \cap \overleftarrow{\;\gamma_{u_2}} \cap \overrightarrow{\gamma_{w_2}}.
    \end{align*}

Case $y_1 = y_2$. Now $\overrightarrow{\gamma_{w_2}} \cap \overrightarrow{\gamma_{v_1}} \cap J^-(y) = \emptyset$ implies $y < y_1$.
Hence by Lemma \ref{lem_fout}
    \begin{align*}
f_a(v_1, v_2) = f_{\mu_a}^+(y) \leq f_{\mu_a}^+(y_1) \leq s.
    \end{align*}

Case $y_2 < y_1$. It follows from Lemma \ref{lem_C} that 
the geodesic $\gamma_{v_1}$ is not optimizing from $\gamma_{v_1}(-\epsilon_1)$ to $y_1$. Therefore $y < y_1$.
As above, this implies $f_a(v_1,v_2) \leq s$.

Case $y_1 < y_2$. Lemma \ref{lem_C2} implies 
    \begin{align*}
f_a(v_1, v_2) = f_{\mu_a}^+(y) \leq f_{\mu_a}^+(y_2) \leq s.
    \end{align*}
We have shown (\ref{fa_bound}).

Let $u \in C(v_1,w_2)$ satisfy $\pi(u) = F(s,0)$ for some $s \in [-1,1]$. 
Then there are $C_j \in \mathcal C(v_1,w_2)$, $u_j \in C_j$, $s_j \in (-1,1)$ and $a_j \in B(0,\delta')$ such that $\pi(u_j) = F(s_j, a_j)$ and $u_j \to u$. 
Now (\ref{fa_bound}) implies $f_{a_j}(v_1, v_2) \leq s_j$,
and letting $j \to \infty$, we obtain
$f_0(v_1, v_2) \leq s$. 
\end{proof}

\begin{lemma}\label{lem_f_small}
Let $v_1 \in L^+ \Omega_\i$ and let $x_1 \in \overline{\gamma_{v_1}}$ satisfy $x_1 < \pi(v_1)$.
Suppose that $\gamma_{v_1}$ is optimizing from $x_1$ to a point $y$.
Then there is a neighborhood $\mathcal U \subset L^+ \Omega_\i$ of $v_1$ such that all 
$v_2 \in \mathcal U$ satisfy the following. If there is
$\tilde y \in \overrightarrow{\gamma_{v_2}} \cap \overrightarrow{\gamma_{v_1}} \cap J^-(y)$,
then there is $x_2 < \pi(v_2)$ such that $\gamma_{v_2}$ is optimizing from $x_2$ to $\tilde y$.
\end{lemma}
\begin{proof}
To get a contradiction suppose that there are sequences $L^+M\ni v_j \to v_1$ and
    \begin{align*}
\tilde y_j \in \overrightarrow{\gamma_{v_j}} \cap \overrightarrow{\gamma_{v_1}} \cap J^-(y),
    \end{align*}
such that for all $x \in \overline{\gamma_{v_j}}$ 
there holds: if $x < \pi(v_j)$ then $x \ll \tilde y_j$.
Due to compactness of $J^+(\pi(v_1)) \cap J^-(y)$
we may assume that $\tilde y_j \to \tilde y$
for some $\tilde y \in M$. Then $\tilde y \in \overrightarrow{\gamma_{v_1}}$ and $\tilde y \le y$.
We choose $\eta_j \in L_{\tilde y_j}^- M$ and $r_j \geq 0$
so that $\overline{\gamma_{\tilde y_j, \eta_j}} = \overline{\gamma_{v_j}}$, $\gamma_{\tilde y_j, \eta_j}(r_j) = \pi(v_j)$ and so that  $\eta_j \to \eta$ and $r_j \to r$ for some $\eta \in L_y^+ M$ and $r \geq 0$.
Then $\rho(\tilde y_j, \eta_j) \leq r_j$ and $\gamma_{\tilde y,\eta}(r) = \pi(v_1)$. Lemma \ref{lem_rho_cont_aux} implies that $\rho(\tilde y,\eta) \leq r$. But this is a contradiction with $\gamma_{v_1}$ being optimizing from $x_1 < \pi(v_1)$ to $y$.
\end{proof}

We are now ready to prove the main theorem in this section, that shows that the earliest arrivals can be reconstructed from the relation $\rel$. Figure \ref{fig_shortcut} outlines the geometric setting of the theorem.

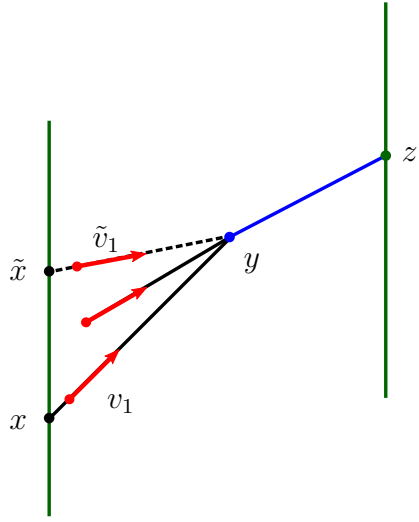
\begin{figure}
\def\svgwidth{6cm}
\begingroup%
  \makeatletter%
  \providecommand\color[2][]{%
    \errmessage{(Inkscape) Color is used for the text in Inkscape, but the package 'color.sty' is not loaded}%
    \renewcommand\color[2][]{}%
  }%
  \providecommand\transparent[1]{%
    \errmessage{(Inkscape) Transparency is used (non-zero) for the text in Inkscape, but the package 'transparent.sty' is not loaded}%
    \renewcommand\transparent[1]{}%
  }%
  \providecommand\rotatebox[2]{#2}%
  \newcommand*\fsize{\dimexpr\f@size pt\relax}%
  \newcommand*\lineheight[1]{\fontsize{\fsize}{#1\fsize}\selectfont}%
  \ifx\svgwidth\undefined%
    \setlength{\unitlength}{376.021975bp}%
    \ifx\svgscale\undefined%
      \relax%
    \else%
      \setlength{\unitlength}{\unitlength * \real{\svgscale}}%
    \fi%
  \else%
    \setlength{\unitlength}{\svgwidth}%
  \fi%
  \global\let\svgwidth\undefined%
  \global\let\svgscale\undefined%
  \makeatother%
  \begin{picture}(1,1.1395771)%
    \lineheight{1}%
    \setlength\tabcolsep{0pt}%
    \put(0,0){\includegraphics[width=\unitlength,page=1]{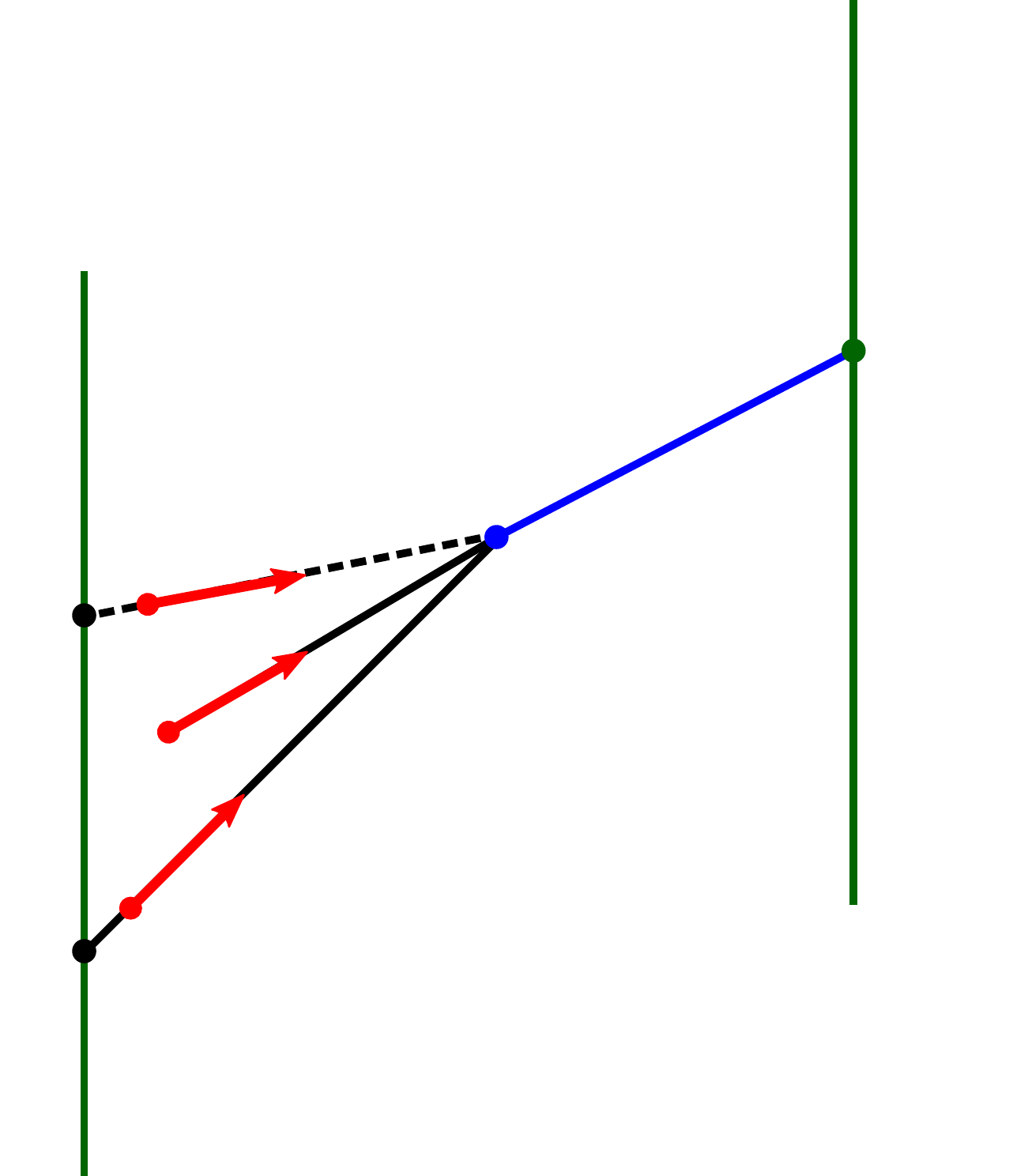}}%
    \put(0.51263563,0.55235115){\color[rgb]{0,0,0}\makebox(0,0)[lt]{\lineheight{1.25}\smash{\begin{tabular}[t]{l}$y$\end{tabular}}}}%
    \put(0.86314641,0.78891374){\color[rgb]{0,0,0}\makebox(0,0)[lt]{\lineheight{1.25}\smash{\begin{tabular}[t]{l}$z$\end{tabular}}}}%
    \put(0.20983109,0.24003836){\color[rgb]{0,0,0}\makebox(0,0)[lt]{\lineheight{1.25}\smash{\begin{tabular}[t]{l}$v_1$\end{tabular}}}}%
    \put(0.17773934,0.59746017){\color[rgb]{0,0,0}\makebox(0,0)[lt]{\lineheight{1.25}\smash{\begin{tabular}[t]{l}$\tilde v_1$\end{tabular}}}}%
    \put(-0.00662258,0.19335897){\color[rgb]{0,0,0}\makebox(0,0)[lt]{\lineheight{1.25}\smash{\begin{tabular}[t]{l}$x$\end{tabular}}}}%
    \put(-0.00662258,0.52445667){\color[rgb]{0,0,0}\makebox(0,0)[lt]{\lineheight{1.25}\smash{\begin{tabular}[t]{l}$\tilde x$\end{tabular}}}}%
  \end{picture}%
\endgroup%

\caption{
Schematic of the geometric setting of Theorem \ref{thm_obs}.
The time-like paths $\mu_\i$ and $\mu_0$ in green. 
Vectors $v_1$, $\tilde v_1$ and $v_2$ in red, last of which is not labelled. Points $x = \mu_\i(s)$
and $\tilde x = \mu_\i(\tilde s)$ in black, and point $z = \mu_0(f_0(v_1,v_2))$ in green. Here $f_0(v_1, v_2) \ge f_\crit(v_1)$. Observe that this case can not arise in the Minkowski space. 
}\label{fig_shortcut}
\end{figure}

\begin{theorem}
\label{thm_obs}
Let $s \in [-1,1)$ 
and suppose that $v_1 \in L^+ \Omega_\i$ satisfies $\mu_\i(s) \in \overline{\gamma_{v_1}}$, $\mu_\i(s) < \pi(v_1)$ and $\overline{\gamma_{v_1}} \cap \overline{\mu_0} = \emptyset$.
Then there is a neighborhood $\mathcal U \subset L^+ \Omega_\i$ of $v_1$ such that for all neighborhoods $\mathcal U' \subset \mathcal U$ of $v_1$ there holds
    \begin{align}\label{E_recovery}
&\{E(v_1, v_2) : \text{$v_2 \in \mathcal U'$, $f_0(v_1,v_2) \leq f_\crit$ and $f_0(v_1, v_2) < 1$}\}
\\\notag&\qquad=
\{E(y) : \text{$y \in \overrightarrow{\gamma_{v_1}}$, $\gamma_{v_1}$ is optimizing from $\mu_\i(s)$ to $y$ and $f_{\mu_0}^+(y) < 1$}\}
\\\notag&\qquad{\mltext =
\{E(y) : \ y=\gamma_{v_1}(r), 0\leq r\leq \rho(v_1)\},}   
 \end{align}
where $f_\crit = \inf \{ f_0(v_1, v_2) : v_2 \in \mathcal W \}$
and
    \begin{align*}
\mathcal W=\{ v_2 \in \mathcal U' :
\ &\text{$\overline{\gamma_{v_1}} \ne \overline{\gamma_{v_2}}$, and there are $\tilde v_1, \tilde v_2 \in L^+ \Omega_\i$}
\\&
\text{and non-empty $C \in \mathcal C(\tilde v_1, \tilde v_2)$
such that}
\\&
\text{$C \subset E(v_1, v_2)$ and $\mu_\i(\tilde s) \in \overleftarrow{\;\gamma_{\tilde v_1}}$ for some $\tilde s > s$}\}.
    \end{align*}
 {\mltext Moreover, we have
     \begin{align}\label{E_recovery B}
&\{E(v_1, v_2) : \text{$v_2 \in \mathcal U'$, $f_0(v_1,v_2) < f_\crit$ and $f_0(v_1, v_2) < 1$}\}
\\\notag&=
\{E(y) : \ y=\gamma_{v_1}(r), 0\leq r<\rho(v_1)\}.   
    \end{align}}
 
\end{theorem}
\begin{proof}
Observe that $f^+_{\mu_0}(x_1) > -1$ since $\Omega_\i \cap J^-(F_\o(\{-1\}\times B(0,\delta))) = \emptyset$.
Let $\mathcal U$ be small enough so that it is contained in the two neighborhoods given be Lemmas \ref{lem_f_large} and \ref{lem_f_small}.

Denote the left-hand side of (\ref{E_recovery}) by $\mathcal E$.
Let $y \in \overrightarrow{\gamma_{v_1}}$
and suppose that $\gamma_{v_1}$ is optimizing from $x_1$ to $y$. We will show that $E(y) \in \mathcal E$.
Lower semi-continuity of $\rho$ implies that there is $v_2 \in \mathcal U'$ and $x_2 < \pi(v_2)$ such that $\gamma_{v_2}$ is optimizing from $x_2$ to $y$ and $\overline{\gamma_{v_1}} \ne \overline{\gamma_{v_2}}$.
Moreover, $\Omega_\i \cap J^-(F_\o(\{-1\}\times B(0,\delta))) = \emptyset$
implies that $y \notin J^-(F_\o(\{-1\}\times B(0,\delta)))$.
Now Lemmas \ref{lem_E_good} and \ref{lem_fout} give $E(v_1, v_2) = E(y)$ and $f_0(v_1, v_2) = f_{\mu_0}^+(y)$.
Hence $E(y) \in \mathcal E$ follows after we show that $f_0(v_1,v_2) \leq f_\crit$.

Let $w_2 \in \mathcal U'$ satisfy $\overline{\gamma_{v_1}} \ne \overline{\gamma_{w_2}}$.
If $\overrightarrow{\gamma_{w_2}} \cap \overrightarrow{\gamma_{v_1}} \cap J^-(y) = \emptyset$,
then $f_0(v_1,v_2) \leq f_0(v_1,w_2)$ by
Lemma~\ref{lem_f_large}.
On the other hand, if $\overrightarrow{\gamma_{w_2}} \cap \overrightarrow{\gamma_{v_1}} \cap J^-(y) \ne \emptyset$
then Lemmas \ref{lem_f_small} and \ref{lem_optim} imply that $w_2 \notin \mathcal W$. Hence $f_0(v_1,v_2) \leq f_\crit$.

Suppose now that $v_2 \in \mathcal U'$ and $f_0(v_1,v_2) \leq f_\crit$ and $f_0(v_1, v_2) < 1$. 
To get a contradiction suppose that 
there does not exist
$y \in \overrightarrow{\gamma_{v_1}} \cap \overrightarrow{\gamma_{v_2}}$
such that $\gamma_{v_1}$ is optimizing from $x_1$ to $y$.
Let $\tilde x_1, \tilde y \in \overline{\gamma_{v_1}}$, and suppose that $\gamma_{v_1}$ is optimizing from $\tilde x_1$ to $\tilde y$ and 
    \begin{align*}
x_1 < \tilde x_1 < \pi(v_1) < \tilde y.
    \end{align*}
Clearly such $\tilde x_1$ and $\tilde y$ exist.
Then there are $w_2 \in \mathcal U'$ and $\tilde x_2 < \pi(w_2)$ such that $\gamma_{w_2}$ is optimizing from $\tilde x_2$ to $\tilde y$ and $\overline{\gamma_{v_1}} \ne \overline{\gamma_{w_2}}$.
Lemmas \ref{lem_fout} and \ref{lem_f_large}
imply that 
    \begin{align}\label{E_rec_aux}
f_{\mu_0}^+(\tilde y) = f_0(v_1,w_2) \leq f_0(v_1,v_2).
    \end{align}
We will now consider two cases.

Case that there is $y \in \overline{\gamma_{v_1}}$ such that $x_1 \ll y$ and $y \in J^-(\overline{\Omega_\o})$. Then the points $\tilde x_1, \tilde y \in \overline{\gamma_{v_1}}$ can be chosen so that
$x_1 \ll \tilde y$.
Observe that $f_0(v_1, v_2) < 1$ implies $E(v_1, v_2) \ne \emptyset$. Moreover, $\overline{\Omega_\o} \cap J^+(\mu_\i(1)) = \emptyset$ implies that $\mu_\i(1) \not\le y$.
Lemma \ref{lem_almost_optim} implies that $w_2 \in \mathcal W$ and hence, recalling (\ref{E_rec_aux}),
    \begin{align*}
f_0(v_1,v_2) \leq f_\crit \leq f_{\mu_0}^+(\tilde y) = f_0(v_1,w_2) \leq f_0(v_1,v_2).
    \end{align*}
But also $f_0(v_1, v_2) = f_{\mu_0}^+(y')$ whenever 
$y' \in \overline{\gamma_{v_1}}$ is close to $\tilde y$ and $\pi(v_1) < y' < \tilde y$.
As $f^+_{\mu_0}(x_1) > -1$ and $\overline{\gamma_{v_1}} \cap \overline{\mu_0} = \emptyset$, Lemma \ref{lem_f_inc}
implies that $f_0(v_1, v_2) = 1$, a contradiction with $f_0(v_1, v_2) < 1$.

Case that there does not exist $y \in \overline{\gamma_{v_1}}$ such that $x_1 \ll y$ and $y \in J^-(\overline{\Omega_\o})$. 
By Lemma \ref{lem_exit} the point $\tilde y \in \overline{\gamma_{v_1}}$ can be chosen so that $\tilde y$ is not in the interior of the set $K$ in (\ref{def_K}). Then $f_{\mu_0}^+(\tilde y) = 1$, and (\ref{E_rec_aux}) gives a contradiction with $f_0(v_1, v_2) < 1$.

There is $y \in \overrightarrow{\gamma_{v_1}} \cap \overrightarrow{\gamma_{v_2}}$
such that $\gamma_{v_1}$ is optimizing from $x_1$ to $y$. Lemma~\ref{lem_f_small} implies that there is $x_2 < \pi(v_2)$ such that $\gamma_{v_2}$ is optimizing from $x_2$ to $y$. Now Lemmas \ref{lem_E_good} and \ref{lem_fout} give $E(v_1, v_2) = E(y)$ and $f_0(v_1, v_2) = f_{\mu_0}^+(y)$,
and $E(v_1, v_2)$ is in the set on the right-hand side of (\ref{E_recovery}).
{\mltext The above considerations also give the equation \eqref{E_recovery B} 
when the inequality $f_0(v_1,v_2) \leq f_\crit$ is replaced by a strict inequality.}
    \end{proof}

We are now ready to complete the proof of Theorem~\ref{t2}.
Recall that $\pi:TM\to M$  is the maps to the base point of the vector. Note, by denoting $U=\Omega_{out}$, we have $\mathcal E_U(y)=\pi(E(q))$ where the notation $\mathcal E_U(y)$ denotes the earliest light observation set of point $y$  {\mltext in the observation set $U$, see \cite{KLU}.  Roughly speaking, the set $\mathcal E_U(y)$ corresponds to the first observations made in the set $U$  when there is a point source at $y$ that sends light to all directions. 
By Theorem~\ref{thm_obs}, we see that the relation $R$  determines for all $x=\mu_\i(s)$, with $s\in (t_0^-,t_0^+)$,
and  
$v_1 \in L^+_xM$  the set
$\{E(y): \ y=\gamma_{v_1}(r), 0\leq r< \rho(v_1)\}$.
%

Therefore, the relation $\rel$ uniquely determines the set
$$E({\mathbb D})=\bigcup_{s\in (t_0^-,t_0^+)} \{E(y)\subset T\Omega_{\o}:\ y=\gamma_{v_1}(r), v_1\in L^+_{\mu_{\i}(s)}M, r\in [0,\rho(v_1))\}.$$
Thus, $R$ determines {\mltext $\mathcal E_U({\mathbb D})=\{\mathcal E_U(y):\ y\in  {\mathbb D}\}=\{\pi(E(y)):\ y\in{\mathbb D}\}$.}
Thus the problem of recovering the manifold is reduced to the inverse problem with passive measurements studied {\mltext in \cite{KLU}}. By 
 \cite[Theorem 1.2]{KLU}, the set $\mathcal E_U({\mathbb D})$ 
determines  the topological, differential and conformal structures
of $\mathbb D$. This proves Theorem~\ref{t2}.\hfill $\square$}

\section{Proof of Theorem~\ref{t1}}
\label{pf_thm_sec}

This section is concerned with the proof of Theorem~\ref{t1}. The first claim in the theorem, namely, determining the conformal, topological and differential structure of the manifold $(M,g)$ from either of the source-to-solution maps $\mathscr L$ or $\mathscr N$ follows from Theorems~\ref{thm_anal_data_1}--\ref{thm_anal_data_2} and Theorem~\ref{t2}. To see this, we begin by defining the relations

   \begin{align*}
\rel_{\text{semi-lin}} = \{&(v_{0},v_{1}, v_{2},v_{3})\in  L^+ \Omega_\o \times (L^+ \Omega_\i)^3: 
\text{$\gamma_{v_j}$'s are pair-wise not identical,}
\\& \text{there are $f \in C^{\infty}_c(\Omega_\i)$, $\kappa_j \in \R \setminus \{0\}$ \quad \text{and}\quad $\iota_j \in \mathcal T$, $j=0,1,2,3$, }
\\&\text{s.t for all small $\delta'>0$,
$\mathscr D^{\textrm{semi}}_{\sigma,\delta',f} \ne 0$ where $\sigma = (v_{0},\kappa_0,\iota_0,\dots, v_{3},\kappa_3,\iota_3)$}
\}
    \end{align*}
 and
   \begin{align*}
\rel_{\text{quasi-lin}} = \{&(v_{0},v_{1}, v_{2},v_{3})\in  L^+ \Omega_\o \times (L^+ \Omega_\i)^3: 
\text{$\gamma_{v_j}$'s are pair-wise not identical,}
\\& \text{there are $\kappa_j \in \R \setminus \{0\}$ \quad \text{and}\quad $\iota_j \in \mathcal T$, $j=0,1,2,3$, }
\\&\text{s.t for all small $\delta'>0$,
$\mathscr D^{\textrm{quasi}}_{\sigma,\delta'} \ne 0$ where $\sigma = (v_{0},\kappa_0,\iota_0,\dots, v_{3},\kappa_3,\iota_3)$}
\}
    \end{align*}

It follows as a consequence of Theorem~\ref{thm_anal_data_1}--\ref{thm_anal_data_2} that the source-to-solution map $\mathscr L$ (respectively $\mathscr N$) determines $\rel_{\textrm{semi-lin}}$ (respectively $\rel_{\textrm{quasi-lin}}$) and that the latter relations are both examples of three-to-one scattering relations, that is to say, they both satisfy conditions (R1) and (R2). We can therefore apply Theorem~\ref{t2} to uniquely determine the topological, differential and conformal structure of the manifold $(M,g)$ on $\mathbb D$ from either of the source-to-solution maps $\mathscr L$ or $\mathscr N$.

In the remainder of this section, we complete the proof of Theorem~\ref{t1} by showing that in the case of the semi-linear equation \eqref{pf0} and if $(n,m)\neq 3$, the conformal factor can also be determined uniquely. We will follow the ideas set out in \cite{UW}. Here, there will be some modifications as we are using Gaussian beams. The exceptional case $(n,m)=(3,3)$ will require an alternative approach that will be briefly discussed at the end of this section.  

To set the idea in motion, we write $g=c\hat{g}$ for the metric on $\mathbb D$, where $\hat{g}$ is known and $c>0$ is a smooth unknown function. Naturally, we will think of the metrics $g$ and $\hat{g}$ as metrics on the manifold $M$ that are conformal to each other only on the set $\mathbb D$. Let us consider the Gaussian beams $\mathcal U_\lambda$ described in Section~\ref{formalgaussian}. Our aim here is to show that the values of the phase function $\phi$ restricted to the set $\mathbb D$ is independent of the conformal factor, while the principal part of the amplitude function, $a_{0,0}$, restricted to the set $\mathbb D$ is given by
\bel{conf_p} a_{0,0}=c^{-\frac{n-1}{4}}\hat{a}_{0,0}, \ee
where $\hat{a}_{0,0}$ is independent of the conformal factor. Showing that $\phi$ is independent of the conformal factor is trivial since the equation \eqref{wkb} for the phase function is conformally invariant.

To show \eqref{conf_p}, we start by recalling that the wave operator changes under conformal scalings of the metric according to the expression:
\bel{wave_conf} 
\Box_{c\hat g} u = c^{-\frac{n+3}{4}} (\Box_{\hat{g}} + q_c) (c^{\frac{n-1}{4}}u)\quad \text{on $\mathbb D$},
\ee  
where $q_c= -c^{\frac{1-n}{4}}\Box_{\hat{g}}(c^{\frac{n-1}{4}})$.

We now return to the construction of Gaussian beams associated to the operator $\Box_g$ on $M$ and note that due to the scaling property above on the set $\mathbb D$, there is a one-to-one correspondence between Gaussian beams for $\Box_g$ and $\Box_{\hat{g}}+q_c$. Here, by a Gaussian beam for $\Box_{\hat g} +q_c$, we mean an ansatz 
$$ \hat{\mathcal U}_\lambda = e^{i\lambda \hat{\phi}} \hat{A}_{\lambda} \quad \text{for $\lambda>0$}$$
and 
$$ \hat{\mathcal U}_\lambda = \overline{e^{i\lambda \hat{\phi}}\hat{A}_{\lambda}} \quad \text{for $\lambda<0$}$$
where we are using Fermi coordinates $(\hat{s},\hat{y}')$ near $\gamma$ with respect to the metric $\hat{g}$ and define the phase, $\hat{\phi}$, and amplitude, $\hat{A}_{\lambda}$, analogously to \eqref{phase-amplitude}. Here, because of the presence of the zeroth order term $q_c$, the governing equations for construction of the phase and amplitude terms read as follows:
\bel{wkb_conf} 
\begin{aligned}
&\frac{\p^{|\alpha|}}{\p \hat{y}'^{\alpha}}\langle d\hat{\phi},d\hat{\phi}\rangle_{\hat g}=0 \quad \text{on $(\hat{a},\hat{b})\times\{\hat{y}'=0\}$,}\\ 
&\frac{\p^{|\alpha|}}{\p \hat{y}'^{\alpha}}\left( 2\langle d\hat{\phi},d\hat{a}_j \rangle_{\hat g}+ (\Box_{\hat g}\hat{\phi})\hat{a}_j + i (\Box_{\hat g}+q_c)\hat{a}_{j-1}\right)=0\quad \text{on $(\hat{a},\hat{b})\times\{\hat{y}'=0\}$},
\end{aligned}\ee
for all $j=0,1,\ldots,N$ and all multi-indices $\alpha=(\alpha_1,\ldots,\alpha_n) \in \{0,1,\ldots\}^n$ with $|\alpha|=\alpha_1+\ldots+\alpha_n\leq N$. Thus,  by setting $j=0$ in \eqref{wkb_conf}, it follows that at each point $y \in \mathbb D$, $\hat{a}_{0,0}(y)$ is independent of the conformal factor $c$. To summarize, the principal part of the amplitude $a_{0,0}$ for the Gaussian beams $\mathcal U_\lambda$ on the set $\mathbb D$ must be given by \eqref{conf_p}
for some $\hat{a}_{0,0}$ that arises from solving \eqref{wkb_conf} and is only dependent on the conformal class of the metric on $\mathbb D$. 

We now return to the task of showing that the conformal factor $c$ can be uniquely determined at every point $y \in \mathbb D$. Applying arguments similar to the proof of \cite[Lemma 4]{FO}, we can show that there exists a null geodesic $\gamma_{v_0}$ for some $v_0 \in L^{-}\Omega_\o$ passing through $y$ and a null geodesic $\gamma_{v_1}$ with $v_1 \in L^+\Omega_\i$ passing through $y$, such that $\gamma_{v_0}$ and $\gamma_{v_1}$ have a single intersection point on the set $\mathbb D$. Note that this property can be checked via the knowledge of the topological, differential and conformal structure of the manifold since null vectors are conformally invariant. 

We now consider two null geodesics $\gamma_{v_2}$ and $\gamma_{v_3}$ in a small neighborhood of $\gamma_{v_1}$ passing through $y$ and such that 
$$ \{ \dot{\gamma}_{v_0}(s_0), \dot{\gamma}_{v_1}(s_1),\dot{\gamma}_{v_2}(s_2),\dot{\gamma}_{v_3}(s_3)\}$$
forms a linearly dependent set. Here, $\gamma_{v_j}(s_j)=y$ for $j=0,1,2,3$. We emphasize that the existence of such null geodesics is guaranteed by Lemma~\ref{Lauri's Lemma}. Now given any choice $v_0,\ldots,v_3$ as above, we pick $\iota_j \in \mathcal T_{v_j}$ such that the amplitude term
$a_{0,0}^{(j)}$ is real-valued and non-zero at the point $y$. We note from \eqref{conf_p}, that this condition can also be checked via just the conformal structure of $\mathbb D$. Finally, we set $\kappa_0=1$ and let $\kappa_1,\kappa_2,\kappa_3\in \R\setminus \{0\}$ be arbitrary. We then consider $\sigma \in \Sigma_{v_0,v_1}$
given by $v_j$, $\iota_j$ and $\kappa_j$ with $j=0,1,2,3$ constructed as above. As discussed the choice $\sigma$ can be determined by just knowning the topological, differential and conformal structure of $\mathbb D$.

Note that although $\kappa_0=1$ is fixed and $v_j$, $\iota_j$ are also fixed for $j=0,1,2,3$, we are free to vary $\kappa_1$, $\kappa_2$ and $\kappa_3$ and also an arbitrary real valued source term $f \in C^{\infty}_c(\Omega_\i)$. For each choice of $f$ and each non-zero $\kappa_1,\kappa_2,\kappa_3$, we proceed to compute $\mathscr D^{\textrm{semi}}_{\sigma,\delta',f}$. Following the steps of the proof in Theorem~\ref{thm_anal_data_1} and in view of the linear dependence of $\{\dot{\gamma}_{v_j}(s_j)\}_{j=0}^3$ and the fact that $a^{(j)}_{0,0}(y)$ are all real-valued, we conclude that if $\mathscr D^{\textrm{semi}}_{\sigma,\delta',f} \neq 0$ for some choice of $\kappa_1$, $\kappa_2$ and $\kappa_3$ and some function $f \in C^{\infty}_c(\Omega_\i)$, then there holds:
\bel{lin_dep_conf} \sum_{j=0}^3 \kappa_j \dot{\gamma}_{v_j}(s_j)=0.\ee 

We fix $\sigma$ corresponding to such a choice of $\kappa_1$, $\kappa_2$ and $\kappa_3$ and proceed to explicitly find the value of $\mathscr D^{\textrm{semi}}_{\sigma,\delta'}$ showing that it determines $c$. Indeed, by retracing the proof of Theorem~\ref{thm_anal_data_1}, using the fact that $dV_{g}=c^{\frac{n+1}{2}}dV_{\hat g}$, together with fact that the values of the phase functions $\phi^{(j)}$, $j=0,1,2,3$ are independent of the conformal factor and that $a_{0,0}^{(j)}$ is real valued at $y$ and \eqref{conf_p} holds, we obtain: 
\bel{almost_final_data} \mathscr D^{\textrm{semi}}_{\sigma,\delta',f}= C \,c(y)^{\frac{n+1}{2}} \,(c(y)^{-\frac{n-1}{4}})^4u_f(y)^{m-3}= C c(y)^{-\frac{n-3}{2}}u_f^{m-3}(y),\ee
where $C$ is a constant that only depends on the conformal class in a neighborhood of $y$. 

The preceding analysis shows that given each $f \in C^{\infty}_c(\Omega_\i)$, we can recover the value of $c(y)^{-\frac{n-3}{2}}u_f^{m-3}(y)$ at each point $y \in \mathbb D$. This can be simplified further by using sources $f$ that generate real parts of Gaussian beams and such that they have (asymptotically) prescribed values at each point $y \in \mathbb D$ as in Lemma~\ref{non-vanishing_lem_wave}. Indeed, owing to equations \eqref{wave_conf}--\eqref{wkb_conf}, we can repeat the argument in the proof of Lemma~\ref{non-vanishing_lem_wave} to construct explicit sources $f_\lambda \in C^{\infty}_c(\Omega_\i)$ only depending on the conformal class $\hat{g}$ on $\mathbb D$ such that
$$ u_{f_\lambda}(y)= c(y)^{-\frac{n-1}{4}}+\mathcal O(\lambda^{-1}),$$
where $\lambda>0$ is a large parameter. Combining this with \eqref{almost_final_data} and taking a limit as $\lambda\to \infty$, we conclude that the knowledge of the source-to-solution map $\mathscr L$ determines uniquely the values  
$$ c(y)^{-\frac{n-3}{2}} \,c(y)^{-\frac{(n-1)(m-3)}{4}},$$
at each point $y \in \mathbb D$. Thus, it follows that $c$ can be determined uniquely on the set $\mathbb D$, unless $(n,m)=(3,3)$.

We remark that in the case $(n,m)=(3,3)$ this simple approach does not yield any information. To treat this case, one needs to look further in the asymptotic expansion of $\mathcal I_{\lambda,\sigma,\delta',f}$ (see Section~\ref{subsec_semi}) with respect to the parameter $\lambda$, than just the principal behavior that is captured by $\mathscr D^{\textrm{semi}}_{\sigma,\delta',f}$. This will also require explictly evaluating the sub-principal term $a_{1,0}$ in the expression for the Gaussian beams (see \eqref{phase-amplitude}). As one of the main novelties of this paper is the generalization to arbitrary dimensions and also for the sake of brevity we omit this analysis in this paper. Note also that the paper \cite{UW} already deals with the particular case $n=3$ although there the authors use a four wave interaction.

Before closing the section, we also remark that in the case of the quasi-linear source-to-solution map $\mathscr N$, this approach entangles information about the tensors $h$ and the conformal factor $c$ at the point $y$ and additional efforts may be needed to uniquely reconstruct the conformal factor.

\end{document}